\documentclass[12pt, reqno]{amsart}
\usepackage{amsmath, amsthm, amsfonts, amssymb, verbatim, graphicx, mathtools}
\usepackage{a4wide}
\usepackage{todonotes}
\usepackage{float}
\usepackage{enumerate}
\usepackage[makeroom]{cancel} 

\usepackage{tikz}
\usetikzlibrary{patterns, arrows.meta}
\usetikzlibrary{matrix}
\usetikzlibrary{backgrounds}

\usepackage[normalem]{ulem}

\usepackage{xcolor}

\usepackage{hyperref}

\usepackage{calc}

\usepackage{accents}

\usepackage{multicol}

\usepackage{caption}
\usepackage{subcaption} 

\makeatletter
\@namedef{subjclassname@2020}{%
  \textup{2020} Mathematics Subject Classification}
\makeatother

\allowdisplaybreaks

\newcommand{\R}{\mathbb{R}}
\newcommand{\Q}{\mathbb{Q}}
\newcommand{\N}{\mathbb{N}}
\newcommand{\C}{\mathbb{C}}
\newcommand{\Z}{\mathbb{Z}}
\newcommand{\uhp}{\mathbb{H}}
\newcommand{\im}{\textnormal{Im}}
\newcommand{\re}{\textnormal{Re}}
\newcommand{\cB}{\mathcal{B}}
\newcommand{\cW}{\mathcal{W}}
\newcommand{\cM}{\mathcal{M}}
\newcommand{\cO}{\mathcal{O}}

\newcommand{\Oc}[1]{\mathcal{O}^{#1}(\overline{\C}\setminus [0,1])}
\newcommand{\Occ}[2]{\mathcal{O}^{#1}(\overline{\C}\setminus #2)}
\newcommand{\Li}[1]{\textnormal{Li}_2\left(#1\right)}
\newcommand{\Gauss}{G}

\newcommand{\B}[2]{B_{#1,#2}}

\newcommand{\T}[2]{T_{#1,#2}}
\newcommand{\Ss}[1]{S_{#1}}
\newcommand{\be}[2]{\beta^{(#1)}_{#2}}
\newcommand{\x}[2]{x^{(#1)}_{#2}}

\newcommand{\wL}{\overline{L}}
\newcommand{\wsum}{\overline\sum}
\newcommand{\wsuml}[1]{\overline{\sum_{#1}}}
\newcommand{\wvarphi}{\overline{\varphi}}

\newcommand{\dbtilde}[1]{\accentset{\approx}{#1}}
\newcommand{\veps}{\varepsilon}

\DeclareMathOperator{\SL}{SL}
\DeclareMathOperator{\GL}{GL}

\newtheorem{thm}{Theorem}[section]
\newtheorem{cor}[thm]{Corollary}
\newtheorem{lem}[thm]{Lemma}
\newtheorem{prop}[thm]{Proposition}

{\theoremstyle{definition}

\newtheorem{rmk}[thm]{Remark}
}

\numberwithin{equation}{section}
\date{\today}

\begin{document}

\title{Regularity properties of $k$-Brjuno and Wilton functions}

\author{Seul Bee Lee}
\address{Centro di Ricerca Matematica Ennio De Giorgi, Scuola Normale Superiore, Piazza dei Cavalieri 3, 56126 Pisa, Italy}
\email{seulbee.lee@sns.it}

\author{Stefano Marmi}
\address{Scuola Normale Superiore, Palazzo della Carovana, Piazza dei Cavalieri 7, 56126 Pisa, Italy}
\email{stefano.marmi@sns.it}

\author{Izabela Petrykiewicz}
\email{ipetrykiewicz@gmail.com}

\author{Tanja I. Schindler}
\address{Centro di Ricerca Matematica Ennio De Giorgi, Scuola Normale Superiore, Piazza dei Cavalieri 3, 56126 Pisa, Italy}
\email{tanja.schindler@sns.it}

\begin{abstract}
   We study functions related to the classical Brjuno function, namely $k$-Brjuno functions and the Wilton function. Both appear in the study 
of boundary regularity properties of (quasi) modular forms and their integrals.
We consider various possible versions of them, based on the $\alpha$-continued fraction developments.
We study their BMO regularity properties and their behaviour near rational numbers of their finite truncations. 
We then complexify the functional 
equations which they fulfill and we construct analytic extensions of the $k$-Brjuno and of the Wilton function to the upper half plane. We study their boundary behaviour using an extension of the continued fraction algorithm to the complex plane. We also prove that the harmonic conjugate of 
the real $k$-Brjuno function is continuous at all irrational numbers and has a decreasing jump of $\pi/q^k$ at rational points $p/q$.
\end{abstract}

\thanks{The first and the fourth author acknowledge the support of the Centro di Ricerca Matematica Ennio de
Giorgi. The first, second and third author acknowledge the support of UniCredit Bank R\&D group for financial support through
the ‘Dynamics and Information Theory Institute’ at the Scuola Normale Superiore. The second and the
fourth author acknowledge the support through the PRIN Grant ”Regular and stochastic behaviour in
dynamical systems” (PRIN 2017S35EHN)}

\maketitle




\section{Introduction}

Let $\Gauss$ denote \emph{the Gauss map} defined on the interval $[0,1)$ by $\Gauss(0)=0$ and $\Gauss(x)=\left\{\frac{1}{x}\right\}$ otherwise, where $\{x\}=x-\lfloor x\rfloor$
and let
$$\beta_j(x) = \prod_{i=0}^j \Gauss^i(x)$$ 
for $j\ge 0$ with the convention $\beta_{-1}(x) = 1$. 

\subsection{Brjuno function}
In 1988, Yoccoz introduced the following function - now called \emph{Brjuno function} - defined for irrational
numbers $x\in[0,1]\setminus\Q$ as
\begin{equation*}\label{eq:real Brjuno}
B_1(x)=\sum_{n=0}^{\infty}\beta_{n-1}(x)\, \log\left(\frac{1}{\Gauss^n(x)}\right),
\end{equation*}
see \cite{Y, MMY1, MMY2}, see Figure~\ref{fig:BW} for its graph.

Let $\frac{p_n(x)}{q_n(x)}$ denote the $n$th convergent of $x$ with respect to its continued fraction expansion. 
The series $B_1(x)$ converges if and only if
\begin{equation*}
\sum_{n=0}^{\infty} \frac{\log(q_{n+1}(x))}{q_n(x)} < \infty.
\end{equation*}
This condition is called \emph{Brjuno condition} and was introduced by Brjuno in the study of certain problems in dynamical systems, see \cite{B,Br2}.
The points of convergence are called \emph{Brjuno numbers}.
The importance of Brjuno numbers comes from the study of analytic small
divisors problems in dimension one. Indeed, extending previous fundamental work of C.L. Siegel \cite{siegel1942}, 
Brjuno proved that all germs of holomorphic diffeomorphisms
of one complex variable with an indifferent fixed point
with linear part $e^{2\pi i x}$ are linearisable if $x$ is a Brjuno number. 
Conversely, in 1988 J.-C. Yoccoz \cite{Y, yoccoz1995}
 proved that this condition is also necessary. Similar results hold
for the local conjugacy of analytic diffeomorphisms of the circle \cite{yoccoz2002} 
and for some complex area–preserving maps \cite{marmi_1990, davie1994critical}. 
Moreover the Brjuno function provides a continuous interpolation of 
the radius of convergence of the quadratic polynomial \cite{buff2006brjuno} and the interpolation is conjectured to be 
H\"older continuous with exponent $1/2$ \cite{MMY1}. This has been recently proved for high-type irrational numbers  \cite{cheraghi2015proof}.
The Brjuno condition has been of interest also in different contexts. For instance, it is conjectured that it is optimal for the existence of real analytic invariant circles in the standard family \cite{mackay_exact_1988, mackay_erratum, marmi1992standard}. See also \cite{berretti2001bryuno, gentile2015invariant} and references therein for related results.

Furthermore, the Brjuno function is $\Z$--periodic and
satisfies the functional equation 
$$
B_1(x)=-\log(x)+xB_1\left(\frac{1}{x}\right)
$$
for $x\in (0,1)$.
The second author together with Moussa and Yoccoz investigated the regularity properties of $B_1$ in \cite{MMY1} and later 
constructed an analytic extension of $B_1$ to the upper-half plane $\uhp = \{z\in\C\colon \im z>0\}$
\cite{MMY3}. Let $T$ denote the linear operator 
\begin{equation}\label{eq: op T Gauss}
 Tf(x)=xf\Big(\frac{1}{x}\Big)
\end{equation}
acting, for example, on measurable $\Z$--periodic functions on $\R$. Then in all $L^p$ spaces the Brjuno function is the solution of the linear equation 
\begin{equation}\label{eq: brjuno fe}
[(1-T)B_1](x) = -\log x\, .
\end{equation}
By exploiting the fact that the operator $T$ as in \eqref{eq: op T Gauss} acting on $L^p$ spaces has spectral radius strictly
smaller than $1$, one can indeed obtain \eqref{eq:real Brjuno} by a Neumann series for $(1-T)^{-1}$, see \cite{MMY1} for details.

Local properties of the Brjuno function have been recently investigated by M. Balazard and
B. Martin \cite{BM1} and its multifractal spectrum was determined by S. Jaffard and B. Martin in \cite{jaffard_multifractal_2018}.

\subsection{$k$-Brjuno functions}
For the following, for $k\ge 2$ even, let $E_k$ be the Eisenstein series of weight $k$ defined in the upper-half plane $\uhp$. 
Then its Fourier expansion is given by 
\begin{equation*}
 E_k(z)=1-\frac{2k}{b_k}\sum_{n=1}^{\infty}\sigma_{k-1}(n)e^{2\pi i nz},
\end{equation*}
where $b_k$ is the $k$th Bernoulli number and $\sigma_{k-1}(n)=\sum_{d|n} d^{k-1}.$ For all $k\geq4$, $E_k$ is modular of weight $k$ under the action of $\SL_2(\Z)$, and $E_2$ is
quasi-modular of weight $2$ under the action of $\SL_2(\Z)$, see for example \cite{Z2}.
The function $E_2$ can be viewed as a modular (or Eichler) integral on $\SL_2(\Z)$ of weight $2$ with the rational period function $-\frac{2\pi i}{z}$, see for example \cite{Kn}.

For $k \geq 2$ even and $z \in \uhp$, denote $\varphi_k(z) = \sum_{n=1}^\infty \frac{\sigma_{k-1}(n)}{n^{k+1}}e^{2\pi i nz} .$
We have that
$$\sum_{n=1}^{\infty}\sigma_{k-1}(n)e^{2\pi i nz} = \left(\frac{1}{2\pi i}\frac{\partial}{\partial z}\right)^{k+1}\varphi_k(z) ,$$
and
$$\varphi_k(z) = \frac{B_k(2\pi i)^{k+1}}{k!2k} \int_{i\infty}^z(z-t)^k(E_k(t)-1)dt .$$
Consider the imaginary part of $\varphi_k$
\begin{equation*}
 F_{k}(x)= \sum_{n=1}^\infty \frac{\sigma_{k-1}(n)}{n^{k+1}} \sin(2\pi n x) \quad \text{for } x\in\R.
\end{equation*}

Analytic properties, differentiability and H\"older regularity exponent, of the function $F_k$ (and the real part of $\varphi_k$) were studied by the third author. It has been proved that for $F_k$ the differentiability is related to a condition resembling the Brjuno condition.
Considering the special case $k=2$, the third author proved that if $\sum_{n=0}^{\infty} \frac{\log(q_{n+1}(x))}{q_n^2(x)} < \infty$ 
and $\lim_{n \to \infty} \frac{\log (q_{n+4}(x))}{q_n^2(x)} =0$, then $F_{2}$ is differentiable at $x\in\R\setminus\Q$, whereas if
$\sum_{n=0}^{\infty} \frac{\log(q_{n+1}(x))}{q_n^2(x)}$ diverges, then $F_{2}$ is not differentiable at $x\in\R\setminus\Q$. 
It has been conjectured that for all $k\in \N$ even, $F_k$ is differentiable at $x\in\R\setminus\Q$ if and only if it fulfills the \emph{ $k$-Brjuno condition}
\begin{equation}\label{eq: kBrjuno cond}
 \sum_{n=0}^{\infty} \frac{\log(q_{n+1}(x))}{q_n^k(x)} < \infty,
\end{equation}
see \cite{P2, P3}.
The occurrence of a condition of this type motivates the following definition.

\medskip

For $k\in \N$, let 
\begin{equation}\label{eq: Bk with Gauss}
B_k(x)=\sum_{n=0}^{\infty}(\beta_{n-1}(x))^k\log\left(\frac{1}{\Gauss^n(x)}\right)
\end{equation}
be called \emph{$k$-Brjuno function}. 
From this equation we already get the implicit definition in terms of an analogue of the functional equation of the Brjuno function: if $x\in (0,1)$
 $$
B_k(x)= -\log(x)+x^k\cdot B_k(\Gauss(x))\, ,
$$
see \cite{MMY1, MMY2}.  
The $k$--Brjuno function converges at an irrational $x$ if and only if \eqref{eq: kBrjuno cond} holds, see also Proposition \ref{prop: kBrjuno equiv to prop2.3} for an even stronger statement about the relation between the $k$-Brjuno function and \eqref{eq: kBrjuno cond}. 
 Obviously, for $k=1$, the function in \eqref{eq: Bk with Gauss} gives the Brjuno function introduced before.

\subsection{$\alpha$-continued fractions}
Instead of considering the $k$-Brjuno function with respect to the Gauss map as in \eqref{eq: Bk with Gauss},  it is also possible to use $\alpha$-continued fractions instead.
The classical Brjuno function associated to $\alpha$-continued fractions was already investigated in \cite{MMY1}.
Let $\alpha\in\left[\frac{1}{2},1\right]$ and let $A_\alpha:(0,\alpha)\to[0,\alpha]$ the transformation of \emph{the $\alpha$-continued fractions} being given by 
\begin{equation}
A_{\alpha}(x)=\left|\frac{1}{x}-\left\lfloor \frac{1}{x}-\alpha+1\right\rfloor\right|.\label{eq: def Aalpha}
\end{equation}
For $\alpha=1$, we obtain the Gauss map associated to the regular continued fraction transformation and for $\alpha=1/2$, we obtain the transformation associated to the nearest integer continued fractions.
Nakada \cite{N} was the first to consider all these types of continued fractions as a one-parameter family, however he considered the `unfolded' version of the $\alpha$-continued fraction which is defined by the Gauss map
$\widetilde{A_\alpha}(x) = \frac{1}{x}-\left\lfloor\frac{1}{x}-\alpha+1\right\rfloor.$
The version as in \eqref{eq: def Aalpha} was little later considered in \cite{TanakaIto}. 

Further, let
\begin{equation*}
\B{k}{\alpha}(x)=\sum_{n=0}^{\infty} \left(\be{\alpha}{n-1}(x)\right)^k \log\left(\frac{1}{A_\alpha^n(x)}\right),
\end{equation*}
where $\be{\alpha}{j}(x)=\prod_{i=0}^j A_\alpha^i(x)$ for $j\ge 0$ and $\be{\alpha}{-1}(x)=1$,
be the generalisation of \eqref{eq: Bk with Gauss} in the sense that we consider the $k$-Brjuno function not only for the Gauss transformation but also for other $\alpha$-continued fraction transformations $A_{\alpha}$ with $\alpha\in [1/2, 1]$. 

For given $\alpha\in [1/2,1]$, any $x\in (0,1]$ has the $\alpha$-continued fraction expansion given by $$x = \dfrac{1}{a_1+\dfrac{\epsilon_1}{\ddots+\dfrac{\epsilon_{j-1}}{a_j+\ddots}}},$$
where $a_j := a_j^{(\alpha)}(x):=\left\lfloor \frac{1}{A_{\alpha}^{j-1}(x)}-\alpha+1\right\rfloor$ and $\epsilon_j:=\epsilon_j^{(\alpha)}(x)$ is the sign of $\frac{1}{A_{\alpha}^{j-1}(x)} - a_j$.

Related to \eqref{eq: op T Gauss} we define an operator $\T{k}{\alpha}$ acting on $\Z$--periodic measurable functions $f$
such that $f(-x)=f(x)$ for a.e. $x\in (0,1-\alpha )$
 as
\begin{equation}\label{eq:k-Brjuno operator}
\T{k}{\alpha} f(x)=x^k f\left(\frac{1}{x}\right), \quad x\in (0,\alpha).
\end{equation}
It is the operator $T^{(\alpha)}_\nu$ which is introduced in \cite{MMY1}, where $\nu$ corresponds to the exponent $k$ in \eqref{eq:k-Brjuno operator}.
It is understood that the function $\T{k}{\alpha} f$ is completed outside the interval $(0,\alpha )$ by imposing to $\T{k}{\alpha} f$  the same parity and periodicity conditions 
imposed to $f$.

Then we have
\begin{equation}\label{eq:brjuno func eq}
[(1-\T{k}{\alpha})\B{k}{\alpha}](x) = -\log x, \quad x\in(0,\alpha)
\end{equation}
which follows by a simple calculation.

\subsection{Wilton function}
Next we consider the related concept of the \emph{Wilton function} which is given by
\begin{equation*}
W(x)=\sum_{n=0}^{\infty} (-1)^n \beta_{n-1}(x)\log\left(\frac{1}{\Gauss^n(x)}\right), 
\end{equation*}
namely by the alternate signs version of the Brjuno function series \eqref{eq:real Brjuno}.
\begin{figure}
$\begin{matrix}
\includegraphics[width=7.9cm]{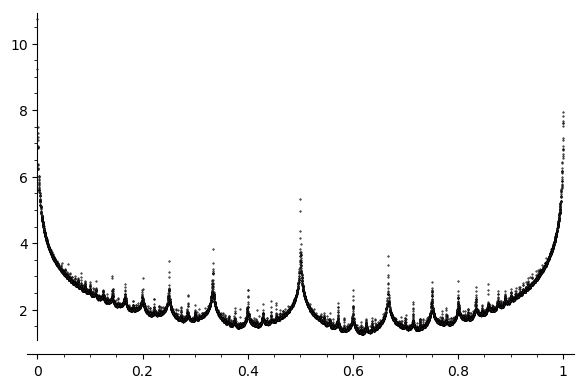}
&
\includegraphics[width=7.9cm]{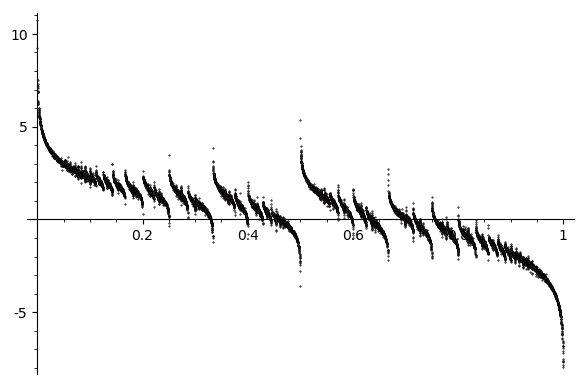}
\end{matrix}
$
\caption{Numerical computation of the Brjuno function $B$ (left) and of the Wilton function $W$ (right) when $\alpha=1$. The asymmetric 
logarithmic singularities at rational points provide an intuitive justification for $W$ not belonging to the BMO space, see Section \ref{sec:BMO}.}
\label{fig:BW}
\end{figure}
(see Figure~\ref{fig:BW} for its graph). It converges if and only if it fulfills the \emph{Wilton condition}
\begin{equation*}
\left|\sum_{n=0}^\infty (-1)^n \frac{\log(q_{n+1}(x))}{q_n(x)}\right| < \infty,
\end{equation*}
see \cite[Prop.~7]{BM2} and Remark \ref{rmk: Wilton cond} for an even stronger connection between the Wilton function and the Wilton condition.
The points of convergence are called \emph{Wilton numbers}
and appear in the work of Wilton, see \cite{W}. Clearly all Brjuno numbers are Wilton, but not vice versa (it is not difficult to build counterexamples 
by using the continued fraction).

The function $W$ satisfies the functional equation for $x\in (0,1)$ being a Wilton number: 
\begin{equation*}
W(x)=\log\left(\frac{1}{x}\right)-xW(\Gauss(x)),
\end{equation*}
which by using the same linear operator $T$ as in \eqref{eq: op T Gauss} can be written as 
\begin{equation}\label{eq: wilton fe}
[(1+T)W](x) = -\log x.
\end{equation}

We can extend the Wilton function in the same way we did for the $k$-Brjuno functions, i.e.\ replacing the Gauss map with the transformation of the $\alpha$-continued fractions: 
for $\alpha\in\left[\frac{1}{2},1\right)$ and for all irrational $x$ we define
\begin{equation*}
W_\alpha(x)=\sum_{n=0}^{\infty} (-1)^n \be{\alpha}{n-1}(x) \log\left(\frac{1}{A_\alpha^n(x)}\right).
\end{equation*}

We then have 
$$
[(1-\Ss{\alpha})W_\alpha](x) = -\log x, \quad x\in(0,\alpha),
$$
where the operator $\Ss{\alpha}=-\T{1}{\alpha}$. Also in this case it is understood that 
$\Ss{\alpha}$ acts on $\Z$--periodic measurable functions $f$
such that $f(-x)=f(x)$ for a.e. $x\in (0,1-\alpha )$.

The Wilton function and its primitive have been studied recently by Balazard and Martin
in terms of its convergence properties \cite{balazard_equations_2019} and in the context of the Nyman and Beurling criterion, see \cite{BM1, BM2} and \cite{BBLS}.

\subsection{BMO properties and complex $k$-Brjuno and Wilton functions.}

In \cite{MMY1} the real regularity properties of the Brjuno function were systematically investigated by exploiting the associated functional equation, culminating with the proof that the Brjuno function belongs to the space $\textnormal{BMO}$ of functions with Bounded Mean Oscillation (we refer to 
\cite{garnett2007bounded, garcia2011weighted} for its definition and properties). Here (see Section \ref{sec:BMO})
we extend those results to the $k$-Brjuno function, 
which turn out to belong to BMO,  and to the Wilton function, which does not. Moreover we consider their generalizations obtained by replacing the Gauss map with $\alpha$-continued fractions \cite{N}, $\alpha \in \left[\frac{1}{2},1\right]$ and we prove that for all $\alpha$ the associated $k$--Brjuno functions belong to BMO and that the same holds for the associated Wilton functions when $\alpha$ is restricted to belong to the interval $\left[\frac{1}{2},\frac{\sqrt{5}-1}{2}\right]$.

By Fefferman's duality theorem, see for example \cite[p.\ 39]{sarason}, one can add an $L^\infty$ function to a BMO function and obtain that the harmonic conjugate of the sum will be an essentially bounded function. One is then led to look for a periodic holomorphic function defined on the upper half plane whose imaginary parts, when looked upon $\R$, is the original function. As in \cite{MMY3} we make this construction for the $k$-Brjuno functions and for the Wilton function. The associated {\it complex $k$-Brjuno functions} $\mathcal{B}_k$ and {\it complex Wilton function} $\mathcal{W}$
and their properties are summarized in the following Theorem: 
\begin{thm}\label{thm:summarycomplex}
\begin{enumerate} 
\item The complex $k$-Brjuno function $\mathcal{B}_k\, , \, \uhp\rightarrow\C$ is given by the series
\begin{align}
\cB_k(z) 
&= - \sum_{p/q\in\Q} \det\begin{pmatrix}p'&p\\q'&q\end{pmatrix}^{k+1}\notag\\
&\qquad\cdot \Bigg\{\frac{1}{\pi}\left[ (p'-q'z)^k \Li{\frac{p'-q'z}{qz-p}} - (q''z-p'')^k\Li{\frac{p''-q''z}{qz-p}}\right]\notag\\
& \quad\qquad + \sum_{n=0}^k \frac{\det \left(\begin{smallmatrix}p'&p\\q'&q\end{smallmatrix}\right)^{n} }{n\pi q^n} 
\Bigg[(p'-q'z)^{k-n} \, \left(\log\left(1+\frac{q'}{q}\right)-\sum_{i=1}^{n-1}\frac{1}{i}\left(\frac{1}{(1+q'/q)^i}-1\right)\right) \notag\\ 
&\qquad \qquad - (q''z-p'')^{k-n} \,  \left(\log\left(1+\frac{q''}{q}\right)-\sum_{i=1}^{n-1}\frac{1}{i}\left(\frac{1}{(1+q''/q)^i}-1\right)\right)\Bigg] \Bigg\},\label{eq: complex kBrjuno}
\end{align}
where $[\frac{p'}{q'},\frac{p''}{q''}]$ is the Farey interval such that $\frac{p}{q}=\frac{p'+p''}{q'+q''}$ (with the convention $p'=p-1$, $q'=1$, $p''=1$, $q''=0$ if $q=1$) and $\Li{z}$ is the dilogarithm. 
The complex Wilton function $\mathcal{W}\, , \, \uhp\rightarrow\C$ is given by the series
\begin{align}
 \cW(z)  = -\frac{1}{\pi}\sum_{p/q\in\Q} &\Bigg\{\det\begin{pmatrix}p'&p\\q'&q\end{pmatrix}(p'-q'z)\left[\Li{\frac{p'-q'z}{qz-p}}-\Li{-\frac{q'}{q}}\right] \notag\\
&\quad +\det\begin{pmatrix}p''&p\\q''&q\end{pmatrix}(p''-q''z)\left[\Li{\frac{p''-q''z}{qz-p}}-\Li{-\frac{q''}{q}}\right]\notag\\
&\quad+\frac{1}{q}\log\left(\frac{q^2}{(q+q')(q+q'')}\right) \Bigg\}\,.\label{eq: complex Wilton}
\end{align}
\item The real part of $\mathcal{B}_k$ is bounded on the upper half plane and its 
non–tangential limit on $\R$ is continuous at all irrational points and has a decreasing jump of $\frac{\pi}{q^k}$ at each rational number $\frac{p}{q}$. 
\item As one approaches the boundary the imaginary part  of $\mathcal{W}$ behaves as follows:
\begin{enumerate}
\item if $x\in\R$ is a  Wilton number, then $\im\mathcal{W}(x+w)$ converges to $W(x)$ as $w\rightarrow 0$ 
in any domain with a finite order of tangency to the real axis;
\item if $x\in\R$ is Diophantine, then in both cases one can allow domains with infinite order of tangency to the real line.
\end{enumerate}
\end{enumerate}
\end{thm}
As the Wilton function is not BMO, it also becomes clear by Fefferman's theorem that its harmonic conjugate can not be an $L^{\infty}$ function.

\subsection{Structure of the paper}

The paper is organised as follows. In Section \ref{sec:BMO} we state and prove the $\textnormal{BMO}$-properties of the $k$-Brjuno and the Wilton function and in Section \ref{sec:kBrjuno real} we give statements about the truncated real $k$-Brjuno and the truncated real Wilton function which corresponds to the finite part of the sum at rational numbers.

The subsequent sections deal with a complexification of the system. In Section \ref{sec:compl CF} we introduce a complex continued fraction algorithm. 
In Section \ref{sec:complexTkS} we extend the operators $T_k:=T_{k,1}$ and $S:= S_{1}$ to the complex plane. The main findings concerning these operators are given in Section \ref{sec:sumboundary}. 
Particularly, Proposition \ref{prop: main prop kBrjuno} indicates how $T_k$ behaves. The analogous behaviour of $S$ can then immediately be deduced by using the fact that $S=-T_1$.

Finally, with the help of these operators we define the complex $k$-Brjuno and the complex Wilton function in Section \ref{sec: compl functions}.

Those proofs which are very similar to the ones in \cite{MMY3} are given in an appendix.

\section{$\textnormal{BMO}$ properties of the real $k$-Brjuno and the real Wilton function}\label{sec:BMO}
In this section, we study the bounded mean oscillation ($\textnormal{BMO}$) properties of the real $k$-Brjuno and the Wilton function - both with respect to different transformations $A_{\alpha}$ with $\alpha\in [1/2, 1]$. 
Before stating the main results of this section, we will first recall the definition of a $\textnormal{BMO}$ function.

Let $L^1_{\text{loc}}(\R)$ be the space of the locally integrable functions on $\R$.
Recall that the mean value of a function $f\in L^1_{\text{loc}}(\R)$ on an interval $I$ is defined as
\begin{equation*}
f_I=\frac{1}{|I|}\int_{I}f(x) dx.
\end{equation*}
For an interval $U$, we say that a function $f\in \textnormal{BMO}(U)$ if
\begin{equation}\label{eq:BMO seminorm}
\Vert f \Vert _{*,U}:=\sup_{I \subset U}\frac{1}{|I|}\int_{I}|f(x)-f_I|dx<\infty.
\end{equation}
For further properties of the BMO space, see for example \cite[Appendix]{MMY1} and the monographies 
\cite{garnett2007bounded, garcia2011weighted}.

In the following, we will state the main properties of this section which show that the $\textnormal{BMO}$ properties fundamentally differ between the $k$-Brjuno functions and the Wilton function. We first give the statement for the $k$-Brjuno functions.
\begin{prop}\label{prop: kBrjuno is BMO}
For all $k\in\mathbb{N}$ and all $\alpha\in [1/2, 1]$, the $k$-Brjuno function $\B{k}{\alpha}$ is a $\textnormal{BMO}$ function.
\end{prop}
Contrarily, for the Wilton function we have the following statement:
\begin{thm}\label{prop: Wilton not BMO}
The Wilton function $W=W_1$ is not a $\textnormal{BMO}$ function.
\end{thm}
On the other hand, we define for the following $g:=\frac{\sqrt{5}-1}{2}$ and have:
\begin{thm}\label{prop: Wilton BMO}
 For all $\alpha\in\left[\frac{1}{2}, g\right]$, the function $W_\alpha$ is a $\textnormal{BMO}$ function.  
\end{thm}

Before we start with the proofs of the statements above, we first want to give some remarks about them: 
Proposition \ref{prop: kBrjuno is BMO} is an extension of \cite[Thm.~3.2]{MMY1} from the classical Brjuno function to $k$-Brjuno functions. 

As a comparison to Figure \ref{fig:BW} showing $B_{1,1}$ and $W_1$, in Figure \ref{fig: graphs}, some numerical simulations of $W_{\alpha}$ with different values of $\alpha$ are shown: the numerical evidence 
supports the conjecture that also for $\alpha\in (g,1)$ the function $W_{\alpha}$ is a $\textnormal{BMO}$ function. 
Indeed, the graph for $\alpha =e-2$ and $\alpha =0.9$ (and nearby values, not displayed in Figure \ref{fig:BW}) 
suggest that the function $W_\alpha$ exhibits a mix of {\it symmetric} logarithmic singularities at rational points and of jumps, a singular behaviour
which is compatible with belonging to BMO.
However, unfortunately, the results from Proposition \ref{prop: Wilton BMO} can not immediately be transferred to $\alpha\in (g,1)$, see Remark \ref{rmk: Wilton alpha large} for an explanation which difficulties occur.

\begin{figure}[h]
$\begin{matrix}
\includegraphics[width=7cm]{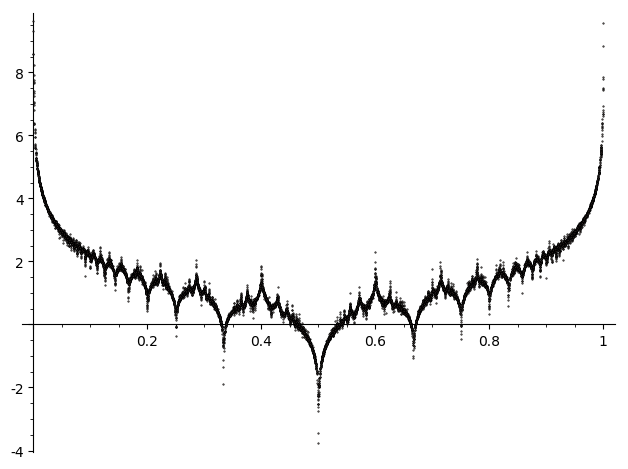}
&
\raisebox{1.7ex}{\includegraphics[width=7cm]{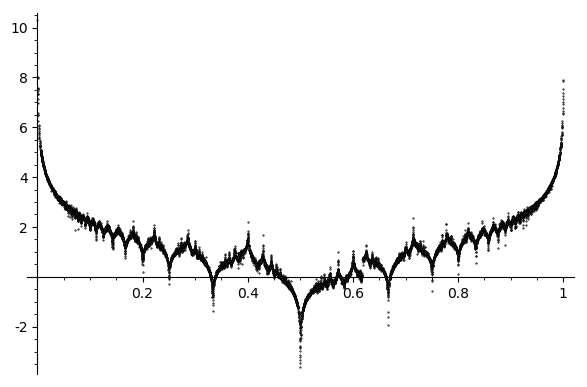}}
\\
\includegraphics[width=7cm]{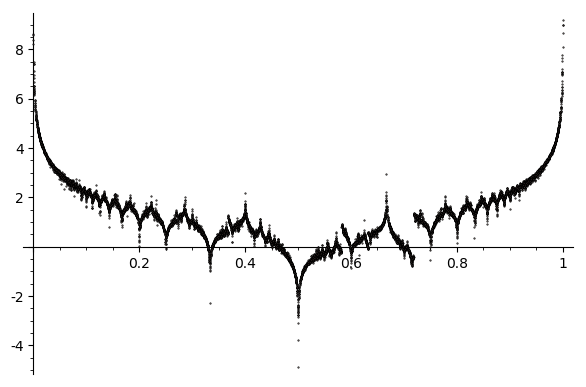}
&
\includegraphics[width=7cm]{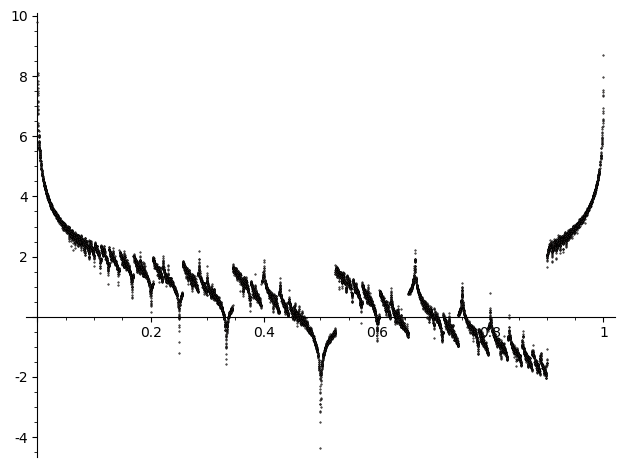}
\end{matrix}$
\caption{Numerical computations of $W_\alpha$ for $\alpha=1/2$ (upper left), $\alpha=\frac{\sqrt{5}-1}{2}$ (upper right), $\alpha=e-2$ (lower left) and $\alpha=0.9$ (lower right).}\label{fig: graphs}
\end{figure}

\subsection{Proof that the real $k$-Brjuno functions are $\textnormal{BMO}$ functions}\label{subsec: real kBrjuno BMO}

The main idea of the proof is to prove the statement for $\B{k}{1/2}$, see Proposition \ref{prop: thm3.3 MMY1}, and to show then that $\B{k}{\alpha}$ differs from $\B{k}{1}$ only by an $L^{\infty}$ function which follows from Propositions \ref{prop: kBrjuno equiv to prop2.3} and \ref{prop: kBrjuno equiv to prop2.4}. 
As the proof follows from very similar arguments as those in \cite[Thm.~3.2]{MMY1}, we will only describe shortly the necessary changes in the proofs. 

We start by introducing 
\begin{equation}\label{eq: def Xast}
 X_\ast = \{f\in\text{BMO}(\R): f(x+1)=f(x) \text{ for all }x\in\R, ~ f(-x)=f(x) \text{ for all }x\in[0,1/2]\}.
\end{equation}
endowed with a norm which is the sum of the BMO seminorm $\|\!\cdot\!\|_{\ast,[0,1/2]}$ as in \eqref{eq:BMO seminorm} and of the $L^2$ norm on the interval $(0,1/2)$ (w.r.t. the $A_{1/2}$-invariant
probability measure). Then one has:
\begin{prop}[{\cite[Thm.~3.3]{MMY1}}]\label{prop: thm3.3 MMY1}
For all $k \in\mathbb{N}$, the operator $\T{k}{1/2}$ as in \eqref{eq:k-Brjuno operator} is a bounded linear operator
from $X_\ast$ to $X_\ast$ whose spectral radius is bounded above by $(\sqrt{2}-1)^k$.
\end{prop}

 To proceed, we prove an analog of \cite[Prop.~2.3, eq.~(iv)]{MMY1}.
\begin{prop}\label{prop: kBrjuno equiv to prop2.3}
 For all $k\in\N$, there exists a constant $C_{1,k} > 0$ such that for all $\alpha\in[1/2,1]$ and $x \in \mathbb{R} \backslash \mathbb{Q}$, one has
 \begin{align*}
  \left| \B{k}{\alpha}(x)-\sum_{j=0}^{\infty}\frac{\log q_{j+1}^{(\alpha)}(x)}{\big(q_j^{(\alpha)}(x)\big)^k} \right|<C_{1,k}.
 \end{align*}
\end{prop}

Before we start with the proof, we recall the following property.
\begin{rmk}\label{rmk:rmk1.7}
In \cite[Remark 1.7]{MMY1} it is showed that, for all $\alpha\in [1/2,1]$,
$$\sum_{j=0}^\infty\frac{\log q_j^{(\alpha)}}{q_j^{(\alpha)}}\le 2/e=: c_1\text{ and }\sum_{j=0}^\infty \frac{\log 2}{q_j^{(\alpha)}}\le 5\log2 =: c_2.$$
\end{rmk}

\begin{proof}[Proof of Proposition~\ref{prop: kBrjuno equiv to prop2.3}]
For the following calculations we drop the dependence on $\alpha$ and $x$.
 We obtain by analogous calculations as in \cite[Prop.~2.3, eq.~(iv)]{MMY1} that
 \begin{align*}
  &\bigg|-\B{k}{\alpha}(x) +\sum_{j=0}^{\infty}\frac{\log q_{j+1}}{q_j^k} \bigg|\\
  &=\bigg| \sum_{j=0}^{\infty}  \beta_{j-1}^k\log \frac{\beta_j}{\beta_{j-1}} 
+ \sum_{j=0}^{\infty}\left(\beta_{j-1}  + \epsilon_{j} \frac{q_{j-1}}{q_j} \beta_j\right)^k\log q_{j+1} \bigg|\\
  &=\bigg|\sum_{j=0}^{\infty}  \beta_{j-1}^k\log (\beta_jq_{j+1}) 
-\sum_{j=0}^{\infty}\beta_{j-1}^k  \log \beta_{j-1}+ 
\sum_{j=0}^{\infty} \left(\left( \beta_{j-1}+\epsilon_{j} \frac{q_{j-1}}{q_j}\beta_j\right)^k -\beta_{j-1}^k\right)\log q_{j+1}\bigg|\\
&\le \sum_{j=0}^{\infty}  \left|\beta_{j-1}^k\log (\beta_jq_{j+1})\right| 
+\sum_{j=0}^{\infty}\left|\beta_{j-1}^k  \log \beta_{j-1}\right|\\
&\qquad +\sum_{j=0}^{\infty}\left|\left(\left( \beta_{j-1}+\epsilon_{j} \frac{q_{j-1}}{q_j}\,\beta_j\right)^k -\beta_{j-1}^k\right)\log q_{j+1}\right|.
\end{align*}
We notice that $|\beta_i|\le 1$ and thus 
\begin{align*}
 \sum_{j=0}^{\infty}  \left|\beta_{j-1}^k\log (\beta_jq_{j+1})\right|\le 
 \sum_{j=0}^{\infty}  \left|\beta_{j-1}\log (\beta_jq_{j+1})\right|\le 2c_2,
\end{align*}
where the last estimate follows as in the proof of \cite[Prop.~2.3, eq.~(iv)]{MMY1} and $c_2$ is given in Remark~\eqref{rmk:rmk1.7}.
Furthermore, 
\begin{align*}
 \sum_{j=0}^{\infty}\left|\beta_{j-1}^k  \log \beta_{j-1}\right|
 &\le \sum_{j=0}^{\infty}\left|\beta_{j-1}  \log \beta_{j-1}\right|\le 2(c_1+c_2)
\end{align*}
which also follows as in the proof of \cite[Prop.~2.3, eq.~(iv)]{MMY1} and $c_1$ is given in Remark~\eqref{rmk:rmk1.7}.
Finally, we have 
\begin{align*}
 \MoveEqLeft\sum_{j=0}^{\infty}\left|\left(\left( \beta_{j-1}+\epsilon_{j} \frac{q_{j-1}}{q_j} \beta_j \right)^k -\beta_{j-1}^k\right)\log q_{j+1}\right|\\
 &\le \sum_{j=0}^{\infty}\left|\left(\sum_{\ell=1}^k {k\choose \ell}\beta_{j-1}^{k-\ell}\cdot \left(\epsilon_{j} \frac{q_{j-1}}{q_j}\, \beta_j\right)^{\ell}\right)\log q_{j+1}\right|\\
 &\le \sum_{j=0}^{\infty}\left|(2^{k}-1) \frac{q_{j-1}}{q_j}\, \beta_j\log q_{j+1}\right|
 \le 2^{k+1} c_1,
\end{align*}
where the last estimate again follows analogously to \cite[Prop.~2.3, eq.~(iv)]{MMY1} and the proof can be completed in the same way as in \cite[Prop.~2.3, eq.~(iv)]{MMY1}.
\end{proof}

\begin{rmk}\label{rmk: Wilton cond}
 With analogous methods as above using the absolute values of the sum we also obtain the following statement:
 There exists a constant $C > 0$ such that for all $\alpha\in[1/2,1]$ and $x \in \mathbb{R} \backslash \mathbb{Q}$, one has
 \begin{align*}
  \left| W_{\alpha}(x)-\sum_{j=0}^{\infty}(-1)^j\frac{\log  q_{j+1}^{(\alpha)}(x)}{q_j^{(\alpha)}(x)} \right|<C.
 \end{align*}
\end{rmk}

The next proposition is an equivalent to \cite[Prop.~2.4]{MMY1}.
\begin{prop}\label{prop: kBrjuno equiv to prop2.4}
 For all $k\in\N$, there exists a constant $C_{2,k} > 0$ such that for all $\alpha\in [1/2, 1]$ and for all $x \in \mathbb{R} \backslash \mathbb{Q}$ one has
 \begin{align*}
  \left|\B{k}{\alpha}(x) - \sum_{j=0}^{\infty}\frac{\log q_{j+1}^{(1)}(x)}{(q_j^{(1)}(x))^k} \right|\le  C_{2,k}.
 \end{align*}
\end{prop}
\begin{proof}
 The proof follows completely analogously to that of \cite[Prop,~2.4]{MMY1} with the only difference that we use Proposition \ref{prop: kBrjuno equiv to prop2.3} instead of \cite[Prop.~2.3]{MMY1} and each $q_j^{(1)}$ (which is denoted by $Q_j$ in \cite{MMY1}) in the denominator is replaced by $(q_j^{(1)})^k$.
\end{proof}

\begin{proof}[Proof of Proposition \ref{prop: kBrjuno is BMO}]
By Proposition~\ref{prop: thm3.3 MMY1}, $1-\T{k}{\alpha}$ is invertible on $X_\ast$.
A $\Z$--periodic even function equal to $-\log x$ on $(0,1/2]$ is in $X_\ast$.
Thus, $\B{k}{1/2}$ is $\textnormal{BMO}$.
Since by Proposition \ref{prop: kBrjuno equiv to prop2.4}
for all $\alpha\in [1/2,1]$
the $k$-Brjuno function $\B{k}{\alpha}$ differs from $\B{k}{1/2}$ only by an $L^{\infty}$ function, $\B{k}{\alpha}$ is a $\textnormal{BMO}$ function as well.
\end{proof}

\subsection{Proofs of the $\textnormal{BMO}$ properties of the Wilton function}

\begin{proof}[Proof of Theorem \ref{prop: Wilton not BMO}]
For brevity, in the following we write  $O_I(f)=\frac{1}{|I|}\int_{I}|f(x)-f_I|dx$. Furthermore, as we only consider $\alpha=1$, we also always write $W$ instead of $W_1$.
By \cite[Prop.~A.7]{MMY4}, if $I_1$ and $I_2$ are two consecutive intervals, then
\begin{equation}\label{oscil}
O_{I_1\cup I_2}(f)=\frac{|I_1|}{|I_1|+|I_2|}O_{I_1}(f)+ \frac{|I_2|}{|I_1|+|I_2|}O_{I_2}(f)+\frac{2|I_1||I_2|}{(|I_1|+|I_2|)^2}|f_{I_1}-f_{I_2}|.
\end{equation}
Let $I_{n}:=\left[-\frac{1}{n},\frac{1}{n}\right]=\left[-\frac{1}{n},0\right]\cup\left[0,\frac{1}{n}\right] =:I_n^-\cup I_n^+$. 
By \eqref{oscil}, $O_{I_n}(f)\geq \frac{1}{2}|f_{I_n^-}-f_{I_n^+}|$ for any $f\in L^1_{\text{loc}}(\R)$.
By \cite[Lem.~2]{BM2}, we have
\begin{equation*}
\int_{0}^{x}W(t)\mathrm{d}t=x\log(1/x)+x+O(x^2),
\end{equation*} and 
\begin{equation*}
\int_{1-x}^1W(t)\mathrm{d}t=-x\log(1/x)-x+O(x^2\log(2/x)),
\end{equation*}
thus we clearly have
\begin{equation*}
 W_{I_n^+}=\log(n)+1+O(1/n).
\end{equation*}
Since $W$ is $\Z$--periodic, we also have
\begin{equation*}
W_{I_n^-}=-\log(n)-1+O(n^{-1}\log(2n)).
\end{equation*}
Thus, 
\begin{equation*}
O_{I_n}(W)\geq \log(n)+1+O(n^{-1}\log(2n)),
\end{equation*}
which completes the proof of the proposition.
\end{proof}

For the proof of Theorem \ref{prop: Wilton BMO}, we can not use exactly the same strategy as for the proof of Proposition \ref{prop: kBrjuno is BMO}. The reason is that, as we have seen in Theorem \ref{prop: Wilton not BMO}, $W_1$ is not a $\textnormal{BMO}$ function. Hence, comparing $W_{\alpha}$ for $\alpha<1$ with $W_1$ can not work. Instead, the underlying idea of the proof is to use that $W_{1/2}$, the Wilton function with respect to the nearest integer continued fraction, is a $\textnormal{BMO}$ function and compare $W_{1/2}$ with $W_{\alpha}$ for $\alpha\in [1/2,g]$, see Proposition \ref{prop: unif bdd W_1/2 - W_alpha}.

Let $X_\ast$ be defined as in \eqref{eq: def Xast}. Since $\Ss{1/2}=-\T{1}{1/2}$ clearly by Proposition \ref{prop: thm3.3 MMY1}
we have $W_{1/2}\in X_\ast$. On the other hand one has

\begin{prop}\label{prop: unif bdd W_1/2 - W_alpha}
For $\alpha\in[\frac12,g]$, we have $W_{1/2}-W_\alpha\in L^\infty$, where
$g=\frac{\sqrt{5}-1}{2}$.
\end{prop}

\begin{proof}
Let $x\in[0,\frac12]$.
Since $\alpha\in[\frac 12,g]$, we have $2-\frac1\alpha\le 1-g\le \frac1{2+\alpha}\le 1-\alpha$.
Recall that
$$
A_{1/2}(x) = \begin{cases}
\frac 1x - k & \text{ if }x\in(\frac{1}{k+1/2},\frac 1k], \\
(k+1)-\frac1x & \text{ if } x\in(\frac{1}{k+1},\frac{1}{k+1/2}],
\end{cases}
\quad\text{ for }k\ge 2.
$$
Since $\alpha\le \frac1{1+\alpha}$, we have
$$
A_{\alpha}(x) = \begin{cases}
\frac 1x - k & \text{ if }x\in(\frac{1}{k+\alpha},\frac 1k], \\
(k+1)-\frac1x & \text{ if } x\in(\frac{1}{k+1},\frac{1}{k+\alpha}], \\
2-\frac1x & \text{ if } x\in(\frac 12,\alpha],
\end{cases}
\quad\text{ for }k\ge 2.
$$
See Figure~\ref{fig:graphofA} for the graphs of $A_{1/2}$ and $A_\alpha$ for a typical $\alpha\in(1/2,g)$.
\begin{figure}[h] 
$\begin{matrix}
\begin{tikzpicture}[scale=8]
\draw[black] (0,0) -- (3/5,0);
\draw[black] (0,0) -- (0,3/5);
\draw[domain=2/5:1/2,black] plot (\x, 1/\x-2);
\draw[domain=1/3:2/5,black] plot (\x, 3-1/\x);
\draw[domain=2/7:1/3,black] plot (\x, 1/\x-3);
\draw[domain=1/4:2/7,black] plot (\x, 4-1/\x);
\draw[domain=2/9:1/4,black] plot (\x, 1/\x-4);
\draw[domain=1/5:2/9,black] plot (\x, 5-1/\x);
\draw[domain=2/11:1/5,black] plot (\x, 1/\x-5);
\draw[domain=1/6:2/11,black] plot (\x, 6-1/\x);
\draw[domain=2/13:1/6,black] plot (\x, 1/\x-6);
\draw[domain=1/7:2/13,black] plot (\x, 7-1/\x);
\draw[domain=2/15:1/7,black] plot (\x, 1/\x-7);
\draw[domain=1/8:2/15,black] plot (\x, 8-1/\x);

\draw[black] (3/5,-0.01) -- (3/5,0.01);
\draw[black] (1/2,-0.01) -- (1/2,0.01);
\draw[black] (2/5,-0.01) -- (2/5,0.01);
\draw[black] (1/3,-0.01) -- (1/3,0.01);
\draw[black] (2/7,-0.01) -- (2/7,0.01);
\draw[black] (1/4,-0.01) -- (1/4,0.01);
\draw[black] (2/9,-0.01) -- (2/9,0.01);

\draw[black, dotted] (3/5,0) -- (3/5,3/5);
\draw[black, dotted] (1/2,0) -- (1/2,3/5);
\draw[black, dotted] (2/5,0) -- (2/5,3/5);
\draw[black, dotted] (1/3,0) -- (1/3,3/5);
\draw[black, dotted] (2/7,0) -- (2/7,3/5);
\draw[black, dotted] (1/4,0) -- (1/4,3/5);
\draw[black, dotted] (2/9,0) -- (2/9,3/5);
\draw[black, dotted] (1/5,0) -- (1/5,3/5);
\draw[black,dotted] (0,1/2)--(3/5,1/2);
\draw[black,dotted] (0,3/5)--(3/5,3/5);
\node[black] at (3/5,-0.04) {$\alpha$};
\node[black] at (1/2,-0.04) {$\frac12$};
\node[black] at (2/5+0.015,-0.04) {$\frac25$};
\node[black] at (1/3,-0.04) {$\frac13$};
\node[black] at (2/7+0.01,-0.04) {$\frac27$};
\node[black] at (1/4,-0.04) {$\frac14$};
\node[black] at (2/9,-0.04) {$\frac29$};
\draw[black] (-0.01,1/2)--(0.01,1/2);
\node[black] at (-0.04,1/2) {$\frac 12$};

\end{tikzpicture}
&
\begin{tikzpicture}[scale=8]
\draw[black] (0,0) -- (3/5,0);
\draw[black] (0,0) -- (0,3/5);
\draw[domain=1/2:3/5, black] plot (\x, 2-1/\x);
\draw[domain=1/(2+3/5):1/2,  black] plot (\x, 1/\x-2);
\draw[domain=1/3:1/(2+3/5), samples=200,  black] plot (\x, 3-1/\x);
\draw[domain=1/(3+3/5):1/3, black] plot (\x, 1/\x-3);
\draw[domain=1/4:1/(3+3/5), samples=200,  black] plot (\x, 4-1/\x);
\draw[domain=1/(4+3/5):1/4, black] plot (\x, 1/\x-4);
\draw[domain=1/5:1/(4+3/5), samples=200,  black] plot (\x, 5-1/\x);
\draw[domain=1/(5+3/5):1/5, black] plot (\x, 1/\x-5);
\draw[domain=1/6:1/(5+3/5), samples=200,  black] plot (\x, 6-1/\x);
\draw[domain=1/(6+3/5):1/6, black] plot (\x, 1/\x-6);
\draw[domain=1/7:1/(6+3/5), samples=200,  black] plot (\x, 7-1/\x);
\draw[domain=1/(7+3/5):1/7, black] plot (\x, 1/\x-7);
\draw[domain=1/8:1/(7+3/5), samples=200,  black] plot (\x, 8-1/\x);
\draw[domain=1/(8+3/5):1/8, black] plot (\x, 1/\x-8);
\draw[domain=1/9:1/(8+3/5), samples=200,  black] plot (\x, 9-1/\x);

\draw[black] (3/5,-0.01) -- (3/5,0.01);
\draw[black] (1/2,-0.01) -- (1/2,0.01);
\draw[black] (5/13,-0.01) -- (5/13,0.01);
\draw[black] (1/3,-0.01) -- (1/3,0.01);
\draw[black] (5/18,-0.01) -- (5/18,0.01);

\draw[black, dotted] (3/5,0) -- (3/5,3/5);
\draw[black, dotted] (1/2,0) -- (1/2,3/5);
\draw[black, dotted] (5/13,0) -- (5/13,3/5);
\draw[black, dotted] (5/18,0) -- (5/18,3/5);
\draw[black, dotted] (5/23,0) -- (5/23,3/5);

\draw[black,dotted] (0,1-3/5)--(3/5,1-3/5);
\draw[black,dotted] (0,1/2)--(3/5,1/2);
\draw[black,dotted] (0,3/5)--(3/5,3/5);

\node[black] at (3/5,-0.04) {$\alpha$};
\node[black] at (1/2,-0.04) {$\frac12$};
\node[black] at (5/13,-0.04) {$\frac{1}{2+\alpha}$};
\node[black] at (1/3,-0.04) {$\frac13$};
\node[black] at (5/18,-0.04) {$\frac{1}{3+\alpha}$};

\draw[black] (-0.01,1-3/5)--(0.01,1-3/5);
\draw[black] (-0.01,1/2)--(0.01,1/2);
\draw[black] (-0.01,3/5)--(0.01,3/5);

\node[black] at (-0.08,1-3/5) {$1-\alpha$};
\node[black] at (-0.04,1/2) {$\frac 12$};
\node[black] at (-0.04,3/5) {$\alpha$};

\end{tikzpicture}
\end{matrix}$
\caption{The graph of $A_{1/2}$ and $A_\alpha$ for $\alpha\in[\frac12,g]$, where $g=\frac{\sqrt{5}-1}{2}$.}
\label{fig:graphofA}
\end{figure}
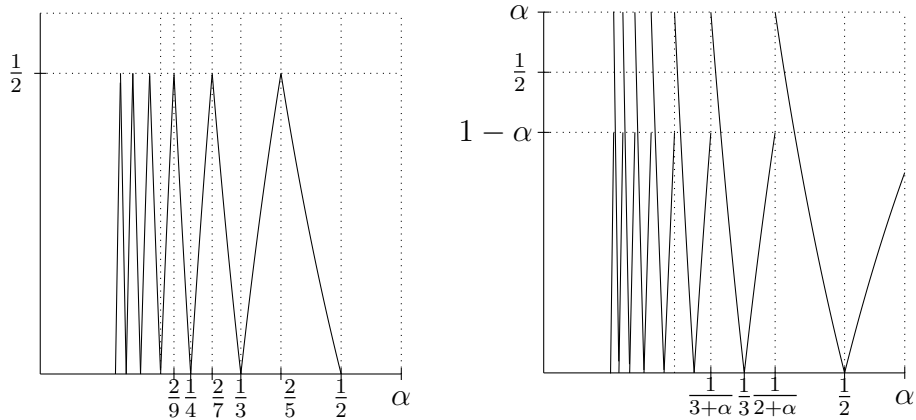

Let us denote by $(a_n^{(\alpha)},\epsilon_n^{(\alpha)})$ the $n$th partial quotient of $x$ given as a $\alpha$-continued fraction, $x_n^{(\alpha)} = A_{\alpha}^n(x)$, where $x_0^{(\alpha)} = x$ and 
$p_n^{(\alpha)}/q_n^{(\alpha)}$ its $n$th principal convergent.

We have $\epsilon_1^{(1/2)}\not=\epsilon_1^{(\alpha)}$ if and only if $x\in (\frac{1}{k+\alpha},\frac{1}{k+1/2})$ for some $k\ge 2$.
In this case, 
$$(a_1^{(1/2)},\epsilon_1^{(1/2)})=(k+1,-1),~ (a_1^{(\alpha)},\epsilon_1^{(\alpha)})=(k,1)~ \text{ and } 
~x_1^{(1/2)}=1-x_1^{(\alpha)}.$$

Let 
$$t_i = \cfrac{1}{3-\cfrac{1}{3-\cfrac{1}{\ddots -\cfrac{1}{3-\cfrac 12}}}}\quad\text{ and }\quad\frac{r_i}{s_i} = \cfrac{1}{3-\cfrac{1}{3-\cfrac{1}{\ddots -\cfrac{1}{3}}}},$$
where $3$ in the continued fraction expansion appears $i$ times.
Note that $r_i=s_{i-1}$ and $2-\frac{1}{1- t_i}=\frac{r_i}{s_i}$.
More precisely,
\begin{align*}
& \{t_i\}_{i\ge 0} = \left\{\frac12, \frac 25, \frac 5{13}, \frac{13}{34}, \frac{34}{89},\cdots\right\}\subset \left(1-g,\frac 25\right)\cup\left\{\frac12\right\},\\
& \left\{\frac{r_i}{s_i}\right\}_{i\ge0} = \left\{0,\frac{1}{3},\frac{3}{8},\frac{8}{21},\frac{21}{55},\cdots\right\}\subset\{0\}\cup\left[\frac 13,1-g\right).\end{align*}
Then, $A_{1/2}(t_i)=t_{i-1}$, $A_\alpha(\frac{r_i}{s_i})=A_{1/2}(\frac{r_i}{s_i})=\frac{r_{i-1}}{s_{i-1}}$  for $i\ge1$ and $t_i \searrow 1-g$ and $\frac {r_i}{s_i} \nearrow 1-g$.

Now, we suppose that $n$ is the minimal index such that
\begin{equation*}
x_1^{(1/2)}=x_1^{(\alpha)},~ x_2^{(1/2)}=x_2^{(\alpha)},~ \cdots, x_{n-1}^{(1/2)}=x_{n-1}^{(\alpha)},~\text{but} ~ x_n^{(1/2)}\not=x_n^{(\alpha)}.
\end{equation*}
Then,
\begin{align*}
& (a_{i}^{(1/2)},\epsilon_i^{(1/2)})=(a_i^{(\alpha)},\epsilon_i^{(\alpha)}) \text{ and } q_i^{(1/2)} = q_i^{(\alpha)} ~\text{for}~1\le i\le n-1, \\
& x_n^{(1/2)} = 1- x_n^{(\alpha)}, ~a_n^{(1/2)} = a_n^{(\alpha)} + 1 \ge 3,~ \epsilon_n^{(1/2)} = -1, ~ \epsilon_n^{(\alpha)} = 1 \text{ and } q_n^{(1/2)} = q_n^{(\alpha)} + q_{n-1}^{(1/2)}.
\end{align*}
Since $x_n^{(1/2)}\not= x_n^{(\alpha)}$, $x_n^{(1/2)}=1-x_n^{(\alpha)}$ and the domain of $A_{1/2}$ is $[0,\frac12]$, 
we have 
\begin{equation}\label{eq:x_n}
x_n^{(1/2)} \in \left[1-\alpha,\frac12\right)~\text{ and } ~ x_n^{(\alpha)} \in \left(\frac 12, \alpha\right].
\end{equation}
Since $[1-\alpha, \frac 12)\subset \bigsqcup_{i=1}^\infty (t_{i},t_{i-1}]$, there is $i\ge 1$ such that $x_n^{(1/2)}\in (t_{i},t_{i-1}]$.
This implies
\begin{align*}
& x_{n+j}^{(1/2)} = 3- \frac{1}{x_{n+j-1}^{(1/2)}} \in (t_{i-j},t_{i-j-1}]  \text{ and } (a_{n+j}^{(1/2)},\epsilon_{n+j}^{(1/2)}) = (3,-1) \text{ for } 1\le j \le i-1, \\
& x_{n+i}^{(1/2)} = \frac{1}{x_{n+i-1}^{(1/2)}} - 2  \text{ and } (a_{n+i}^{(1/2)},\epsilon_{n+i}^{(1/2)}) = (2,1).
\end{align*}
On the other hand,
$x_n^{(\alpha)}\in[1-t_{i-1},1-t_i)$, which implies
\begin{align*}
& x_{n+1}^{(\alpha)} = 2-\frac{1}{x_n^{(\alpha)}}\in \left(\frac{r_{i-1}}{s_{i-1}},\frac{r_{i}}{s_i}\right] \text{ and } (a_{n+1}^{(\alpha)},\epsilon_{n+1}^{(\alpha)}) = (2,-1), \\
& x_{n+j}^{(\alpha)} = 3-\frac{1}{x_{n+j-1}^{(\alpha)}}\in\left(\frac{r_{i-j}}{s_{i-j}},\frac{r_{i-j+1}}{s_{i-j+1}}\right] \text{ and } (a_{n+j}^{(\alpha)},\epsilon_{n+j}^{(\alpha)})=(3,-1) \text{ for } 2\le j \le i.
\end{align*}
Then,
$$q_{n+j}^{(1/2)}-q_{n+j}^{(\alpha)}
=q_{n+j-1}^{(1/2)},\text{ for }1\le j \le i-1,\quad\text{and }
 q_{n+i}^{(1/2)}=q_{n+i}^{(\alpha)}.
$$

In the following for 
$
g=\left(\begin{smallmatrix} a & b \\ c& d \end{smallmatrix} \right)\in \GL_2(\Z),
$
we denote by $g.z=\frac{az+b}{cz+d}$ the M\"obius transform applied on $z$.
With this notation we have
$$x_{n+i}^{(1/2)} =\begin{pmatrix}-2&1\\1&0\end{pmatrix}\begin{pmatrix}3&-1\\1&0\end{pmatrix}^{i-1}.x_{n}^{(1/2)}~\text{and}~ x_{n+i}^{(\alpha)} =\begin{pmatrix}3&-1\\1&0\end{pmatrix}^{i-1}\begin{pmatrix}2&-1\\1&0\end{pmatrix}.x_{n}^{(\alpha)}.$$
Since $\begin{pmatrix}3&-1\\1&0\end{pmatrix}\begin{pmatrix}-1&1\\-1&2\end{pmatrix}\begin{pmatrix}3&-1\\1&0\end{pmatrix}^{-1} = \begin{pmatrix}-1&1\\-1&2\end{pmatrix}$,
\begin{align*}
x_{n+i}^{(1/2)} 
& =\begin{pmatrix}-2&1\\1&0\end{pmatrix}\begin{pmatrix}3&-1\\1&0\end{pmatrix}^{i-1}
\begin{pmatrix}-1&1\\0&1\end{pmatrix}\left(\begin{pmatrix}3&-1\\1&0\end{pmatrix}^{i-1}\begin{pmatrix}2&-1\\1&0\end{pmatrix}\right)^{-1}.x_{n+i}^{(\alpha)}\\
& = \begin{pmatrix}-2&1\\1&0\end{pmatrix}\begin{pmatrix}-1&1\\-1&2\end{pmatrix}.x_{n+i}^{(\alpha)}
 = \begin{pmatrix}1&0\\-1&1\end{pmatrix}.x_{n+i}^{(\alpha)}. 
\end{align*}
It means that $\frac{1}{x_{n+i}^{(1/2)}} = \frac{1}{x_{n+i}^{(\alpha)}}-1$, thus, for $0\le c_1,c_2 \le 1$ and $k\in \mathbb{N}$ such that $k\ge 2$,
$$x_{n+i}^{(1/2)}\in \left(\frac{1}{k+c_1},\frac{1}{k+c_2}\right]~ \text{ if and only if }~x_{n+i}^{(\alpha)}\in \left(\frac{1}{k+1+c_1},\frac{1}{k+1+c_2}\right].$$
If $x_{n+i}^{(1/2)}\in \left(\frac{1}{k+1},\frac{1}{k+\alpha}\right]\cup\left(\frac{1}{k+1/2},\frac{1}{k+1}\right]$,
then
$$x_{n+i+1}^{(1/2)} = x_{n+i+1}^{(\alpha)}, ~ a_{n+i+1}^{(1/2)} = a_{n+i+1}^{(\alpha)} - 1,~\epsilon_{n+i+1}^{(1/2)} = \epsilon_{n+i+1}^{(\alpha)}~\text{ and } q_{n+i+1}^{(1/2)}=q_{n+i+1}^{(\alpha)}.$$
If $x_{n+i}^{(1/2)}\in \left(\frac{1}{k+\alpha},\frac{1}{k+1/2}\right]$,
then
\begin{align*}
& x_{n+i+1}^{(1/2)} = 1- x_{n+i+1}^{(\alpha)}, ~ a_{n+i+1}^{(1/2)} = a_{n+i+1}^{(\alpha)} ,~\epsilon_{n+i+1}^{(1/2)} = -1, ~\epsilon_{n+i+1}^{(\alpha)} = 1\text{ and }\\
& q_{n+i+1}^{(1/2)}-q_{n+i+1}^{(\alpha)}=q_{n+i}^{(1/2)}.
\end{align*}
Thus, by repeating the above process, we conclude that
$$q_{j}^{(1/2)}-q_{j}^{(\alpha)} = 0 \text{ or } q_{j-1}^{(1/2)}$$
and if $q_j^{(1/2)}-q_j^{(\alpha)}=q_{j-1}^{(1/2)}$, then $(a_{j+1}^{(1/2)}, \epsilon_{j+1}^{(1/2)})= (3,-1)$ or $(2,1)$.

Since $a_j^{(1/2)}\ge 2$ and $a_j^{(1/2)}=2$ implies $\epsilon_j^{(1/2)}=1$, we have
\begin{equation}\label{eq: qj1/2alpha}
 q_j^{(\alpha)} \ge q_j^{(1/2)}-q_{j-1}^{(1/2)} \ge q_{j-1}^{(1/2)}
\end{equation}
Hence,
\begin{equation}\label{eq: 1/q-1/q}
 \left|\frac{1}{q_j^{(1/2)}}-\frac{1}{q_j^{(\alpha)}}\right| =0\qquad\text{ or }\qquad\left|\frac{1}{q_j^{(1/2)}}-\frac{1}{q_j^{(\alpha)}}\right| = \frac{q_{j-1}^{(1/2)}}{q_j^{(1/2)}q_j^{(\alpha)}}\le \frac{1}{q_{j}^{(1/2)}}\le \frac{4}{q_{j+1}^{(1/2)}}.
\end{equation}
On the other hand, \eqref{eq: qj1/2alpha} also implies
$$\left|\log q_{j+1}^{(1/2)}-\log q_{j+1}^{(\alpha)}\right| = \log\left(1+\frac{q_{j+1}^{(1/2)}-q_{j+1}^{(\alpha)}}{q_{j+1}^{(\alpha)}}\right)\le\log\left(1+\frac{q_{j}^{(1/2)}}{q_{j+1}^{(\alpha)}}\right)\le \log 2.$$
Combining this with \eqref{eq: 1/q-1/q} and Remark~\ref{rmk:rmk1.7} yields
\begin{align*}
\MoveEqLeft \left|\sum_{j=0}^\infty (-1)^j\frac{\log q_{j+1}^{(1/2)}}{q_j^{(1/2)}}-\sum_{j=0}^\infty (-1)^j\frac{\log q_{j+1}^{(\alpha)}}{q_{j}^{(\alpha)}}\right| \\
& \le \sum_{j=0}^\infty \left|\frac{\log q_{j+1}^{(1/2)}}{q_j^{(1/2)}}-\frac{\log q_{j+1}^{(\alpha)}}{q_{j}^{(1/2)}} \right|+\sum_{j=0}^\infty \left | \frac{\log q_{j+1}^{(\alpha)}}{q_{j}^{(1/2)}} - \frac{\log q_{j+1}^{(\alpha)}}{q_{j}^{(\alpha)}}\right| \\
&  \le \sum_{j=0}^\infty \frac{\log 2}{q_j^{(1/2)}}+\sum_{j=0}^\infty  \frac{4\log q_{j+1}^{(1/2)}}{q_{j+1}^{(1/2)}} 
\le 4c_1+c_2.
\end{align*}

For $x\in(\alpha,1]$, the values of $W_{1/2}(x)$ and $W_\alpha(x)$ are defined symmetrically as
$$W_{1/2}(x) = W_{1/2}(1-x) \text{ and } W_\alpha(x)=W_\alpha(1-x).$$
By Remark~\ref{rmk: Wilton cond}, $W_{1/2}-W_\alpha$ is bounded on $[\alpha,1]$.

For $x\in(1/2,\alpha]$, $1-x\in[1-\alpha,1/2)$. 
Since $W_{1/2}(x) = W_{1/2}(1-x)$, we can consider that $n=0$ as in \eqref{eq:x_n}.
\end{proof}

\begin{proof}[Proof of Theorem \ref{prop: Wilton BMO}]
As we already observed  $1-S^{(1/2)}$ is invertible in $X_\ast$, which together with the fact that the 
 $\Z$--periodic even function equal to $-\log x$ on $(0,1/2]$ is in $X_\ast$ implies that $W_{1/2}$ is $\textnormal{BMO}$.
By Proposition~\ref{prop: unif bdd W_1/2 - W_alpha}, $W_\alpha$ is $\textnormal{BMO}$ for $\alpha\in[1/2,g]$.
\end{proof}

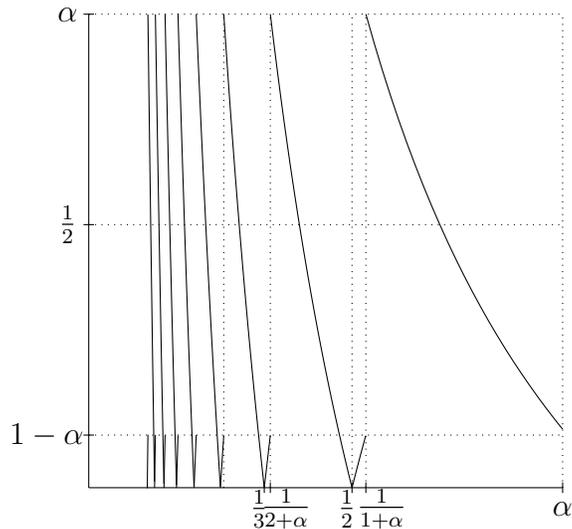
\begin{figure}[t]
\begin{tikzpicture}[scale=7]
\draw[black] (0,0) -- (9/10,0);
\draw[black] (0,0) -- (0,9/10);
\draw[domain=1/(1+9/10):9/10,  black] plot (\x, 1/\x-1);
\draw[domain=1/2:1/(1+9/10), black] plot (\x, 2-1/\x);
\draw[domain=1/(2+9/10):1/2,  black] plot (\x, 1/\x-2);
\draw[domain=1/3:1/(2+9/10), samples=200,  black] plot (\x, 3-1/\x);
\draw[domain=1/(3+9/10):1/3, black] plot (\x, 1/\x-3);
\draw[domain=1/4:1/(3+9/10), samples=200,  black] plot (\x, 4-1/\x);
\draw[domain=1/(4+9/10):1/4, black] plot (\x, 1/\x-4);
\draw[domain=1/5:1/(4+9/10), samples=200,  black] plot (\x, 5-1/\x);
\draw[domain=1/(5+9/10):1/5, black] plot (\x, 1/\x-5);
\draw[domain=1/6:1/(5+9/10), samples=200,  black] plot (\x, 6-1/\x);
\draw[domain=1/(6+9/10):1/6, black] plot (\x, 1/\x-6);
\draw[domain=1/7:1/(6+9/10), samples=200,  black] plot (\x, 7-1/\x);
\draw[domain=1/(7+9/10):1/7, black] plot (\x, 1/\x-7);
\draw[domain=1/8:1/(7+9/10), samples=200,  black] plot (\x, 8-1/\x);
\draw[domain=1/(8+9/10):1/8, black] plot (\x, 1/\x-8);
\draw[domain=1/9:1/(8+9/10), samples=200,  black] plot (\x, 9-1/\x);

\draw[black] (9/10,-0.01) -- (9/10,0.01);
\draw[black] (1/2,-0.01) -- (1/2,0.01);
\draw[black] (10/19,-0.01) -- (10/19,0.01);
\draw[black] (1/3,-0.01) -- (1/3,0.01);
\draw[black] (10/29,-0.01) -- (10/29,0.01);

\draw[black, dotted] (9/10,0) -- (9/10,9/10);
\draw[black, dotted] (1/2,0) -- (1/2,9/10);
\draw[black, dotted] (10/19,0) -- (10/19,9/10);
\draw[black, dotted] (10/29,0) -- (10/29,9/10);
\draw[black, dotted] (10/39,0) -- (10/39,9/10);

\draw[black,dotted] (0,1-9/10)--(9/10,1-9/10);
\draw[black,dotted] (0,1/2)--(9/10,1/2);
\draw[black,dotted] (0,9/10)--(9/10,9/10);

\node[black] at (9/10,-0.04) {$\alpha$};
\node[black] at (1/2-0.01,-0.04) {$\frac12$};
\node[black] at (10/19+0.03,-0.04) {$\frac{1}{1+\alpha}$};
\node[black] at (1/3-0.01,-0.04) {$\frac13$};
\node[black] at (10/29+0.03,-0.04) {$\frac{1}{2+\alpha}$};

\draw[black] (-0.01,1-9/10)--(0.01,1-9/10);
\draw[black] (-0.01,1/2)--(0.01,1/2);
\draw[black] (-0.01,9/10)--(0.01,9/10);

\node[black] at (-0.08,1-9/10) {$1-\alpha$};
\node[black] at (-0.04,1/2) {$\frac 12$};
\node[black] at (-0.04,9/10) {$\alpha$};

\end{tikzpicture}
\caption{Graph of $A_\alpha$ when $\alpha>g$.}\label{fig: alpha>g}
\end{figure}

Finally, we give a remark which difficulties occur if one wants to extend the results of Proposition \ref{prop: Wilton BMO} to $\alpha\in (g,1)$.
\begin{rmk}\label{rmk: Wilton alpha large}
 For the case $\alpha>g$, we cannot directly apply the same argument as in the proof of Proposition~\ref{prop: unif bdd W_1/2 - W_alpha}.
  If $\alpha>g$, then $A_\alpha$ has a branch which is defined by $1/x-1$ (see Figure~\ref{fig: alpha>g} for the graph of $A_\alpha$) contrary to the case of $\alpha\le g$.
 It causes a different behaviour of the orbits of the points under $A_\alpha$.
 In the proof, we showed a relation between $x_n^{(1/2)}$ and $x_n^{(\alpha)}$.
 By following the same argument for $\alpha>g$, we only obtain a relation between $x_n^{(1/2)}$ and $x_{n+N}^{(\alpha)}$, where $N$ depends on the number of consecutive points in the orbit of $x$ visiting $[\frac{1}{1+\alpha},\alpha]$.
\end{rmk}

\section{Behaviour of the truncated real Brjuno function and the truncated real Wilton function}\label{sec:kBrjuno real}
For $x\in\R$, recall that $\beta_{-1}=1$ and 
$$\beta_j(x)= \{x\}\, \Gauss(\{x\})\cdots \Gauss^j(\{x\}) = |p_j(\{x\})-q_j(\{x\}) \{x\}|\quad \text{ for }j\ge 0,$$
where $\Gauss$ is the Gauss map and $p_j(\{x\})/q_j(\{x\})$ is the $j$th principal convergent of $\{x\}$ with respect to the regular continued fraction algorithm.
Here, contrary to the previous section, we omit the $\alpha$ in $\beta^{(\alpha)}_j$ as we will always assume $\alpha$ to be one.

In this section, we are interested to compare a finite $k$-Brjuno sum or finite Wilton sum with the $k$-Brjuno or Wilton of its principal convergent. For doing so we first have to define the finite $k$-Brjuno or finite Wilton function respectively for a rational number. 

Within this section, let $p_j/q_j$ be a rational number whose continued fraction algorithm terminates after $r$ steps, i.e.\ it can be written as
$$p_j/q_j = m_0+\cfrac{1}{m_1+\cfrac{1}{m_2+\ddots+\cfrac{1}{m_r}}}$$ 
with $m_r\ge 2$ when $q_{j}>1$. (Of course this can correspond to the $r$th principal convergent of a number whose continued fraction expansion starts with $[m_0;m_1,\ldots, m_r,\ldots]$.)
With this we can define the truncated real Brjuno function by
$$B_{k,\mathrm{finite}}(p/q)=\sum_{j=0}^{r-1}\left(\beta_{j-1}(p/q-m_0)\right)^k\log \left(\frac{1}{\Gauss^j(p/q-m_0)}\right)$$
and the truncated real Wilton function by 
$$W_{\mathrm{finite}}(p/q)=\sum_{j=0}^{r-1}(-1)^j\beta_{j-1}(p/q-m_0)\log\left(\frac{1}{\Gauss^j(p/q- m_0)}\right).$$

Before stating the results of this section, we also introduce the notation $x_j=\Gauss^j(x)$ for $x\in (0,1)$. 
This enables us to state the next lemma which is an analog to \cite[Lem.~5.20]{MMY3}.
\begin{lem} \label{lem:lem5.20}
For each $k\in\N$ there exists $C_k>0$ such that for all $x\in (0,1)$ and $r\in \N$, we have
$$\left|B_{k,\mathrm{finite}}\left(\frac{p_r(x)}{q_r(x)}\right)-\sum_{j=0}^{r-1}\left(\beta_{j-1}(x)\right)^k\log\frac{1}{x_j}\right|\le C_k x_r (q_r(x))^{-1}$$
and
$$\left|W_{\mathrm{finite}}\left(\frac{p_r(x)}{q_r(x)}\right)-\sum_{j=0}^{r-1}(-1)^j\beta_{j-1}(x)\log\frac{1}{x_j}\right|\le C_1 x_r (q_r(x))^{-1}.$$
\end{lem}

\begin{proof}
To ease notation we write in the following $p_r$ and $q_r$ instead of $p_r(x)$ and $q_r(x)$ when the dependence on $x$ is clear.
If $r=1$, then 
\begin{align*}
 |B_{k,\mathrm{finite}}(p_1/q_1)-\log(1/x)|
 & =|W_{\mathrm{finite}}(p_1/q_1)-\log(1/x)|
 =| \log(q_1/p_1)-\log(1/x)| \\
 &=| \log(m_1)-\log(m_1+x_1)|
 = \log (1+ x_1/m_1) \le x_1/m_1.
\end{align*}

Suppose that $r>1$, then we have
\begin{align}
\MoveEqLeft \nonumber  \left|\sum_{j=0}^{r-1}\left[\left(\beta_{j-1}(p_r/q_r)\right)^k\log\frac{1}{\Gauss^j(p_r/q_r)} - \left(\beta_{j-1}(x)\right)^k\log\frac{1}{x_j}\right]\right| \\
&  \le \sum_{j=0}^{r-1}\left|\left(\beta_{j-1}(p_r/q_r)\right)^k\left[\log\frac{1}{\Gauss^j(p_r/q_r)}- \log \frac{1}{x_j}\right]\right| \nonumber\\
&\qquad + \sum_{j=0}^{r-1}\left|\left[\left(\beta_{j-1}(p_r/q_r)\right)^k- \left(\beta_{j-1}(x)\right)^k\right]\log\frac{1}{x_j}\right|\label{eq:finite}
\end{align}
and similarly for the Wilton case if $k=1$.

Note that 
\begin{equation}\label{eq:beta=qx-p}
\beta_{j-1}(p_r/q_r) = \left|q_{j-1}\frac{p_r}{q_r}-p_{j-1}\right|
\end{equation}
for $r\le j-1$.
Since $x_j,~\Gauss^j(p_r/q_r)\in\left[\frac{1}{m_{j+1}+1},\frac{1}{m_{j+1}}\right]$, we have 
$2^{-1} \le \frac{x_j}{\Gauss^j(p_r/q_r)}\le 2.$
By using \eqref{eq:beta=qx-p} and the fact that 
$ \log(y_2/y_1)  
\le (y_2-y_1)y_1^{-1}$ for $y_1<y_2\in\R_{>0}$, 
we have
\begin{align*}
 \left|\log\frac{1}{\Gauss^j(p_r/q_r)} - \log \frac{1}{x_j}\right| 
 &\le \max\left\{\Gauss^j(p_r/q_r),x_j\right\}\left|\frac{1}{\Gauss^j(p_r/q_r)}-\frac{1}{x_j}\right| \\
& \le 2\, \Gauss^j(p_r/q_r)\, \frac{|x-p_r/q_r|}{|q_{j}\,\frac{p_r}{q_r}-p_{j}|\,|q_{j}x-p_{j}|} 
 \le \frac{2|x-p_r/q_r|}{|q_{j-1}\,\frac{p_r}{q_r}-p_{j-1}|\,|q_{j}x-p_{j}|}.
\end{align*}

Furthermore, we note
\begin{equation}\label{eq:z_i1}
x = \frac{p_{i-1}x_i+p_{i}}{q_{i-1}x_i+q_{i}} \quad \text{ and } \quad x_i = -  \frac{q_{i}x - p_{i}}{q_{i-1}x-p_{i-1}}
\end{equation}
which imply
$\prod_{i=0}^\ell(-x_i) = q_\ell x - p_\ell$
and
\begin{equation}\label{eq:dirichlet1}
|p_\ell-q_\ell x| = \Big|p_\ell - q_\ell \frac{p_{\ell}x_{\ell+1}+p_{\ell+1}}{q_{\ell}x_{\ell+1}+q_{\ell+1}}\Big| 
= \frac{|p_\ell q_{\ell+1}- q_{\ell}p_{\ell+1}|}{|q_{\ell}x_{\ell+1}+q_{\ell+1}|}< \frac 1{q_{\ell+1}}.
\end{equation}

Hence, for the first summand of \eqref{eq:finite}, 
from \eqref{eq:beta=qx-p}, \eqref{eq:z_i1} and \eqref{eq:dirichlet1}, we have
\begin{align*}
\MoveEqLeft \sum_{j=0}^{r-1}\left|\left(q_{j-1}\frac{p_r}{q_r}-p_{j-1}\right)^k\left[\log\frac{1}{\Gauss^j(p_r/q_r)} - \log \frac{1}{x_j}\right]\right|\\ 
& \le \frac{2}{q_r}\sum_{j=0}^{r-1}\left|q_{j-1}\frac{p_r}{q_r}-p_{j-1}\right|^{k-1}\left|\frac{p_r-xq_r}{q_{j}x-p_{j}}\right| 
 \le \frac{2x_r}{q_r}\sum_{j=0}^{r-1}x_{j+1}\cdots x_{r-1} \\
&   \le \frac{2x_r}{q_r}\sum_{j=0}^{r-1} \left(\frac{\sqrt{5}-1}{2}\right)^{r-j-1} < 2C' x_r q_r^{-1}, \quad\quad \text{ (by Proposition 1.4-(iv) in \cite{MMY1})},
\end{align*}
where $C' = \sum_{j=0}^{\infty}\left(\frac{\sqrt{5}-1}{2}\right)^j=\frac{\sqrt{5}+3}{2}$.

On the other hand, by setting $X_{j,r}=q_{j-1}\,\frac{p_r}{q_r}-p_{j-1}$ and $Y_{j,r}= q_{j-1}x-p_{j-1}$, noting that $|X_{j,r}|\le 1/q_{j-1}$ and $|Y_{j,r}|\le 1/q_{j-1}$
we obtain for
the second term of \eqref{eq:finite} that
\begin{align}
\MoveEqLeft \sum_{j=0}^{r-1}\left|\left(q_{j-1}\,\frac{p_r}{q_r}-p_{j-1}\right)^k- \left(q_{j-1}x-p_{j-1}\right)^k\right|\log\frac{1}{x_j} \\
&= \nonumber \sum_{j=0}^{r-1}\left|X_{j,r}-Y_{j,r}\right|\left|X_{j,r}^{k-1}+X_{j,r}^{k-2}Y+\cdots +X_{j,r}Y_{j,r}^{k-2}+Y_{j,r}^{k-1}\right|\log\frac{1}{x_j} \\
& \le  \nonumber \sum_{j=0}^{r-1} \frac{k}{q_{j-1}^{k-2}}|p_r/q_r-x|\log\frac{1}{x_j} \\
& \le   \sum_{j=0}^{r-1} k q_{j-1}|p_r/q_r-x|\log\frac{1}{x_j}. \label{eq:finite2}
\end{align}

Since $\log(y)<y$ and using \eqref{eq:dirichlet1} and \eqref{eq:z_i1}, 
the value in \eqref{eq:finite2} is bounded above by
$$ \frac{k}{q_r} \sum_{j=0}^{r-1} \frac{|p_r-xq_r|}{|q_{j-1}x-p_{j-1}|}\,\frac{1}{x_j} =  \frac{k}{q_r} \sum_{j=0}^{r-1} x_{j+1}\cdots x_{r-1}x_{r} \le kC'x_rq_r^{-1}.$$
By letting $C_k := 2kC'$, we complete the proof.
\end{proof}

\section{Complex continued fractions}\label{sec:compl CF}
We consider a continued fraction on a compact subset of $\C$ which is a complex analog to the regular continued fractions.
At the beginning of this section we will first define some domains which will be important to define our complex continued fraction algorithm which we introduce in the sequel.

Similarly as in \cite{MMY3}, we consider the following sets:
\begin{align}
D_0&=\left\{z\in\C \; | \; |z+1| \le 1, \re(z) \geq \frac{\sqrt{3}}{2}-1\right\}&\textnormal{(Figure } {\text{\subref{fig:D0}}} \textnormal{),} \label{eq: def D0}\\
D_1&=\left\{z\in\C \; | \; |z| \geq 1, \left|z-\frac{\sqrt{3}}{3}\right| \le \frac{\sqrt{3}}{3}\right\}&\textnormal{(Figure }{\text{\subref{fig:D0}}}\textnormal{),} \label{eq: def D1}\\
D&=\left\{z\in\C \; | \; |z| \le 1, |z-i| \geq 1, |z+i| \geq 1, \re(z) > 0\right\}&\textnormal{(Figure }{\text{\subref{fig:D}}} \text{ and }\ref{fig:D partition}\textnormal{),} \label{eq: def D}
\end{align}

\renewcommand\thesubfigure{\arabic{figure}\alph{subfigure}} 
\captionsetup[subfigure]{labelfont = rm, textfont = normalfont, singlelinecheck = on} 
 
\begin{figure}[H]
\begin{subfigure}{0.49\textwidth}
\centering
\begin{tikzpicture}[scale=2.6]
\clip (-1.1,-1.2) rectangle (1.3,1.2);
  \draw[->] (-1.3,0) -- (1.3,0) ;
  \draw[->] (0,-1.2) -- (0,1.2) ;
  \node at (0,0) {\small$\bullet$}; \node at (0.07,-0.08) {$0$};
  \node at (-1,0) {\small$\bullet$}; \node at (-1,-0.12) {$-1$} ;
  \node at (1,0) {\small$\bullet$}; \node at (1-0.05,-0.11) {$1$} ;
  \node at (-0.13,0) {\small$\bullet$};  \node at (-0.3, -0.12)  {$\frac{\sqrt{3}-2}{2}$} ;
  \node at (0.577,0) {\small$\bullet$}; \node at (0.577,-0.12) {$\frac{\sqrt{3}}{3}$} ;
  \node at (1.154,0) {\small$\bullet$}; \node at (1.154,-0.12) {$\frac{2\sqrt{3}}{3}$};
  \node at (0,-1) {\small$\bullet$}; \node at (0.1,-1.1) {$-i$} ;
  \node at (0,1) {\small$\bullet$}; \node at (0.1,1.1) {$i$} ;
  
  \draw (-0.13, -1.5) -- (-0.13, 1.5);
  \draw (-2.1,1/2) -- (2.1,1/2);
  \draw (-2.1,-1/2) -- (2.1,-1/2);

  \draw  (0,0) circle [radius=1];
  \draw  (0.577,0) circle [radius=0.577];


  \draw  (-1, 0) circle [radius=1];

 \filldraw[opacity=0.2]  (0.87,0.5) arc[radius = 1, start angle = 30, end angle =-30]	
			arc[radius = 0.57, start angle = -60, end angle =60];

 \filldraw[opacity=0.2]  (-0.13,0.5) arc[radius = 1, start angle = 30, end angle =-30];

 \node(D0) at (-0.36,0.31) {$D_0$}; \draw[->] (-0.25,0.25)--(-0.05,0.05);
 \node(D1) at (0.7,0.3) {$D_1$}; \draw[->] (0.8,0.25)--(1+0.05,0.05);

\end{tikzpicture}
\caption{$D_0$ and $D_1$.}
\label{fig:D0}
\end{subfigure}
\begin{subfigure}{0.5\textwidth}
\centering
\begin{tikzpicture}[scale=2.6]
\clip (-1.1,-1.2) rectangle (1.3,1.2);
  \draw[->] (-1.3,0) -- (1.3,0) ;
  \draw[->] (0,-1.2) -- (0,1.2) ;
  \node at (0,0) {\small$\bullet$}; \node at (0.07,-0.08) {$0$};
  \node at (-1,0) {\small$\bullet$}; \node at (-1,-0.12) {$-1$} ;
  \node at (1,0) {\small$\bullet$}; \node at (1-0.05,-0.11) {$1$} ;
  \node at (0,-1) {\small$\bullet$}; \node at (0.1,-1.1) {$-i$} ;
  \node at (0,1) {\small$\bullet$}; \node at (0.1,1.1) {$i$} ;

  \draw  (0,0) circle [radius=1];
  \draw (0,1) circle [radius=1];
  \draw (0,-1) circle [radius=1];

 \filldraw[opacity=0.2]  (0,0) arc[radius = 1, start angle = 90, end angle =30]
		  arc[radius =1, start angle = -30, end angle = 30]
		  arc[radius = 1, start angle = -30, end angle = -90];
		  
 \node[fill=black!20] at (0.65,0) {$D$};
 
\end{tikzpicture}
\caption{$D$}
\label{fig:D}
\end{subfigure}
\end{figure}

\begin{align}
H_0&=\left\{z\in\C \; | \; |z-i| \le 1, |z+1| \geq 1, |\im(z)| \le \frac{1}{2}\right\}&\textnormal{(Figure }{\text{\subref{fig:H0Comp}}}\textnormal{),} \label{eq: def H0}\\
H_0'&= \left\{z\in\C \; | \; |z+i| \le 1, |z+1| \geq 1, |\im(z)| \le \frac{1}{2}\right\}&\textnormal{(Figure }{\text{\subref{fig:H0Comp}}}\textnormal{),} \label{eq: def H0'}\\
H & = H_0 \cup H_0' &\textnormal{(Figure }{\text{\subref{fig:H0Comp}}}\textnormal{),}  \label{eq: def H2} \\
\Delta&= D\cup H_0\cup H_0'=\left\{z\in\C \; | \; |z| \le 1, |z+1| \geq 1, |\im(z)| \le \frac{1}{2}\right\}&\textnormal{(Figure }{\text{\subref{fig:Delta}}}\textnormal{),} \label{eq: def Delta}
\end{align}

\begin{figure}[H]
\begin{subfigure}{0.45\textwidth}
\centering
\begin{tikzpicture}[scale=2.6]
\clip (-1.1,-1.2) rectangle (1.3,1.2);
  \draw[->] (-1.3,0) -- (1.3,0) ;
  \draw[->] (0,-1.2) -- (0,1.2) ;
  \node at (0,0) {\small$\bullet$}; \node at (0.07,-0.08) {$0$};
  \node at (-1,0) {\small$\bullet$}; \node at (-1,-0.12) {$-1$} ;
  \node at (1,0) {\small$\bullet$}; \node at (1-0.05,-0.11) {$1$} ;
  \node at (-0.13,0) {\small$\bullet$};  \node at (-0.3, -0.15)  {$\frac{\sqrt{3}-2}{2}$} ;
  \node at (0.57,0) {\small$\bullet$}; \node at (0.57,-0.12) {$\frac{\sqrt{3}}{3}$} ;
  \node at (0,-1) {\small$\bullet$}; \node at (0.1,-1.1) {$-i$} ;
  \node at (0,1) {\small$\bullet$}; \node at (0.1,1.1) {$i$} ;
  
  \draw (-0.13, -2.1) -- (-0.13, 2.1);
  \draw (-2.1,1/2) -- (2.1,1/2);
  \draw (-2.1,-1/2) -- (2.1,-1/2);

  \draw  (0,0) circle [radius=1];

  \draw (0,-1) circle [radius=1];
  \draw (0,1) circle [radius=1];
  \draw  (-1, 0) circle [radius=1];

 \filldraw[opacity=0.2]  (0.87,0.5) arc[radius = 1, start angle = -30, end angle =-90]	
			arc[radius = 1, start angle = 0, end angle =30];
 \filldraw[opacity=0.2]  (0.87,-0.5) arc[radius = 1, start angle = 30, end angle =90]	
			arc[radius = 1, start angle = 0, end angle =-30];

 \node[fill=black!20] at (0.3,0.27) {$H_0$};
 \node[fill=black!20] at (0.3,-0.27) {$H_0'$};

\end{tikzpicture}
\caption{$H_0$, $H_0'$ and $H = H_0\cup H_0'$.}
\label{fig:H0Comp}
\end{subfigure}
\begin{subfigure}{0.45\textwidth}
\centering
\begin{tikzpicture}[scale=2.6]
\clip (-1.1,-1.2) rectangle (1.3,1.2);
  \draw[->] (-1.3,0) -- (1.3,0) ;
  \draw[->] (0,-1.2) -- (0,1.2) ;
  \node at (0,0) {\small$\bullet$}; \node at (0.07,-0.08) {$0$};
  \node at (-1,0) {\small$\bullet$}; \node at (-1,-0.12) {$-1$} ;
  \node at (1,0) {\small$\bullet$}; \node at (1-0.05,-0.11) {$1$} ;
  \node at (0,-1) {\small$\bullet$}; \node at (0.1,-1.1) {$-i$} ;
  \node at (0,1) {\small$\bullet$}; \node at (0.1,1.1) {$i$} ;

  \draw (-2.5, 0.5) -- (1.5, 0.5);
  \draw (-2.5, -0.5) -- (1.5, -0.5);

  \draw (0, 0) circle [radius=1];
  \draw (-1, 0) circle [radius=1];

 \filldraw[opacity=0.2]  (0.87,0.5) arc[radius = 1, start angle = 30, end angle =-30]	
			-- (-0.13, -0.5)
			arc[radius = 1, start angle = -30, end angle =30];

 \node[fill=black!20] at (0.5,0) {$\Delta$};

\end{tikzpicture}
\caption{$\Delta$.}
\label{fig:Delta}
\end{subfigure}
\end{figure}

\begin{align}
D_\infty = &  \overline{\C}\backslash(D_0\cup \Delta\cup D)\notag\\
= & \left\{z\in \overline{\C} \; | \; |\im(z)| > \frac{1}{2}\right\} \cup \left\{z\in \overline{\C} \; | \; \re(z) > \frac{\sqrt{3}}{2}-1\right\} \notag\\
&\cup \left\{z\in\overline{\C} \; | \; \re(z) > \frac{\sqrt{3}}{2},~ \left|z-\frac{\sqrt{3}}{3}\right|>\frac{\sqrt{3}}{3}\right\}&\textnormal{(Figure }{\text{\subref{fig:Dinf}}}\textnormal{).} \label{eq: def Dinfty}
\end{align}

\begin{figure}[H]
\begin{subfigure}{0.49\textwidth}
\centering
\begin{tikzpicture}[scale=2.6]
\clip (-0.7,-1) rectangle (1.5,1.1);
  \draw[->] (-0.7,0) -- (1.3,0) ;
  \draw[->] (0,-1.2) -- (0,1.2) ;
  \node at (0,0) {\small$\bullet$}; \node at (0.07,-0.08) {$0$};
  \node at (-0.13,0) {\small$\bullet$}; \node at (-0.3, -0.12)  {$\frac{\sqrt{3}-2}{2}$} ;
  \node at (0.577,0) {\small$\bullet$}; \node at (0.577,-0.12) {$\frac{\sqrt{3}}{3}$};
  \node at (1.154,0) {\small$\bullet$}; \node at (1.154+0.12,-0.12) {$\frac{2\sqrt{3}}{3}$};

  \draw[very thick] (-0.13,-0.5) -- (-0.13,0.5);
  \draw[very thick] (-0.13,-0.5) -- (0.866,-0.5);
  \draw[very thick] (-0.13,0.5) -- (0.866,0.5);
  \draw[very thick] (1.154,0) arc (0:60:0.577);
  \draw[very thick] (1.154,0) arc (0:-60:0.577);

  \draw (-0.7, -0.5) -- (1.5, -0.5);
  \draw (-0.7, 0.5) -- (1.5, 0.5);

  \draw (-0.13, -1.2) -- (-0.13, 1.2);

  \draw (0.87, -0.75) -- (0.87, 1.25);
  \draw (0.577,0) circle [radius=0.577];

 \fill[opacity=0.5, pattern=north west lines]  (0.87,0.5) arc[radius = 0.57, start angle = 60, end angle =-60]	
			-- (0.87, -1.2)
			-- (2, -1.2)
			-- (2, 1.25)
			-- (0.87, 1.25);
 \fill[opacity=0.5, pattern=north west lines]  (0.87, 1.25) -- (-0.7, 1.25) -- (-0.7, -1.2)
			-- (0.87, -1.2) -- (0.87, -0.5) 
			-- (-0.13, -0.5)
			-- (-0.13, 0.5)
			-- (0.87, 0.5)
			-- (0.87, 1.25);

 \node at (-0.4,0.7) {$D_\infty$};

\end{tikzpicture}
\caption{$D_\infty$}
\label{fig:Dinf}
\end{subfigure}
\begin{subfigure}{0.5\textwidth}
\centering
\begin{tikzpicture}[scale=0.9]
  \draw[->] (-7.7,0) -- (1.5,0) ;
  \draw[->] (0,-3) -- (0,3) ;
  \node at (0,0) {\small$\bullet$}; \node at (0.1,-0.15) {$0$};
 \node at (0,-1) {\small$\bullet$}; \node at (0.1,-1.2) {$-i$} ;
 \node at (0,1) {\small$\bullet$}; \node at (0.1,1.2) {$i$} ;
 \node at (0.866,1/2) {\small{$\bullet$}};
 \node at (0.866,-1/2) {\small{$\bullet$}};
 \node at (-1/2,1.866) {\small{$\bullet$}};
 \node at (-1/2,-1.866) {\small{$\bullet$}};
 \node at (-3.732,0) {\small{$\bullet$}};
 \node at (-3.832+0.2,-0.2) {\small$-\!(\!\sqrt{3}\!+\!2)$};
  \draw (-1/2, -3) -- (-1/2, 3);
  \draw (-7.7,1/2) -- (1.5,1/2);
  \draw (-7.7,-1/2) -- (1.5,-1/2);
  \draw (0.866,-1/2) -- (0.866,1/2);

  \draw  (0,0) circle [radius=1];

  \draw (0,1) circle [radius=1];
  \draw (0,-1) circle [radius=1];

  \draw (-3.732,0) circle (3.732);
  
  \draw[very thick] (-1/2,1.866) arc (30:330:3.732);
  \draw[very thick] (-1/2,1.866) arc (120:-30:1);
  \draw[very thick] (-1/2,-1.866) arc (-120:30:1);
  \draw[very thick] (0.866,-1/2) -- (0.866,1/2);

 \fill[opacity=0.2] (-1/2,-1.866) arc [radius = 1, start angle = -120, end angle =30] -- (0.866,-0.5) -- (0.866,0.5) arc [radius = 1, start angle = -30, end angle =120] -- (-1/2,1.866) arc (30:330:3.732) ;

  \node[fill=black!20] at (-3.7,1) {$\iota(D_\infty)$};

\end{tikzpicture}
\caption{$\iota(D_\infty)$ where $\iota(z)=1/z$.}\label{fig:iota Dinfty}
\end{subfigure}
\end{figure}
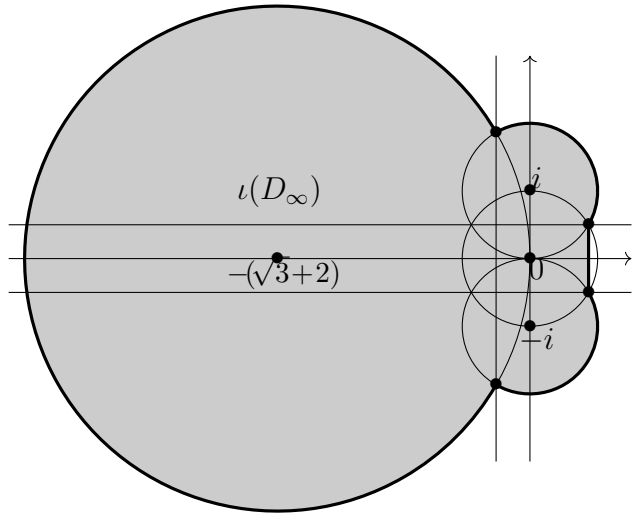
We define an extension of the continued fraction to the complex plane as follows. 
Let $z\in D$ as in Figure~\subref{fig:D}. Then $1/z \in \bigcup_{n\in\N} (n+\Delta)$, see Figure~\subref{fig:H0Comp}.
If $1/z \in n+\Delta$, then we take $m_1 = n$
and we set $z_1:= 1/z-m_1$. If $z_1 \in \Delta-D$, then we finish the process. 
If $z_1 \in D$, then we define $m_2$ by an integer $n$ such that $1/z_1 \in n+\Delta$.
Repeating this process, we obtain a continued fraction expansion $\{m_i\}_{i=1}^r$ such that
\begin{equation}\label{eq:CF}
\frac{1}{z_i}=m_{i+1}+z_{i+1} \quad \text{ for all }0< i\le r,
\end{equation}
such that $z_i\in D$ for $i<r$, and $z_r\in \Delta$. 
Let $D(m_1,\cdots, m_r)$ be the set of $z_0\in D$ whose first $r$ complex continued fraction entries equal $\{m_i\}_{i=1}^r$.

Note that this continued fraction algorithm does not coincide with Hurwitz' continued fractions, which are often used on the complex plane, but which are expanded with Gaussian integers \cite{Hu}, see also \cite{He}. The goal here is very different: our complex continued fraction expansion
follows as closely as possible the standard continued fraction of neighboring real points. For non-real numbers the iteration of our algorithm leads to an increasing sequence of imaginary parts and it stops after finitely many steps, since ``it is not possible anymore to associate a neighboring real point".
We will see in the following section why this continued fraction definition makes sense.

For $z\in\Delta$ and $i\ge 1$, let
\begin{equation}\label{eq:def eps}
 \varepsilon_i =
  \begin{cases}
  0 & \text{if } z\in D(m_1,\cdots, m_i) \text{ and } m_i = 1, \\
  1 & \text{otherwise}.
  \end{cases}
\end{equation}
We note here that the just defined $\varepsilon_i$ is independent to the definition of $\epsilon_i$ defined for the $\alpha$-continued fractions. However, since in the following sections we only consider the regular continued fraction algorithm with the Gauss map $\Gauss$ we do not expect too much confusion for the reader.

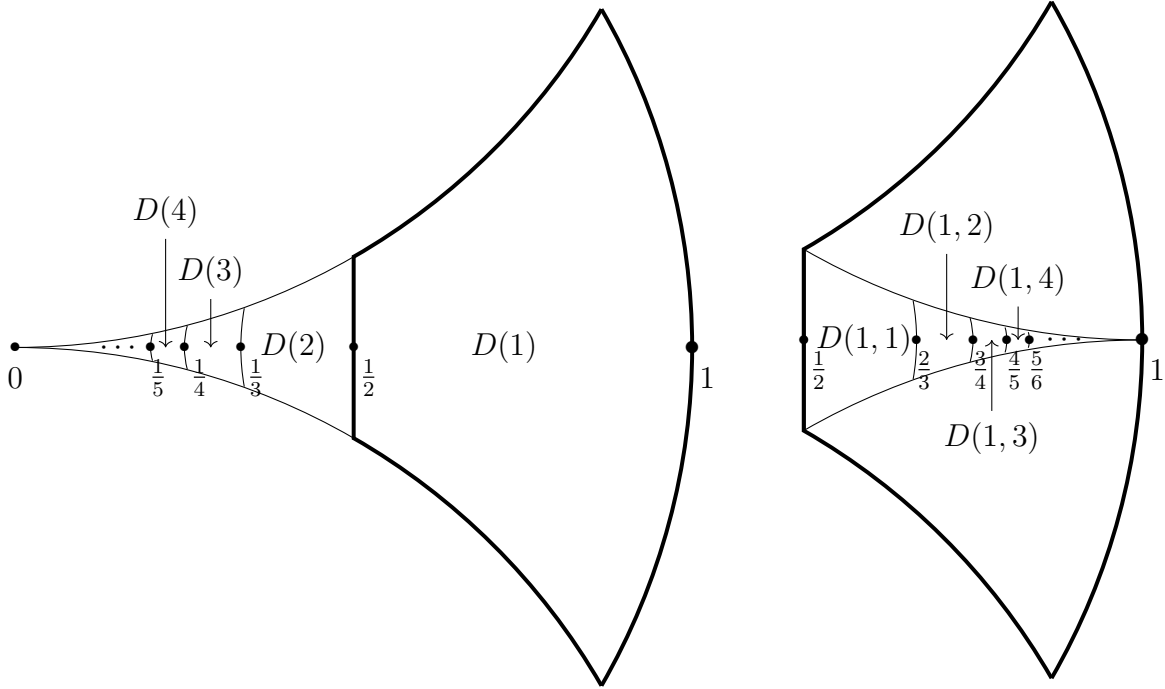
\begin{figure}
\begin{multicols}{2}
\begin{tikzpicture}[scale=1] 

\node at (0,0) {\tiny$\bullet$}; \node at (9*0,0-0.4){$0$};
\node at (9*1,0) {\tiny$\bullet$}; \node at (9*1+0.2,0-0.4){$1$};

\draw[line width=0.4pt] (9*1,9*0) arc (0:30:9*1);
\draw[line width=0.4pt] (9*1,9*0) arc (0:-30:9*1);
\draw[line width=0.4pt] (0,0) arc (-90:-30:9*1);
\draw[line width=0.4pt] (0,0) arc (90:30:9*1);

\node(d1) at (6.5,0) {$D(1)$}; 
\node(d2) at (3.7,0) {$D(2)$}; 
\node(d3) at (2.6,1) {$D(3)$}; \draw[->] (d3) -- (2.6,0);
\node(d4) at (2,1.8) {$D(4)$}; \draw[->] (d4) -- (2,0);

\draw[-, line width=1.5pt] (9*1/2, 9*0.136) -- (9*1/2, -9*0.136);
\node at (9*1/2,0){\tiny$\bullet$}; \node at (9*1/2+0.2,0-0.4){$\frac 12$};
\draw[line width=0.4pt] (9*1/3, 9*0) arc (180:180-10:9*1/2-9*1/3/2);
\draw[line width=0.4pt] (9*1/3, 9*0) arc (180:180+10:9*1/2-9*1/3/2);
\node at (9*1/3,0){\tiny$\bullet$}; \node at (9*1/3+0.2,0-0.4){$\frac 13$};
\draw[line width=0.4pt] (9*1/4, 9*0) arc (180:180-15:9*1/2/2-9*1/4/2);
\draw[line width=0.4pt] (9*1/4, 9*0) arc (180:180+15:9*1/2/2-9*1/4/2);
\node at (9*1/4,0){\tiny$\bullet$}; \node at (9*1/4+0.2,0-0.4){$\frac 14$};
\draw[line width=0.4pt] (9*1/5, 9*0) arc (180:180-17:9*1/3/2-9*1/5/2);
\draw[line width=0.4pt] (9*1/5, 9*0) arc (180:180+17:9*1/3/2-9*1/5/2);
\node at (9*1/5,0){\tiny$\bullet$}; \node at (9*1/5+0.1,0-0.4){$\frac 15$};
\node at (1.4,0) {$\cdots$};

\clip (9*1/2-0.01,-9*1/2-0.1) rectangle (9*1+0.1,9*1/2+0.1);

\node at (0,0) {$\bullet$}; \node at (9*0,0-0.4){$0$};
\node at (9*1,0) {$\bullet$}; \node at (9*1+0.2,0-0.4){$1$};

\draw[line width=1.5pt] (9*1,9*0) arc (0:30:9*1);
\draw[line width=1.5pt] (9*1,9*0) arc (0:-30:9*1);
\draw[line width=1.5pt] (0,0) arc (-90:-30:9*1);
\draw[line width=1.5pt] (0,0) arc (90:30:9*1);

\end{tikzpicture}

\hspace{9ex}
\begin{tikzpicture}[scale=1] 
\node(p1) at (9*1/2,0){\tiny$\bullet$}; 
\draw[-, line width=1.5pt] (9*1/2, 9*0.136) -- (9*1/2, -9*0.136);

\clip (9*1/2-0.01,-9*1/2-0.1) rectangle (9*1+0.5,9*1/2);

\node at (0,0) {$\bullet$}; \node at (9*0,0-0.4){$0$};
\node at (9*1,0) {$\bullet$}; \node at (9*1+0.2,0-0.4){$1$};

\draw[line width=1.5pt] (9*1,9*0) arc (0:30:9*1);
\draw[line width=1.5pt] (9*1,9*0) arc (0:-30:9*1);
\draw[line width=1.5pt] (0,0) arc (-90:-30:9*1);
\draw[line width=1.5pt] (0,0) arc (90:30:9*1);

\node(d11) at (5.3,0) {$D(1,1)$}; 
\node(d12) at (6.4,1.5) {$D(1,2)$};\draw[->] (d12) -- (6.4,0);
\node(d13) at (7,-1.3) {$D(1,3)$};\draw[->] (d13) -- (7,-0);
\node(d14) at (7.35,0.8) {$D(1,4)$};\draw[->] (d14) -- (7.35,0);
\node at (8,0) {$\cdots$};

\draw[-, line width=1.5pt] (9*1/2, 9*0.136) -- (9*1/2, -9*0.136);
\node at (9*1/2,0){\tiny$\bullet$}; \node at (9*1/2+0.2,0-0.4){$\frac 12$};

\draw[line width=0.4pt] (9*1/3, 9*0) arc (180:180-10:9*1/2-9*1/3/2);
\draw[line width=0.4pt] (9*1/3, 9*0) arc (180:180+10:9*1/2-9*1/3/2);
\node at (9*1/3,0){$\bullet$}; \node at (9*1/3+0.2,0-0.4){$\frac 13$};

\draw[line width=0.4pt] (9*1/4, 9*0) arc (180:180-15:9*1/2/2-9*1/4/2);
\draw[line width=0.4pt] (9*1/4, 9*0) arc (180:180+15:9*1/2/2-9*1/4/2);
\node at (9*1/4,0){$\bullet$}; \node at (9*1/4+0.2,0-0.4){$\frac 14$};

\draw[line width=0.4pt] (9*1,0) arc (-90:-120:9*1);
\draw[line width=0.4pt] (9*1,0) arc (90:120:9*1);

\draw[line width=0.4pt] (9*1/2,0) arc (-90:-133:9*0.226);
\draw[line width=0.4pt] (9*1/2,0) arc (90:133:9*0.226);

\draw[line width=0.4pt] (9*1/3,0) arc (-90:-133:9*0.111);
\draw[line width=0.4pt] (9*1/3,0) arc (90:133:9*0.111);

\draw[line width=0.4pt] (9*2/3,0) arc (0:10:9*1/3);
\draw[line width=0.4pt] (9*2/3,0) arc (0:-10:9*1/3);
\node at (9*2/3,0) {\tiny$\bullet$}; \node at (9*2/3+0.1,0-0.4){$\frac 23$};

\draw[line width=0.4pt] (9*3/4,0) arc (0:15:9*1/8);
\draw[line width=0.4pt] (9*3/4,0) arc (0:-15:9*1/8);
\node at (9*3/4,0){\tiny$\bullet$}; \node at (9*3/4+0.1,0-0.4){$\frac 34$};

\draw[line width=0.4pt] (9*4/5,0) arc (0:16:9*1/15);
\draw[line width=0.4pt] (9*4/5,0) arc (0:-16:9*1/15);
\node at (9*4/5,0){\tiny$\bullet$}; \node at (9*4/5+0.1,0-0.4){$\frac 45$};

\draw[line width=0.4pt] (9*5/6,0) arc (0:16:9*1/24);
\draw[line width=0.4pt] (9*5/6,0) arc (0:-16:9*1/24);
\node at (9*5/6,0){\tiny$\bullet$}; \node at (9*5/6+0.1,0-0.4){$\frac 56$};

\end{tikzpicture}
\end{multicols}
\caption{The sets $D(m_1,\cdots,m_r)$. The left figure is the partition of $D$ by $D(n)$. The right figure is of the sets $D(1,n)$ in $D(1)$.}\label{fig:D partition}
\end{figure}

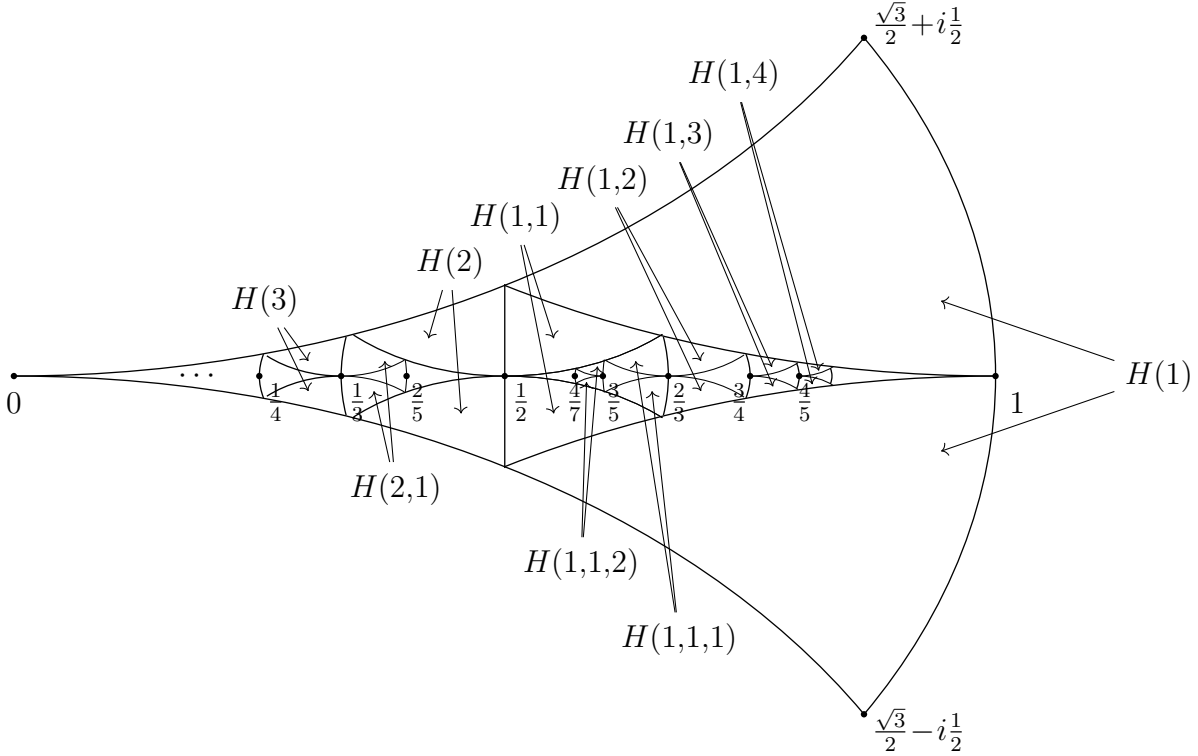
\begin{figure}
\centering
\begin{tikzpicture}[x=1.45cm,y=1cm]

\node at (0,0) {\huge$.$}; \node at (9*0,0-0.35){$0$};
\node at (9*1,0) {\huge$.$}; \node at (9*1+0.2,0-0.35){$1$};

\node at (9*0.866,9*1/2) {\huge$.$}; \node at (9*0.866+0.5,9*1/2+0.2) {$\frac{\sqrt{3}}{2}\!+\!i\frac{1}{2}$};
\node at (9*0.866,-9*1/2) {\huge$.$}; \node at (9*0.866+0.5,-9*1/2-0.2) {$\frac{\sqrt{3}}{2}\!-\!i\frac{1}{2}$};

\draw[line width=0.5pt] (9*1,9*0) arc (0:30:9*1);
\draw[line width=0.5pt] (9*1,9*0) arc (0:-30:9*1);
\draw[line width=0.5pt] (0,0) arc (-90:-30:9*1);
\draw[line width=0.5pt] (0,0) arc (90:30:9*1);

\node(d1) at (10.5,0) {$H(1)$}; \draw[-> ] (d1) -- (8.5,1);  \draw[-> ] (d1) -- (8.5,-1);
\node(d2) at (4,1.5) {$H(2)$}; \draw[-> ] (d2) -- (3.8,0.5); \draw[-> ] (d2) -- (4.1,-0.5);
\node(d3) at (2.3,1) {$H(3)$}; \draw[-> ] (d3) -- (2.7,0.2); \draw[-> ] (d3) -- (2.7,-0.2);
\node(d11) at (4.6,2.1) {$H(1,\!1)$};\draw[-> ] (d11) -- (4.95,0.5);\draw[-> ] (d11) -- (4.95,-0.5);
\node(d12) at (5.4,2.6) {$H(1,\!2)$};\draw[-> ] (d12) -- (6.3,0.2); \draw[-> ] (d12) -- (6.3,-0.2);
\node(d13) at (6,3.2) {$H(1,\!3)$};\draw[-> ] (d13) -- (6.95,0.12); \draw[-> ] (d13) -- (6.96,-0.14);
\node(d14) at (6.6,4) {$H(1,\!4)$};\draw[-> ] (d14) -- (7.38,0.07); \draw[-> ] (d14) -- (7.32,-0.11);
\node(d111) at (6.1,-3.5) {$H(1,\!1,\!1)$};\draw[-> ] (d111) -- (5.7,0.2); \draw[-> ] (d111) -- (5.85,-0.2);
\node(d112) at (5.2,-2.5) {$H(1,\!1,\!2)$};\draw[-> ] (d112) -- (5.35,0.13); \draw[-> ] (d112) -- (5.25,-0.08);
\node(d21) at (3.5,-1.5) {$H(2,\!1)$}; \draw[-> ] (d21) -- (3.4,0.2); \draw[-> ] (d21) -- (3.3,-0.2);

\draw[-, line width=0.5pt] (9*1/2, 9*0.136) -- (9*1/2, -9*0.136);
\node at (9*1/2,0){\huge$.$}; \node at (9*1/2+0.15,0-0.35){$\frac 12$};

\draw[line width=0.5pt] (9*1/3, 9*0) arc (180:180-10:9*1/2-9*1/3/2);
\draw[line width=0.5pt] (9*1/3, 9*0) arc (180:180+10:9*1/2-9*1/3/2);
\node at (9*1/3,0){\huge$.$}; \node at (9*1/3+0.15,0-0.35){$\frac 13$};

\draw[line width=0.5pt] (9*1/4, 9*0) arc (180:180-15:9*1/2/2-9*1/4/2);
\draw[line width=0.5pt] (9*1/4, 9*0) arc (180:180+15:9*1/2/2-9*1/4/2);
\node at (9*1/4,0){\huge$.$}; \node at (9*1/4+0.15,0-0.35){$\frac 14$};

\draw[line width=0.5pt] (9*1,0) arc (-90:-120:9*1);
\draw[line width=0.5pt] (9*1,0) arc (90:120:9*1);

\draw[line width=0.5pt] (9*1/2,0) arc (-90:-133:9*0.226);
\draw[line width=0.5pt] (9*1/2,0) arc (90:133:9*0.226);

\draw[line width=0.5pt] (9*1/3,0) arc (-90:-133:9*0.111);
\draw[line width=0.5pt] (9*1/3,0) arc (90:133:9*0.111);

\draw[line width=0.5pt] (9*2/3,0) arc (0:10:9*1/3);
\draw[line width=0.5pt] (9*2/3,0) arc (0:-10:9*1/3);
\node at (9*2/3,0) {\huge$.$}; \node at (9*2/3+0.1,0-0.35){$\frac 23$};

\draw[line width=0.5pt] (9*3/4,0) arc (0:15:9*1/8);
\draw[line width=0.5pt] (9*3/4,0) arc (0:-15:9*1/8);
\node at (9*3/4,0){\huge$.$}; \node at (9*3/4-0.1,0-0.35){$\frac 34$};

\draw[line width=0.5pt] (9*4/5,0) arc (0:16:9*1/15);
\draw[line width=0.5pt] (9*4/5,0) arc (0:-16:9*1/15);
\node at (9*4/5,0){\huge$.$}; \node at (9*4/5+0.05,0-0.35){$\frac 45$};

\draw[line width=0.5pt] (9*5/6,0) arc (0:16:9*1/24);
\draw[line width=0.5pt] (9*5/6,0) arc (0:-16:9*1/24);

\draw[black] (9*1/2,0) arc (-90:-48:9*0.24);
\draw[black] (9*1/2,0) arc (90:48:9*0.24);

\draw[black] (9*1/2,0) arc (-90:-48:9*0.24);
\draw[black] (9*1/2,0) arc (90:48:9*0.24);

\draw[black] (9*2/3,0) arc (-90:-46:9*0.111);
\draw[black] (9*2/3,0) arc (90:46:9*0.111);

\draw[black] (9*3/4,0) arc (-90:-46:9*0.07);
\draw[black] (9*3/4,0) arc (90:46:9*0.07);

\draw[black] (9*4/5,0) arc (-90:-46:9*0.05);
\draw[black] (9*4/5,0) arc (90:46:9*0.05);

\draw[black] (9*2/3,0) arc (-90:-130:9*0.1);
\draw[black] (9*2/3,0) arc (90:130:9*0.1);

\draw[black] (9*3/5,0) arc (180:180-7:9*1/5);
\draw[black] (9*3/5,0) arc (180:180+7:9*1/5);
\node at (9*3/5,0) {\huge$.$}; 
\node at (9*3/5+0.1,0-0.35){$\frac 35$};

\draw[black] (9*4/7,0) arc (180:180-15:9*1/21);
\draw[black] (9*4/7,0) arc (180:180+15:9*1/21);
\node at (9*4/7,0) {\huge$.$}; 
\node at (9*4/7,0-0.35){$\frac 47$};

\draw[black] (9*2/5,0) arc (0:-7:9*1/5);
\draw[black] (9*2/5,0) arc (0:+7:9*1/5);
\node at (9*2/5,0) {\huge$.$}; 
\node at (9*2/5+0.1,0-0.35){$\frac 25$};

\draw[black] (9*1/3,0) arc (-90:-50:9*0.1);
\draw[black] (9*1/3,0) arc (90:50:9*0.1);

\draw[black] (9*3/5,0) arc (-90:-135:9*0.04);
\draw[black] (9*3/5,0) arc (90:135:9*0.04);

\node at (1.7,0) {$\cdots$};
\end{tikzpicture}
\caption{The partition of $D$ by the sets $H(m_1,\cdots,m_r)$.}\label{fig:H partition}
\end{figure}

We define $p_\ell/q_\ell$ by
$$\frac{p_\ell}{q_\ell} = \cfrac{1}{m_1+\cfrac 1{m_2 + \cfrac 1{\ddots +\cfrac 1{m_\ell}}}}$$
and $q_0 = p_{-1}=q_{-2}=1$ and $p_0 = p_{-2}=q_{-1}=0$.
Then, 
\begin{equation}\label{eq:z_i}
z_0 = \frac{p_{i-1}z_i+p_{i}}{q_{i-1}z_i+q_{i}} \quad \text{ and } \quad z_i = -  \frac{q_{i}z_0 - p_{i}}{q_{i-1}z_0-p_{i-1}}.
\end{equation}
By \eqref{eq:z_i}, we can easily see that
\begin{equation}\label{eq:prod z_i}
\prod_{i=0}^\ell(-z_i) = q_\ell z_0 -p_\ell.
\end{equation}
We have 
\begin{equation}\label{eq:dirichlet2}
p_\ell-q_\ell z_0= p_\ell - q_\ell \frac{p_{\ell}z_{\ell+1}+p_{\ell+1}}{q_{\ell}z_{\ell+1}+q_{\ell+1}}
= \frac{p_\ell q_{\ell+1}- q_{\ell}p_{\ell+1}}{q_{\ell}z_{\ell+1}+q_{\ell+1}} = \frac{(-1)^{\ell+1}}{q_\ell z_{\ell+1}+q_{\ell+1}}.
\end{equation}
Thus, we have 
\begin{equation}\label{eq:dirichlet}|
p_\ell-q_\ell z_0| < \frac 1{q_{\ell+1}}.
\end{equation}
We set
\begin{equation}
H(m_1,\cdots,m_r) = D(m_1,\cdots,m_r)\setminus \mathrm{int}\left(\bigcup_{m_{r+1}\ge 1}D(m_1,\cdots,m_{r+1})\right)
\label{eq: def H}
\end{equation}
where it is defined by $H$ when $r=0$, see \eqref{eq: def H2} for the definition of $H$.
The sets $H(m_1,\cdots,m_r)$ give a partition of $D$ as in Figure~\ref{fig:H partition}.
Then we have 
\begin{align*}
 \big\{z:|\mathrm{Im}z|\le 1/2\big\} = \bigcup_{n\in\Z}\bigcup_{r\ge 0}\bigcup_{m_1,\cdots,m_r\ge 1}\big[H(m_1,\cdots,m_r)+n\big]\sqcup \R\backslash\Q,
\end{align*}
where the sets in the right-hand term have disjoint interiors.
A set $H(m_1,\cdots,m_r)+n$ meets $\R$ in a unique point which is rational.

\section{Complexification of the operators $T_k$ and $S$}\label{sec:complexTkS}
Let $J$ be a closed interval and $k\in\N$. Let $\Occ{k}{J}$ be the complex vector space of holomorphic functions in $\C\setminus J$, meromorphic in $\overline{\C}\setminus J$
with a zero at infinity of order at least $k$. 
There exists $C_{V,k}>0$ such that for each $\varphi\in\Occ{k}{J}$ and each neighbourhood $V$ of $J$ we have
\begin{equation}\label{eq:O^k bound}
|\varphi(z)|\le C_{V,k}|z|^{-k}\sup_{\C\setminus \overline{V}}|\varphi| \quad \text{for } z\in C\setminus\overline{V}.
\end{equation}
This fact can be easily proven in the following way. 
We obtain that the function $\dbtilde{\varphi}$ defined as 
$\dbtilde{\varphi}(w)=\varphi(w)\cdot w^k$
is still analytic on $\C\setminus \overline{V}$. 
By the maximum principle, we obtain that both functions $\dbtilde{\varphi}$ and $\varphi$ attain their maximum on the boundary of $\C\setminus \overline{V}$. Thus, we obtain for all $w\in \C\setminus \overline{V}$
\begin{align}
 \left|\varphi(w)\right|
 &=\frac{\left|\dbtilde{\varphi}(w)\right|}{\left|w\right|^k}
  \le \frac{\sup_{\C\setminus \overline{V}}|\dbtilde{\varphi}|}{\left|w\right|^k}
  =\frac{\sup_{\partial (\C\setminus \overline{V})}|\dbtilde{\varphi}|}{\left|w\right|^k}
  \le \frac{C_{V,k}\,\sup_{\C\setminus \overline{V}}|\varphi|}{\left|w\right|^k}\label{eq: max princ}
\end{align}
with $C_{V, k}= \sup_{x\in \partial (\C\setminus \overline{V})}\left|x\right|^k$.

Let $\Occ{-k}{J}$ be the complex vector space of holomorphic functions in $\C\setminus J$, meromorphic in $\overline{\C}\setminus J$
with a pole at infinity of order at most $k$.

\subsection{Hyperfunctions and extensions to spaces of complex analytic functions of the operators $T_k$ and $S$}
\label{subsec:cxBrjuno overview}
Recall that $T_k:=\T{k}{\alpha}$ as in \eqref{eq:k-Brjuno operator} for given $\alpha\in[1/2,1]$ and $S:=S_{\alpha}$.
We want to extend $T_k$ and $S$ for $\alpha=1$ to the space of complex analytic functions. 

We proceed as in \cite{MMY3}, namely we extend $T_k$ and $S$ to $A'([0,1])$ the space of hyperfunctions with support contained in $[0,1]$ for which we will first introduce some definitions. 
In Proposition \ref{pr:Tkest}, we will give the definition of these operators and prove that they are indeed well defined. In Proposition \ref{pr:inv}, we will see that this definition makes sense in terms of $T_k$ and $S$ being a complex extension of $T_k$ and $S$ being studied in the previous sections.

\medskip 

\subsubsection{Hyperfunctions}
Let $K$ be a non-empty compact set $K\subset\R$. Let us denote by $\cO(K)$ the space of functions analytic in a neighbourhood $V$ of $K$.
\emph{A hyperfunction with support in $K$} is a linear functional $u$ on $\cO(K)$, such that for all neighbourhoods $V$ of $K$, there is a constant $C_V>0$ such that
$$|u(\varphi)|\le C_V \sup_{V}|\varphi|, \quad \forall \varphi\in\cO(V).$$
We denote by $A'(K)$ the space of hyperfunctions with support in $K$.

The space $A'([0,1])$ is canonically isomorphic to $\Oc{1}$ the complex vector space of holomorphic functions on $\overline{\C}\setminus [0,1]$ vanishing at infinity. 
Let 
\begin{equation}\label{eq:c_z}
c_z(x)=\frac{1}{\pi(x-z)}.
\end{equation}
 Given $u\in A'([0,1])$, the corresponding $\varphi \in \Oc{1}$ is obtained by
\begin{equation}\label{eq:hyp_iso2}
\varphi(z)=u(c_z), \quad \forall z\in \C\setminus [0,1].
\end{equation}
On the other hand, for every $\varphi \in \Oc{1}$, the corresponding $u\in A'([0,1])$ is given by
\begin{equation}\label{eq:hyp_iso}
u(\psi)= \frac{i}{2} \int_{\gamma} \varphi(z)\psi(z)\mathrm{d}z, \quad \forall \psi \in \cO(V),
\end{equation}
where $V$ is a complex neighbourhood of $[0,1]$ and $\gamma$ is any piecewise $\mathcal{C}^1$ path winding around $[0,1]$ in the positive direction.
That $u$ as in \eqref{eq:hyp_iso} and $c_z$ as in \eqref{eq:c_z} are inverse to each other can be easily seen by using the substitution $\xi=1/\omega$ resulting in $\mathrm{d}\omega=-\xi^{-2}\mathrm{d}\xi$ which gives
\begin{align*}
 u(c_z)
 &=\frac{i}{2} \int_{\gamma} \varphi(\omega)c_z(\omega)\mathrm{d}\omega
 =\frac{i}{2} \int_{\gamma} \varphi(\omega)\,\frac{1}{\pi}\,\frac{1}{\omega-z}\,\mathrm{d}\omega
 =\frac{i}{2\pi} \int_{1/\gamma} \frac{\varphi(1/\xi)}{\left(1/\xi-z\right)(-\xi^2)}\,\mathrm{d}\xi\\
 &=\frac{1}{z}\int_{1/\gamma}\frac{i}{2\pi}\,\frac{\varphi(\xi^{-1})\xi^{-1}}{\xi-1/z}\,\mathrm{d}\xi
 =\frac{1}{z}\, \varphi(z)\,\left(\frac{1}{z}\right)^{-1}
 =\varphi(z),
\end{align*}
where we denote by $1/\gamma$ the transformed path and notice that $1/\gamma$ goes along the negative direction. 
We remark that in \eqref{eq:hyp_iso} we correct  the formula for $u$ given in \cite[Appendix 2]{MMY3} where it is given with an additional factor $\pi$.
We write $u(x)=\frac{1}{2i}(\varphi(x+i0)-\varphi(x-i0))$ (see \cite[Appendix~2]{MMY3} for details).

\subsubsection{Formulas of the extensions of $T_k$ and $S$}
We obviously have $S=-T_1$. Hence, for the following we will introduce the main formulas for $T_k$ only and give explanations where differences occur for $S$.
For $k \in \N$, $m\in\N^*$ and $f\in L^2([0,1])$, let us consider
\begin{equation}\label{eq:Tkm}
 T_{k,m}f(x) = 
 \left\{
     \begin{array}{ll}
       x^kf\left(\frac{1}{x}-m\right), & \text{if }x \in \left[\frac{1}{m+1},\frac{1}{m}\right],\\
       0, & \textnormal{otherwise}.
     \end{array}
   \right. 
\end{equation}
We then have $T_k=\sum_{m=1}^{\infty}T_{k,m}$.
For $\varphi,~ \psi\in L^2([0,1])$,
we define the $L^2$-adjoint $T^*_{k,m}$ 
by
\begin{equation}
\int_{0}^{1}T_{k,m}\varphi(x)\psi(x)dx=\int_{0}^{1}\varphi(x)T^*_{k,m}\psi(x)dx. \label{eq: L2 adjoint}
\end{equation}
Therefore, we have
\begin{equation}\label{eq:Tkmadj}
 T^*_{k,m}\psi(x)=\frac{1}{(m+x)^{k+2}}\psi\left(\frac{1}{m+x}\right).
\end{equation}

By a slight generalisation of \cite[(1.10)]{MMY3}, we have that if $\psi$ is holomorphic in a neighbourhood $V$ of $[0,1]$, then $T^*_{k,m}\psi$ is holomorphic in $V$ and we have
\begin{equation*}
\sup_{V}|T^*_{k,m}\psi|\le \frac{1}{2m^{k+2}}\sup_{V}|\psi|,
\end{equation*}
where we may take $V= D_\infty^c$ equipped with the Poincar\'{e} metric on the hyperbolic Riemann surface
$\overline{\C}\setminus [0,1]$.
It follows that the series $\sum_{m=1}^{\infty}T_{k,m}u$ 
converges to a hyperfunction $T_{k}u$ in $A'([0,1])$.

We have
\begin{equation*}
T_{k,m}^\ast c_z(x) = 
- z^k  \left(c_{\frac{1}{z}-m}(x) - c_{-m}(x)\right) + \sum_{n=1}^k \frac{z^{k-n}}{n!} \cdot \frac{\partial^n}{\partial z^n}c_z(x)\Big|_{z=-m},
\end{equation*}
where $c_z$ is as in \eqref{eq:c_z}.
For $\varphi\in \mathcal{O}^1(\overline{\C}\backslash[0,1])$, let $u$ be the corresponding hyperfunction in $ A'([0,1])$ as in \eqref{eq:hyp_iso2}.
Then, $T_{k,m}\varphi$ is defined by $(T_{k,m} u)(c_z)$.
If $z\not\in [\frac{1}{m+1},\frac{1}{m}]$, then we have
$$
(T_{k,m}u)(c_z) = u(T_{k,m}^\ast c_z).
$$
This follows from \eqref{eq: L2 adjoint} and the fact that for $z\notin [\frac{1}{m+1},\frac{1}{m}]$ 
we have that $T_{k,m}u$ and $T_{k,m}^\ast c_z$ are analytic. Hence, \eqref{eq: L2 adjoint} already implies the equality.

By \eqref{eq:hyp_iso2} and \eqref{eq:hyp_iso}, we have
\begin{align*}
u(T_{k,m}^\ast c_z) & = - z^k  \left(u(c_{\frac{1}{z}-m}(x)) - u\left(c_{-m}(x)\right)\right) + \sum_{n=1}^k \frac{z^{k-n}}{n!} u\left(\frac{\partial^n}{\partial z^n}c_z(x)|_{z=-m}\right)\\
& = - z^k\left(\varphi\left(\frac1z - m\right)-\varphi(-m)\right) + \sum_{n=1}^{k}\frac{z^{k-n}}{n!}\varphi^{(n)}(-m).
\end{align*}

For $\varphi\in \Oc{1}$, the formula for $T_k$ is given by
\begin{equation}\label{eq:Tk}
 T_k\varphi (z) = - \sum_{m=1}^{\infty} z^k\left(\varphi\left(\frac{1}{z}-m\right)-\varphi(-m)\right) + \sum_{m=1}^{\infty}\sum_{n=1}^{k}\frac{z^{k-n}}{n!}\varphi^{(n)}(-m),
\end{equation}
which will be shown in Proposition~\ref{pr:Tkest}.
By using $S = -T_1$, we can also deduce that a natural extension of $S$ to the space of complex analytic functions is
\begin{equation}\label{eq:S}
S\varphi(z) = \sum_{m=1}^\infty z\left(\varphi\left(\frac{1}{z}-m\right)-\varphi(-m)\right)- \sum_{m=1}^\infty \varphi'(-m).
\end{equation}

\subsubsection{Algebraic properties of the inversion of $(1-T_k)$ and $(1-S)$, the monoid $\mathcal{M}$ and its actions.}
 Let us consider the monoid
\begin{equation*}
\mathcal{M}=\left\{g=\begin{pmatrix} a & b \\ c& d \end{pmatrix} \in \GL_2(\Z) \; | \; d\geq b \geq a \geq 0 \textnormal{ and } d\geq c\geq a\geq 0 \right\}\cup \begin{pmatrix}1& 0\\ 0 & 1\end{pmatrix},
\end{equation*}
which is generated by $g(m) := \left(\begin{smallmatrix}0&1\\1&m\end{smallmatrix}\right)$ for $m\in \N$.
For $g\in\cM$ 
and $\varphi\in\Oc{1}$, we define
\begin{equation}\label{eq:Lkg}
 L_{k,g} \varphi (z) := \det(g)^{k+1}\Bigg[(a-cz)^k\varphi\left(\frac{dz-b}{a-cz}\right)-\sum_{n=0}^{k}(a-cz)^{k-n}\frac{\det(g)^{n}}{c^{n} n!}\varphi^{(n)}\left(-\frac{d}{c}\right)\Bigg].
\end{equation}
and
\begin{equation}\label{eq:Lg}
 \wL_{g} \varphi (z) := \det(g)(a-cz)\left(\varphi\left(\frac{dz-b}{a-cz}\right)-\varphi\left(-\frac{d}{c}\right)\right) -\frac{1}{c}\varphi'\left(-\frac{d}{c}\right)
  = \det(g) L_{1,g}\varphi(z).
\end{equation}
As we will see in Proposition~\ref{pr:inv}, we have
\begin{equation}\label{eq:(1-Tk)^{-1}}
(1-T_k)^{-1}\varphi(z)=\sum_{j=0}^{\infty}T_k^j\varphi(z)=\sum_{g\in\mathcal{M}}L_{k,g}\varphi(z),
\end{equation}
and from analogous arguments we will see that
\begin{equation}\label{eq:1-S}
(1-S)^{-1}\varphi(z)=\sum_{j=0}^{\infty}S^j\varphi(z)=\sum_{g\in\mathcal{M}}\wL_{g}\varphi(z).
\end{equation}
For $g =\big(\begin{smallmatrix}a&b\\c&d \end{smallmatrix}\big) $
and $n\in\Z$, let
\begin{equation}\label{eq:Lkgn}
 L_{g}^{(n)}\varphi(z):=\frac{\det(g)^{n+1}}{(a-cz)^n}\varphi\left(\frac{dz-b}{a-cz}\right)
\end{equation}
 and
\begin{equation}\label{eq:Lgn}
 \wL_{g}^{({n})}\varphi(z):=\frac{\det(g)^n}{(a-cz)^{n}}\,\varphi\left(\frac{dz-b}{a-cz}\right)
 =\det(g)\,L_{g}^{(n)}\varphi(z).
\end{equation}

Analogously to \cite[(2.2)]{MMY3}, we are interested in a connection between $L_{k,g}$ and $L_g^{(k+2)}$ and between $\wL_g$ and $\wL_g^{(3)}$ respectively. 
In the first case, we obtain by induction that
 for $t\in \N^*$ and $n \in \Z$, we have
\begin{align*}
 (L_{g}^{(n)}\varphi)^{(t)} = & \det(g)^{n+1} \sum_{\ell=0}^{t-1} (n+\ell)(n+\ell+1)\cdots (n+t-1) \binom{t}{\ell} c^{t-\ell} \det(g)^\ell L_{g}^{(n+t+\ell)}\varphi^{(\ell)} \\
 & + \det(g)^{t+n+1} L_g^{(n+2t)}\varphi^{(t)},
\end{align*}
where $\sum_{\ell=0}^{-1}$ denotes an empty sum.
In particular, for all $k\in \N^*$, we have
\begin{equation}\label{eq: Lgnphit2}
(L_{g}^{(-k)}\varphi)^{(k+1)}=L_{g}^{(k+2)}\varphi^{(k+1)}. 
\end{equation}

\subsubsection{Polynomial corrections to the actions of $\mathcal{M}$}
If $\varphi \in \Oc{-k}$, then it can be uniquely written as
\begin{equation}\label{eq:phidecomposition}
 \varphi(z)=\xi_kz^k+ \xi_{k-1}z^{k-1}+\cdots+ \xi_0+p_{k}(\varphi)(z), 
\end{equation}
with $p_{k}(\varphi)\in \Oc{1}$.
Note that if $\varphi\in\Oc{-k}$, then 
$\varphi^{(k+1)}=p_{k}(\varphi)^{(k+1)}$.
If $g =\big(\begin{smallmatrix}a&b\\c&d \end{smallmatrix}\big) \in \GL_2(\Z)$ and $\psi(z)=\xi_kz^k+\xi_{k-1}z^{k-1}+\cdots+ \xi_0$, then
\begin{equation}\label{eq:L_g^(-k)psi}
L_{g}^{(-k)}\psi(z)= \det(g)^{1-k}\big(\xi_k(dz-b)^k+ \xi_{k-1}(dz-b)^{k-1}(a-cz)+\cdots+ \xi_0(a-cz)^k \big),
\end{equation}
which means that $p_k( L_g^{(-k)}\psi)=0$.
Thus, $p_k$ is a projection from $\Oc{-k}$ to $\Oc{1}$.
We have the formula
\begin{equation}\label{eq: Lkg projection}
L_{k,g}\varphi =p_k(L_{g}^{(-k)}\varphi)
\end{equation}
for $g\in \mathcal{M},~ \varphi\in\Oc{1}$ which defines an action of $\mathcal{M}$ on $\Oc{1}$, where $p_k$ depends on $k$. 
To verify this formula we define $\psi(u)\coloneqq (L_{g}^{(-k)}\varphi)\big(\frac{a}{c}-\frac{1}{cu}\big)$. Since 
$\lim_{z\to\infty}p_k(L_{g}^{(-k)}\varphi)(z)=0$ is equivalent to $\lim_{u\to 0} \psi(u)=0$, it is enough to look at a projection of $\psi$. We have 
\begin{align*}
 \psi(u)= \det(g)^{1-k} u^{-k}\varphi\left(\frac{u\det(g)-d}{c}\right)
\end{align*}
and if we denote by $\widetilde{p}_k$ the projection to a function vanishing at $0$, then we obtain 
\begin{align}
 \widetilde{p}_k\left(\psi(u)\right)
&\nonumber  = \psi(u)- \det(g)^{1-k} u^{-k} \sum_{n=0}^k  \frac{\psi^{(n)}(0)\, u^n}{n!} \\
& = \psi(u)- \det(g)^{1-k} \sum_{n=0}^k  \frac{\det(g)^n\, \varphi^{(n)}\left(-\frac{d}{c}\right)\, u^{n-k}}{c^n n!}.\label{eq: Lkg proj calc}
\end{align}
Substituting back $u=(a-cz)^{-1}$ gives \eqref{eq: Lkg projection}. 
From the definition of $L_{k,g}$ as in \eqref{eq:Lkg} and \eqref{eq:L_g^(-k)psi}, the following diagram commutes:
\begin{center}
\begin{tikzpicture}
  \matrix (m) [matrix of math nodes,row sep=3em,column sep=4em,minimum width=2em]
  {
     \Oc{-k} & \Oc{-k} \\
     \Oc{1} & \Oc{1} \\};
  \path[-stealth]
    (m-1-1) edge node [left] {$p_k$} (m-2-1)
            edge node [above] {$L_{g}^{(-k)}$} (m-1-2)
    (m-2-1.east|-m-2-2) edge node [below] {$L_{k,g}$} (m-2-2)
    (m-1-2) edge node [right] {$p_k$} (m-2-2);
\end{tikzpicture}
\end{center}
By \eqref{eq:Lkg} and \eqref{eq:Lg}, the connection between $\wL_g$ and $\wL_g^{(-1)}$ is similar as they only differ by a factor $\det(g)$ from $L_{1,g}$ and $L_g^{(-1)}$.

By the calculations in \eqref{eq: Lkg proj calc}, the $(k+1)$th derivative of the difference between $L_{k,g}\varphi$ and $L_g^{(-k)}\varphi$ is zero for $\varphi\in\Oc{1}$.
Thus, it follows from \eqref{eq: Lgnphit2} that
\begin{equation}
(L_{k,g}\varphi)^{(k+1)}(\xi)=L_{g}^{(k+2)}\varphi^{(k+1)}(\xi)
= \frac{\det(g)^{k+1}}{(a-c\xi)^{k+2}} \varphi^{(k+1)}(g^{-1}.\xi),\label{eq: Lkgvarphik+1}
\end{equation}
where $\varphi^{(k+1)}\in \Oc{k+2}$. 
By Taylor's theorem, we have for every $\omega \in \overline{\C}\backslash[0,1]$ that
\begin{align*}
 L_{k,g}\varphi(z)
 &= L_{k,g}\varphi(\omega) + (L_{k,g}\varphi)'(\omega)(z-\omega)+\cdots+\frac{(L_{k,g}\varphi)^{(k-1)}}{(k-1)!}(z-\omega)^{k-1}\\
 &\qquad+\int_\omega^z\frac{(L_{k,g}\varphi)^{(k+1)}(\xi)}{k!}(z-\xi)^kd\xi.
\end{align*}
When $\omega$ goes to $\infty$, then only the remainder term will be left since $\frac{(L_{k,g}\varphi)^{(\ell)}(\omega)}{\ell!}(z-\omega)^{\ell}\in \Oc{1}$ for all $0\le \ell \le k-1$.
From \eqref{eq: Lkgvarphik+1}, the remainder term can be written as
$$\frac{\det(g)^{k+1}}{k!}\int_\infty ^z\frac{(z-\xi)^k}{(a-c\xi)^{k+2}}\varphi^{(k+1)}(g^{-1}.\xi)d\xi.$$
By a change of variable with $\xi=\xi(t)= g.\left(-\frac{d}{c}+\frac{\det(g)t}{c(a-cz)}\right)$, which means that 
$\frac{1}{(a-c\xi)^2}d\xi = \frac{1}{c(a-cz)} dt$ and $\xi = -\frac{(a-cz)}{ct}+\frac ac$, we then deduce that the remainder term equals
$$\frac{\det(g)^{k+1}}{k!}\int_0^1\left(\frac{1-t}{c}\right)^k\frac{\varphi^{(k+1)}(-\frac{d}{c}+\frac{\det(g)t}{c(a-cz)})}{c(a-cz)}dt.$$
Then, for $\varphi\in\Oc{1}, g\in\mathcal{M}$ and $z\notin \left[\frac{b}{d}, \frac{a}{c}\right]$, we have
\begin{equation}\label{eq: Lkgphi1}
L_{k,g}\varphi(z) =\frac{c^{-(k+1)}\det(g)^{k+1}}{k!}(a-cz)^{-1} \int_0^1(1-t)^k\varphi^{(k+1)}\left(-\frac{d}{c}+\frac{\det(g)t}{c(a-cz)}\right)dt. 
\end{equation}
We then have for the special case $g=g(m)$ that
\begin{align}
 L_{k,g(m)} \varphi (z)& = -z^k\left(\varphi\left(\frac{1}{z}-m\right)-\varphi\left(-m\right)\right)
+\sum_{n=1}^{k}\frac{z^{k-n}}{n!}\varphi^{(n)}\left(-m\right)\label{eq: Lkg(m)phi}\\
&=\frac{ (-1)^{k}}{z\,k!}\int_{0}^{1}\varphi^{(k+1)}\left(\frac{t}{z}-m\right)(1-t)^k\mathrm{d}t.\label{eq: Lkg(m)phi1}
\end{align}
Connected to that we note that by \eqref{eq: Lkgvarphik+1} $L_{g}^{(k+2)}$ can be represented as
\begin{align}
 L_g^{(k+2)}\psi(z)= 
\frac{\det(g)^{k+1}}{(a-cz)^{k+2}}\psi\left(\frac{dz-b}{a-cz}\right)=\big(L_{k,g} \psi^{(- (k+1))} \big)^{(k+1)}(z) , \label{eq: Lkg deriv}
\end{align}
for $\psi \in \Oc{-k}$, where $\psi^{(-(k+1))}$ denotes the $(k+1)$th primitive (i.e.\ an integration between $\infty$ and $z$) of $\psi$.
In the special case that $g=g(m)$, we have that $\det\left( g(m)\right)=-1$ and thus
\begin{align}
L_{g(m)}^{(k+2)} \psi (z) &=-z^{-(k+2)}\psi\left(\frac{1}{z}-m\right).\label{eq: Lkg(m) deriv}
\end{align}

For $\varphi \in \Oc{1}$ we can immediately conclude from \eqref{eq:Lg}, \eqref{eq: Lkg(m)phi} and \eqref{eq: Lkg(m)phi1} that
\begin{equation}\label{eq:Lgint} 
\wL_{g(m)} \varphi (z) = z\left(\varphi\left(\frac{1}{z}-m\right)-\varphi\left(-m\right)\right)-\varphi'\left(-m\right)=\frac{1}{z}\int_{0}^{1}\varphi''\left(\frac{t}{z}-m\right)(1-t)dt. 
\end{equation}

We also note that for simplicity we were restricting us to $\varphi\in \Oc{k}$. However, the same calculations of this section hold true if we consider instead 
$\gamma_1> \gamma_0>-1$, $I=[\gamma_0,\gamma_1]$, $z\notin \left[0,\frac{1}{\gamma_0+m}\right]$, and $\varphi\in \Occ{k}{I}$.

\section{Convergence of the sum over the monoid and boundary behaviour}\label{sec:sumboundary}

In this Section we will show how to adapt the results of Sections 3 and 4 of \cite{MMY3} to our monoid action. Also in our case the
complex continued fraction introduced in Section \ref{sec:compl CF} will play a fundamental role.

\subsection{Convergence of the sum over the monoid}
From the following proposition we will be able to deduce that $T_k$ is given by \eqref{eq:Tk}. It is an analog to \cite[Prop.~3.1]{MMY3}.

Let $-1<\gamma_0<\gamma_1$, $I=[\gamma_0,\gamma_1]$, $J=\Big[0,\frac{1}{(1+\gamma_0)}\Big]$, 
and 
\begin{align}
U_\varepsilon:= \{z\in\C\;|\;\re(z)\in(-\infty,\gamma_0-\varepsilon]\cup[ \gamma_1+\varepsilon,+\infty) \text{ or } \im(z)\notin(-\varepsilon,\varepsilon)\}.\label{eq: def Ueps}
\end{align}
\begin{prop}\label{pr:Tkest}
We have the following statements.
\begin{enumerate}[(i)]
\item\label{en: 1} For all $\varphi\in\Occ{1}{I}$, the series $\sum_{m=1}^\infty L_{k,g(m)}\varphi$ converges uniformly 
on compact subsets $K \subseteq \overline{\C}\setminus J$ to a function $T_k\varphi\in \Occ{1}{J}$ and there exist $\varepsilon>0$
and $C_{K,k}>0$ such that $\displaystyle{\sup_{K}|T_k\varphi|\le C_{K,k}\sup_{U_{\varepsilon}}|\varphi|}.$
 \item\label{en: 2} For all $\psi\in\Occ{k+2}{I}$, the series $\sum_{m=1}^\infty L_{g(m)}^{(k+2)}\psi$ converges uniformly on compact subsets $K \subseteq \overline{\C}\setminus J$ to a function
in $\Occ{k+2}{J}$, which we denote by $T_{k}^{(k+2)}\psi$
and there exist $\varepsilon>0$ and $\overline{C}_{K,k}>0$ such that $\displaystyle{\sup_{K}|T_{k}^{(k+2)}\psi|\le \overline{C}_{K,k} \sup_{U_{\varepsilon}}|\psi|}.$ 
\item\label{en: 3} For all $\varphi \in \Occ{1}{I}$, we have $T^{(k+2)}_k\varphi^{(k+1)}=(T_k\varphi)^{(k+1)}$.
\end{enumerate}
\end{prop}

\begin{rmk}
 Analogous statements can immediately be deduced for $\wL_{g(m)}$ and $S$ by using the relations $\wL_{g(m)}=-L_{1,g(m)}$ and $S=-T$.
\end{rmk}

\begin{proof}[Proof of Proposition \ref{pr:Tkest}]
 Let $\varepsilon>0$, then, from \eqref{eq:O^k bound},
there exists $c_{1,\varepsilon,k}>0$ such that for all $z\in U_\varepsilon$ and $\psi\in \Occ{k+2}{I}$ we have
$|\psi(z)|\le c_{1,\varepsilon,k} |z|^{-(k+2)}\sup_{U_\varepsilon}|\psi|$. 

If $K \subseteq \overline{\C}\setminus J$ is compact, then there exists $\varepsilon(K)\in \left(0, \max\{|\gamma_0|, |\gamma_1|\}\right)$ such that
$\frac{1}{z}-m\in U_{\varepsilon(K)}$, for all $z\in K$ and $m\in \N$, 
which implies that $\veps(K)\le |1/z-m|$.
Also, there exist $c_{2,K}>0$ and $M_{\gamma_0}\in\mathbb{N}$ such that for all $z\in K$ and $m\geq M_{\gamma_0}$ we have that 
$\Big|\frac{1}{z}-m\Big|^{-1}\le c_{2,K} m^{-1}$. 
Therefore, for all $z\in K$ and $\psi\in \Occ{k+2}{I}$ we have
\begin{align}
 \sum_{m=1}^{\infty}\left|\psi\left(\frac{1}{z}-m\right)\right|\le& \sum_{m=1}^{\infty}c_{1, \varepsilon(K)} \left|\frac{1}{z}-m\right|^{-(k+2)}\sup_{U_{\varepsilon(K)}}|\psi|\notag\\
=& c_{1, \varepsilon(K),k}\sup_{U_{\varepsilon(K)}}|\psi|\left(\sum_{m=1}^{M_{\gamma_0}-1} \left(\left|\frac{1}{z}-m\right|^{-1}\right)^{k+2}+\sum_{m=M_{\gamma_0}}^{\infty} \left( \left|\frac{1}{z}-m\right|^{-1}\right)^{k+2}\right)\notag\\
\le& c_{1, \varepsilon(K),k}\sup_{U_{\varepsilon(K)}}|\psi|\left(\sum_{m=1}^{M_{\gamma_0}-1} \left({\varepsilon(K)}^{-1}\right)^{k+2}+\sum_{m=M_{\gamma_0}}^{\infty} \left({c_{2,K}}\, m^{-1}\right)^{k+2}\right)\notag\\
=& c_{1,\varepsilon(K),k}\sup_{U_{\varepsilon(K)}}|\psi|\left((M_{\gamma_0}-1){\varepsilon(K)}^{-(k+2)}+{c_{2,K}}^{k+2}\sum_{m=M_{\gamma_0}}^{\infty} m^{-(k+2)}\right)\notag\\
\le&{\overline{C}_{K,k}}\sup_{U_{\varepsilon(K)}}|\psi|,\label{eq: sum m infty}
\end{align}
with $\overline{C}_{K,k}=c_{2,\varepsilon(K),k}\big((M_{\gamma_0}-1)\varepsilon(K)^{-(k+2)}+c_{2,K}^{k+2}\sum_{m=M_{\gamma_0}}^{\infty} m^{-(k+2)}\big)$ and we obtain the first part of \eqref{en: 1}.
By integrating $k+1$ times between $\infty$ and $z$ and substituting $\varphi=\psi^{(-(k+1))}$, 
we obtain the first part of \eqref{en: 2}. 
The assertion \eqref{en: 3} then follows immediately.

Hence, by \eqref{eq: Lkgvarphik+1}, \eqref{eq: Lkg(m) deriv} and \eqref{eq: sum m infty} we have for all $\varphi\in\Occ{1}{I}$ and  all $z\in K$ that
\begin{align*}
|(T_k\varphi)^{(k+1)}(z)|&=|(T_k^{(k+2)}\varphi^{(k+1)})(z)|
\le \sum_{m=1}^{\infty}\left|L_{g(m)}^{(k+2)}\varphi^{(k+1)}(z)\right|\\
&\le |z|^{-(k+2)}\cdot \sum_{k=1}^{\infty}\left|\varphi^{(k+1)}\left(\frac{1}{z}-m\right)\right|
\le \overline{C}_{K,k} |z|^{-(k+2)}\sup_{U_{\varepsilon(K)}}|\varphi^{(k+1)}|,
\end{align*}
implying \eqref{en: 2}. We conclude \eqref{en: 1} by using Cauchy's formula.
\end{proof}

In particular, we immediately obtain from \eqref{eq: Lkg(m)phi} that
\begin{equation*}
 T_k\varphi (z) =- \sum_{m=1}^{\infty} z^k\left(\varphi\left(\frac{1}{z}-m\right)-\varphi(-m)\right) + \sum_{m=1}^{\infty}\sum_{n=1}^{k}\frac{z^{k-n}}{n!}\varphi^{(n)}(-m)
\end{equation*}
proving \eqref{eq:Tk}.

Let $D_{\infty}$ be given as in \eqref{eq: def Dinfty}
and for $\rho>0$ let
\begin{equation*}
V_\rho(D_\infty)=\{z\in\overline{\C}\setminus[0,1]\;|\;d_{\text{hyper}}(z,D_\infty)<\rho\},
\end{equation*}
where $d_{\text{hyper}}$ denotes the Poincar\'{e} metric on $\overline{\C}\setminus[0,1]$.

To define the complex $k$-Brjuno function, we deal with $(1-T_k)^{-1}$ as in the real case, see \eqref{eq:brjuno func eq} and also \eqref{eq:cx brjuno}.
The following proposition is an analogous property to \cite[Thm. 2.6]{MMY1} for the real Brjuno function and it is a generalization of \cite[Prop. 3.3]{MMY3} for the case of $k=1$, which guarantees that $(1-T_k)^{-1} = \sum_{r=1}^\infty T_k^r$ converges.

However, before we state this proposition we first give an estimate of the derivatives of $\varphi$ using Cauchy's integral formula.
For $\varphi\in\Oc{n}$, let $\dbtilde{\varphi}(z)= \varphi(z)\cdot z^{n}$.
Then we have 
$$\varphi^{(j)}(z)=\sum_{i=0}^{j}{j \choose i} \frac{(-1)^i(n+i-1)!}{(n-1)!}\cdot \frac{\dbtilde{\varphi}^{(j-i)}(z)}{z^{n+i}}.$$
By Cauchy's integral formula, for a circle $\{\omega:|\omega-z|=R\}$ contained in $V_\rho(D_{\infty})$ whose center is given by $z\in V_\rho(D_\infty)$, we have
$$
 |\dbtilde{\varphi}^{(\ell)}\left(z \right)|
 = \left|\frac{\ell!}{2\pi i}\int_{|\omega-z|=R}\frac{\dbtilde{\varphi}\left(\omega\right)}{(\omega-z)^{\ell+1}}d\omega\right|
 \le \frac{\ell!\sup_{V_\rho(D_\infty)}|\dbtilde{\varphi}|}{R^\ell}.
$$
Since we have
$$\sup_{V_\rho(D_{\infty})}|\dbtilde{\varphi}|\le \sup_{V_\rho(D_{\infty})}|\varphi|\cdot \sup_{z\in \partial V_{\rho}(D_{\infty})}|z|^{n}\le \sup_{V_\rho(D_{\infty})}|\varphi|\cdot \sup_{z\in \partial D_{\infty}}|z|^{n},$$
we obtain
\begin{equation}\label{eq:jth derivative}
|\varphi^{(j)}(z)| 
\le c_{n,j}'\frac{j!\sup_{V_\rho(D_\infty)}|\varphi|}{R^{n+j}},\quad\text{where }
c_{n,j}' := \sup_{z\in \partial D_\infty}|z|^{n} \cdot \sum_{i=0}^j  {n+i-1 \choose i}.
\end{equation}

\begin{prop}\label{pr:Tkbound}
Let $\rho\geq0$. The following statements hold.
\begin{enumerate}[(i)]
\item \label{en: 1Tkbound}  For all $k\in \mathbb{N}$, there exists $C_{\rho, k}>0$ such that, for all $r\geq 0$ and $\psi\in\Occ{k+2}{[0,1]}$, we have
\begin{equation*}
\sup_{z\in V_\rho(D_\infty)}|((T_k^{(k+2)})^r\psi)(z)|\le C_{\rho,k}\left(\frac{\sqrt{5}-1}{2}\right)^{rk} \sup_{z\in V_\rho(D_\infty)}|\psi(z)|.
\end{equation*}
\item \label{pr:Tkbound-item2} For all $k\in \mathbb{N}$, there exists $\overline{C}_{\rho, k}$ such that, for all $r\geq 0$ and $\varphi\in\Occ{1}{[0,1]}$, we have
\begin{equation*}
\sup_{z\in V_\rho(D_\infty)}|(T_k^r\varphi)(z)|\le \overline{C}_{\rho, k}\left(\frac{\sqrt{5}-1}{2}\right)^{rk} \sup_{z\in V_\rho(D_\infty)}|\varphi(z)|.
\end{equation*}
\end{enumerate}
\end{prop}
Similarly to the last proposition we remark also here that analogous statements can immediately be deduced for $\wL_{g(m)}$ and $S$ by using the relations $\wL_{g(m)}=-L_{1,g(m)}$ and $S=-T$.

The proof of the above proposition is given in the Appendix.

In order to state the next proposition, we define
\begin{equation}
Z=\left\{\begin{pmatrix} 1 & n \\ 0& 1 \end{pmatrix} \in \GL_2(\Z) \; | \; n\in \Z\right\}.\label{eq: def Z}
\end{equation}
From now on, we consider $L_{k,g}$ and $\wL_g$ for $g\in\GL_2(\Z)$.
For $\varphi\in \Oc{1}$, since $\varphi_0^{(n)}\in \Oc{n+1}$,
we have $\lim_{c\to0}\varphi_0^{(n)}(-d/c)/c^n=0$ for $n\ge 0$.
Thus, we define $L_{k,g}$ when $c=0$ by $L_{k,g}\varphi(z) = \det(g)^{k+1}a^k\varphi\left(\frac{dz-b}{a}\right)$.
Especially, for $g\in Z$, we have
$$L_{k,\left(\begin{smallmatrix}1&n\\0&1\end{smallmatrix}\right)}\varphi(z) =\wL_{\left(\begin{smallmatrix}1&n\\0&1\end{smallmatrix}\right)}\varphi(z) = \varphi(z-n).$$

For $H\subset \mathrm{GL}_2(\Z)$, let us denote
$$\sum\nolimits_{H}^{(k+2)}\psi:=\sum_{g\in H}L_{g}^{(k+2)}\psi\quad\text{and}\quad
\sum\nolimits_{H,k}\varphi:=\sum_{g\in H}L_{k,g}\varphi,$$ 
which are uniformly summable on compact subsets of $\overline{\C}\setminus[0,1]$.

The following three propositions which are a generalization of \cite[Coro. 3.6]{MMY3} will show us a relation between $\sum_{\cM, k}$ and $T_k$.
They state that $\sum_{\cM, k}\varphi$ and $\sum_{Z\cdot \cM, k}\varphi$ are well-defined on any compact set of $\overline{\C}\setminus[0,1]$ for $\varphi\in \Oc{k+2}$ and $\sum_{\cM, k}$ is the inverse of $(1-T_k)$.

\begin{prop}\label{pr:Lg3sum}
 Let $\psi\in\Oc{k+2}$.
We have
\begin{equation}
\sum\nolimits_{\mathcal{M}}^{(k+2)}\psi=\sum_{r=0}^\infty (T_k^{(k+2)})^r\psi.\label{eq: in Lg3sum}
\end{equation}
Let $K\subset \overline{\C}\setminus[0,1]$ be compact. Then, for all $k\in\mathbb{N}$, there exist $\varepsilon(K)>0$ and $C_{K,k}$ such that
\begin{equation}
\sup_{K}\left|\sum\nolimits_{\mathcal{M}}^{(k+2)}\psi\right|\le C_{K,k}\sup_{U_{\varepsilon(K)}}|\psi|,\label{eq: in Lg3sum1}
\end{equation}
where $U_{\veps(K)}$ is as in \eqref{eq: def Ueps} with $\gamma_0=0$ and $\gamma_1=1$.

Furthermore, $\sum\nolimits_{Z\cdot\mathcal{M}}^{(k+2)}\psi$ is uniformly summable on all domains of the form $\{|\re(z)|<A, |\im(z)|\geq \delta\}$, for some $A,\delta>0$. 
It is holomorphic in $\C\setminus\R$, $\Z$-periodic, bounded in the neighbourhood of $\pm i\infty$, and
\begin{equation}
\sum\nolimits_{Z\cdot\mathcal{M}}^{(k+2)}\psi=\sum\nolimits_Z^{(k+2)}\sum\nolimits_{\mathcal{M}}^{(k+2)}\psi.\label{eq: in Lg3sum double sum}
\end{equation}
\end{prop}
\begin{proof}
Equation \eqref{eq: in Lg3sum} follows from \eqref{eq: proof Tkbound}, and equation \eqref{eq: in Lg3sum1} follows directly from Proposition~\ref{pr:Tkbound} taking $\veps$ and $\rho$ such that $K\subset V_\rho(D_\infty)\subset U_{\veps(K)}\subset \overline{\C}\setminus[0,1]$.

Furthermore, making use of \eqref{eq: Lkg deriv}, we have for $g' =\left(\begin{smallmatrix} 1 & n \\ 0 & 1  \end{smallmatrix}\right)\in Z$ that $L_{g'}^{(k+2)}\psi(z)=\psi\left(z-n\right)$.
Then, we have for all $g'\in Z$
and $g\in\mathcal{M}$ that $L_{g'\cdot g}^{(k+2)}\psi=L_{g'}^{(k+2)}\left(L_{g}^{(k+2)}\psi\right)$
 giving \eqref{eq: in Lg3sum double sum} in case that at least one side is uniformly summable which we will show in the following. 

 We have by Proposition~\ref{pr:Tkest} that $(T_k^{(k+2)})^r\psi \in\Occ{k+2}{[0,1]}$ for any $r\ge 0$. 
 Furthermore, by the above consideration, we have that $\sum\nolimits_Z^{(k+2)}\sum\nolimits_{\mathcal{M}}^{(k+2)}\psi(z)=\sum_{n=1}^{\infty}\left(\sum\nolimits_{\mathcal{M}}^{(k+2)}\psi\left(z-n\right)\right)$. Since $\sum_{n=1}^{\infty}f(z-n)\in \Occ{k+2}{[0,1]}$ holds if $f \in \Occ{k+2}{[0,1]}$, it follows that $\sum\nolimits_Z^{(k+2)}\sum\nolimits_{\mathcal{M}}^{(k+2)}\psi\in \Occ{k+2}{[0,1]}$.
\end{proof}

\begin{prop}\label{pr:Lgsum}
  Let $\varphi\in\Oc{1}$. 
We have
\begin{equation}
\sum\nolimits_{\mathcal{M},k}\varphi=\sum_{r=0}^\infty T_k^r\varphi.\label{eq: sum M sum T}
\end{equation}

Furthermore, $\sum\nolimits_{Z\cdot\mathcal{M},k}\varphi$ is holomorphic in $\C\setminus\R$, $\Z$-periodic, vanishing at $\pm i\infty$ and we have
\begin{align}\label{eq: M sum deriv} 
\left(\sum_{\mathcal{M}, k}\varphi\right)^{(k+1)}  =\sum\nolimits_{\mathcal{M}}^{(k+2)}\varphi^{(k+1)} \quad \text{ and } \quad
 \left(\sum_{Z\cdot\mathcal{M},k}\varphi\right)^{(k+1)}  =\sum\nolimits_{Z\cdot\mathcal{M}}^{(k+2)}\varphi^{(k+1)}.
\end{align}
\end{prop}
\begin{proof}
By definition of $\sum_{H,k}$, $\sum_{H}^{(k+2)}$ and \eqref{eq: Lkgvarphik+1}, the equations in  \eqref{eq: M sum deriv} hold.
Equation \eqref{eq: sum M sum T} follows from integrating the expression in Proposition~\ref{pr:Lg3sum} with $\psi = \varphi^{(k+1)}$, combining with \ref{pr:Tkest}-\eqref{en: 3} and the first equation of \eqref{eq: M sum deriv}.
The remaining statements follow immediately from the properties of $\varphi$ and the uniform summability.
\end{proof}

\begin{prop}\label{pr:inv}
 We have
\begin{align*}
 (1-T_k)\sum\nolimits_{\mathcal{M},k}&= \sum\nolimits_{\mathcal{M},k}(1-T_k)=\textnormal{id},\quad  \text{and }\\
 (1-T_k^{(k+2)})\sum\nolimits_{\mathcal{M}}^{(k+2)}&= \sum\nolimits_{\mathcal{M}}^{(k+2)} (1-T_k^{(k+2)})=\textnormal{id}. 
\end{align*}
\end{prop}
\begin{proof}
 It follows from Propositions~\ref{pr:Lg3sum} and \ref{pr:Lgsum}.
\end{proof}

For $H\subset \mathrm{GL}_2(\Z)$, 
We denote by 
\begin{equation}\label{eq:w def sum}
\wsum_H^{(3)}\psi := \sum_{g\in H}\wL_{g}^{(3)}\psi\quad \text{and} \quad
\wsum_H\varphi:=\sum_{g\in H}\wL_{g}\varphi,
\end{equation}
which are uniformly summable on compact subsets of $\C\setminus [0,1]$.
As an analog to the statements above and to \cite[Coro.~3.6]{MMY3}
we obtain the same statements as in Propositions \ref{pr:Lg3sum}, \ref{pr:Lgsum} and \ref{pr:inv}
with $k=1$ where we have to replace $T_1$ by $S$, $\sum$ by $\wsum$ and $L_{1,g}$ and $L_{1,g}^{(3)}$ by $\wL_{g}$ and $\wL_{g}^{(3)}$. In particular it shows that \eqref{eq:w def sum} is summable when $H=\cM$ or $H=Z\cdot \cM$.

\subsection{Boundary behaviour of the sums over the monoid $\sum_{\cM,k}\varphi$ and $\wsum_{\cM}\varphi$}

We will consider the boundary behaviour of $\sum_{\cM,k} \varphi$ for $\varphi \in \Oc{k}$.
The following proposition is a generalization of \cite[Prop. 4.1]{MMY3} and explains the behaviour of $\sum_{\cM,k} \varphi$ near $0$.

\begin{prop}\label{pr:decomp}
Let $-1 < \gamma_0 < \gamma_1$ and $I=[\gamma_0, \gamma_1]$.
Also let 
\begin{align}
 U = \{ z\in \C \;|\; \re(z) \in (-\infty,\gamma_0-1]\cup [\gamma_1+1,+\infty)\text{ or }\im(z) \notin (-1/2,1/2)\}.\label{eq: def U}
\end{align}
There exists $C_{I,k} > 0$ such that, for all $\varphi\in\mathcal{O}^{k}(\overline{\mathbb{C}}\backslash I)$ and all $z \in D_0\cup H_0\cup H_0'$, we have
\begin{equation}
\left|T_k\varphi(z)  - \sum_{m =1}^\infty \sum_{n=1}^k \frac{ z^{k-n}}{n!} \varphi^{(n)}(-m) \right| \le C_{I,k} |z|^k\log(1+|z|^{-1})\sup_U|\varphi|.\label{eq: decomp 1st eq}
\end{equation}
Moreover, there exists $C_k>0$ such that, for all $\varphi\in\Oc{k}$ and all $z \in D_0\cup H_0\cup H_0'$, we have
\begin{equation}\label{eq: decomp 2nd eq}
\left|\sum_{\mathcal{M},k}\varphi(z)-\varphi(z) 
 - \sum_{m =1}^\infty \sum_{n=1}^k \frac{ z^{k-n}}{n!} \left(\sum_{\mathcal{M}, k}\varphi\right)^{(n)}(-m) \right| \le C_k |z|^k\log(1+|z|^{-1})\sup_{D_\infty}|\varphi|.
\end{equation}
\end{prop}
We will give the proof of this proposition in much detail as it differs significantly from the case $k=1$.
\begin{proof}
As the proofs of \eqref{eq: decomp 1st eq} and \eqref{eq: decomp 2nd eq} are similar, we will first prove \eqref{eq: decomp 2nd eq} and later give the differences for the proof of \eqref{eq: decomp 1st eq}. 
For the following, we set 
$\widetilde{\varphi}= \sum_{\mathcal{M}, k}\varphi$. 
Then, Proposition~\ref{pr:inv} and Proposition~\ref{pr:Tkest}-\eqref{en: 1} imply
\begin{align*}
 \widetilde{\varphi}=\varphi+T_k\widetilde{\varphi}=\varphi+\sum_{m=1}^{\infty}L_{k,g(m)}\widetilde{\varphi}.
\end{align*}
By the above equation and \eqref{eq:Tk}, the expression of the left hand-side of \eqref{eq: decomp 2nd eq} can thus be summarized as 
\begin{align}
 \MoveEqLeft\left|\sum_{\mathcal{M},k}\varphi(z)-\varphi(z) 
- \sum_{m =1}^\infty \sum_{n=1}^k \frac{z^{k-n}}{n!} \left(\sum_\mathcal{M}\varphi\right)^{(n)}(-m) \right|\notag\\
&=\left|T_k\widetilde{\varphi}(z) - \sum_{m =1}^\infty \sum_{n=1}^k \frac{z^{k-n}}{n!} \widetilde{\varphi}^{(n)}(-m)\right|
=\left| \sum_{m=1}^{\infty}z^k\left(\widetilde{\varphi}\left(\frac{1}{z}-m\right)-\widetilde{\varphi}(-m)\right)\right|.\label{eq: left handside} 
\end{align}
We will split the sum in the following way:
\begin{align}
 \MoveEqLeft\left|\sum_{\mathcal{M},k}\varphi(z)-\varphi(z) 
 \sum_{m =1}^\infty \sum_{n=1}^k \frac{z^{k-n}}{n!} \left(\sum_{\mathcal{M},k}\varphi\right)^{(n)}(-m) \right|\notag\\
 &\le 
 \left| \sum_{m=1}^{\left\lfloor 3|z|^{-1}\right\rfloor+1}z^k\left(\widetilde{\varphi}\left(\frac{1}{z}-m\right)-\widetilde{\varphi}(-m)\right)\right|\label{eq: split sum 1}\\
 &\qquad+\left|\sum_{m=\left\lfloor 3|z|^{-1}\right\rfloor+2}^{\infty} \sum_{n=1}^k \frac{z^{k-n}}{n!} \widetilde{\varphi}^{(n)}(-m)\right|\label{eq: split sum 2}\\
  &\qquad + \left|  \sum_{m=\left\lfloor 3|z|^{-1}\right\rfloor+2}^{\infty}z^k\left(\widetilde{\varphi}\left(\frac{1}{z}-m\right)-\widetilde{\varphi}(-m)\right) 
 - \sum_{m =\left\lfloor 3|z|^{-1}\right\rfloor+2}^{\infty} \sum_{n=1}^k \frac{z^{k-n}}{n!} \widetilde{\varphi}^{(n)}(-m)\right|. \label{eq: split sum 3}
\end{align}

We first estimate the first summand \eqref{eq: split sum 1}.
We note that for $m\in\N_{\le \left\lfloor 3|z|^{-1}\right\rfloor+1}$ and $z\in D_0\cup H_0\cup H_0'$ we have that $-m\in D_\infty$ and $1/z-m\in D_{\infty}$.

Applying then \eqref{eq:O^k bound} with $\C\backslash \overline{V}=D_{\infty}$ yields
\begin{align*}
 \MoveEqLeft\left|\sum_{m=1}^{\left\lfloor 3|z|^{-1}\right\rfloor+1} z^k\left(\widetilde{\varphi}\left(\frac{1}{z}-m\right)-\widetilde{\varphi}(-m)\right)\right|\notag\\
 &\le \sum_{m=1}^{\left\lfloor 3|z|^{-1}\right\rfloor+1} c_{1,k} \left|z\right|^k\cdot \left(\frac{1}{m^k}+\left|-m+\frac{1}{z}\right|^{-k}\right)\cdot \sup_{D_{\infty}}\left|\widetilde{\varphi}\right|\notag\\
 &\le c_{2,k}\left|z\right|^k\cdot \left(\sum_{m=1}^{\left\lfloor 3|z|^{-1}\right\rfloor+1}  \left(\frac{1}{m}+\left|-m+\frac{1}{z}\right|^{-1}\right)\right) \cdot \sup_{D_{\infty}}\left|\widetilde{\varphi}\right|
\end{align*}
with $c_{2,k}=c_{1,k} \sup_{z\in D_0\cup H_0\cup H_0', m\in \N}\left|-m+\frac{1}{z}\right|^{-k+1}=c_{1,k}\, 2^{k-1}$.
Noting that 
for all $z\in D_0\cup H_0\cup H_0'$ we have
\begin{align*}
 \MoveEqLeft\sum_{m=1}^{\left\lfloor 3|z|^{-1}\right\rfloor+1}  \frac{1}{m} 
 \le 1+ \int_1^{\left\lfloor 3|z|^{-1}\right\rfloor+1} \frac{1}{m}\, \mathrm{d}m
 \le 1+ \log( \left\lfloor 3|z|^{-1}\right\rfloor+1),
\end{align*}
and since $\left|\frac{1}{z}-\left(\left\lfloor\re\left(\frac{1}{z}\right)\right\rfloor-n\right)\right| \ge \left|\frac{1}{z}-\left(\left\lfloor\re\left(\frac{1}{z}\right)\right\rfloor+n\right)\right|$ for $1\le n \le \lfloor \re(1/z)\rfloor$ and $\re(z)\ge0$,
we have
\begin{align*}
\sum_{m=1}^{\left\lfloor 3|z|^{-1}\right\rfloor +1} \left|m-\frac{1}{z}\right|^{-1} 
& = \sum_{\substack{\{ m\colon 1\le m\le \lfloor 3|z|^{-1}\rfloor+1, \\ \left|m- \lfloor \re(1/z)\rfloor\right|\le 2\}}} \left|m-\frac{1}{z}\right|^{-1}
+ \sum_{\substack{\{ m\colon 1\le m\le \lfloor 3|z|^{-1}\rfloor+1, \\ |m- \lfloor \re(1/z)\rfloor|> 2\}}} \left|m-\frac{1}{z}\right|^{-1} \\
&  \le 10 + 2 \sum_{m=\lfloor \re(1/z)\rfloor+3}^{\left\lfloor 3|z|^{-1}\right\rfloor+1}\frac{1}{m-\lfloor\re(1/z)\rfloor -1} \\
& \le 10 + 2 \int_{m=1}^{\left\lfloor 3|z|^{-1}\right\rfloor-\lfloor \re(1/z)\rfloor}\frac{1}{m}\,\mathrm{d}m 
 \le 10 + 2 \log(\left\lfloor 3|z|^{-1}\right\rfloor - \lfloor\re(1/z)\rfloor).  
\end{align*}
The last two inequalities even hold true if $\re(1/z)<0$.
The fact that $\lfloor 3|z|^{-1}\rfloor+1 \le (|z|^{-1}+1)^3$ implies that there exists another constant $c_{3,k}>0$ such that, for all $\widetilde{\varphi}\in \Oc{k}$ and $z\in D_0\cup H_0\cup H_0'$, we have
\begin{align}
\left|\sum_{m=1}^{\left\lfloor 3|z|^{-1}\right\rfloor+1}z^{k}\left(\widetilde{\varphi}\left(\frac{1}{z}-m\right)-\widetilde{\varphi}(-m)\right)\right|
 &\le c_{3,k} \left|z\right|^k\cdot\log \left(1+\left|z\right|^{-1}\right)\cdot \sup_{D_{\infty}}\left|\widetilde{\varphi}\right|.\label{eq: estim sum 2}
\end{align}

For the second summand \eqref{eq: split sum 2}, from \eqref{eq:jth derivative}, for all $\varphi\in \Oc{k}$ and $z\in D_0\cup H_0\cup H_0'$ we have
\begin{align*}
 \left|\sum_{m=\left\lfloor 3|z|^{-1}\right\rfloor+2}^{\infty} \sum_{n=1}^k \frac{ z^{k-n}}{n!} \widetilde{\varphi}^{(n)}(-m)\right|
 &\le \sum_{m=\left\lfloor 3|z|^{-1}\right\rfloor+2}^{\infty} \sum_{n=1}^k \frac{|z|^{k-n}}{n!}\cdot c_{k,n}'n!\cdot \frac{1}{m^{k+n}}\cdot \sup_{D_{\infty}}\left|\widetilde{\varphi}\right|\notag\\
 &\le \sum_{m=\left\lfloor 3|z|^{-1}\right\rfloor+2}^{\infty}  |z|^{k-1}\cdot 2^{k-1}c_{k,n}' \cdot \frac{1}{m^{k+1}}\cdot \sup_{D_{\infty}}\left|\widetilde{\varphi}\right|.
\end{align*} 
Noting that 
\begin{align}
 \sum_{m=\left\lfloor 3|z|^{-1}\right\rfloor+2}^{\infty} m^{-k-1}
 < \int_{\left\lfloor 3|z|^{-1}\right\rfloor+1}^{\infty}m^{-k-1}\mathrm{d}m 
 = {k^{-1}} \left(\left\lfloor 3|z|^{-1}\right\rfloor+1\right)^{-k}
 < \left(\frac{|z|}{3}\right)^{k}\label{eq: sum integral estim}
\end{align}
yields that for all $\varphi\in \Oc{k}$ and $z\in D_0\cup H_0\cup H_0'$ we have
\begin{align} 
 \left|\sum_{m=\left\lfloor 3|z|^{-1}\right\rfloor+2}^{\infty} \sum_{n=1}^k \frac{z^{k-n}}{n!}\, \widetilde{\varphi}^{(n)}(-m)\right|
 &\le  |z|^{2k-1}\cdot c_{k,n}'\cdot \sup_{D_{\infty}}\left|\widetilde{\varphi}\right|.\label{eq: estim sum 3}
\end{align}
Finally, we estimate the last summand \eqref{eq: split sum 3} using \eqref{eq: Lkg(m)phi1} 
\begin{align}
 \MoveEqLeft\left|\sum_{m=\left\lfloor 3|z|^{-1}\right\rfloor+2}^{\infty}z^{k}\left(\widetilde{\varphi}\left(\frac{1}{z}-m\right)-\widetilde{\varphi}(-m)\right)-\sum_{m = \left\lfloor 3|z|^{-1}\right\rfloor +2}^\infty \sum_{n=1}^k \frac{z^{k-n}}{n!} \widetilde{\varphi}^{(n)}(-m)\right|\notag\\
 &\le \sum_{m=\left\lfloor 3|z|^{-1}\right\rfloor+2}^{\infty}\left|\frac{1}{z\, k!}\int_{0}^{1}\widetilde{\varphi}^{(k+1)}\left(\frac{t}{z}-m\right)(1-t)^k\mathrm{d}t\right|.\label{eq: split sum 3 estim}
\end{align}
Using Cauchy's integral formula yields for $R$ sufficiently small (possibly depending on $m$ and $z$) to be defined later that 
\begin{align*}
 \MoveEqLeft\left|\int_{0}^{1}\widetilde{\varphi}^{(k+1)}\left(\frac{t}{z}-m\right)(1-t)^k\mathrm{d}t\right|\\
 &\le \left|\int_0^1 \frac{(k+1)!}{2\pi i}\int_{\left|\omega-\left(\frac{t}{z}-m\right)\right|=R}\frac{\widetilde{\varphi}\left(\omega\right)}{\left(\omega-\left(\frac{t}{z}-m\right)\right)^{k+2}}\mathrm{d}\omega\cdot (1-t)^k\mathrm{d}t\right|\\
 &\le \sup_{\left|\omega-\left(\frac{t}{z}-m\right)\right|=R} \big|\widetilde{\varphi}\left(\omega\right)\big|\cdot \frac{(k+1)!}{R^{k+2}}\int_0^1(1-t)^k\mathrm{d}t
 =\sup_{\left|\omega-\left(\frac{t}{z}-m\right)\right|=R}\big|\widetilde{\varphi}\left(\omega\right) \big|\cdot\frac{2\pi k!}{R^{k+1}}.
\end{align*}
Now, we specify the radius $R$. It is sufficient that 
$-m+t/|z|+R<\sqrt{3}/2-1$ is fulfilled, i.e.\ it is sufficient that 
$R< \sqrt{3}/2-1+m-1/|z|$ is fulfilled and thus we can choose $R= m/3$. 
Hence, using the above estimate and inserting it into \eqref{eq: split sum 3 estim} yields 
for all $\varphi\in \Oc{k}$
\begin{align}
\MoveEqLeft\left|\sum_{m= \left\lfloor 3|z|^{-1}\right\rfloor+2}^{\infty}z^k\left(\widetilde{\varphi}\left(\frac{1}{z}-m\right)-\widetilde{\varphi}(-m)\right)-\sum_{m =\left\lfloor 3|z|^{-1}\right\rfloor +2}^\infty \sum_{n=1}^k \frac{z^{k-n}}{n!} \widetilde{\varphi}^{(n)}(-m)\right|\notag\\
  &\le \sum_{m=\left\lfloor 3|z|^{-1}\right\rfloor+2}^{\infty} 2\pi\cdot 3^{k+1}  \cdot|z|^{-1}\cdot m^{-k-1}\cdot \sup_{\left|\omega-\left(\frac{t}{z}-m\right)\right|= m/3} \left|\widetilde{\varphi}\right|.\label{eq: split sum 3 estim sum}
\end{align}
Next, we further estimate $\sup_{\left|\omega-\left(\frac{t}{z}-m\right)\right|=m/3} \left|\widetilde{\varphi}\right|$. Using \eqref{eq:O^k bound} and noting that we have
$D_{z,m,t}:=\{\omega: \left|\omega-\left(\frac{t}{z}-m\right)\right|= m/3\}\subset D_{\infty}$
and thus $\inf_{\omega\in D_{z,m,t}}|\omega| \geq \left|\frac{t}{z}-m\right|-\frac{m}{3}\geq m- |z|^{-1}-\frac{m}{3}\geq \frac{m}{3}$ yields that there exists $c_{4,k}$ such that 
\begin{align}
 \sup_{D_{z,m,t}} \left|\widetilde{\varphi}\right|
 &\le c_{4,k}\,\sup_{D_{\infty}} \left|\widetilde{\varphi}\right|\cdot \sup_{D_{z,m,t}} \left|\omega\right|^{-k}
 \le c_{4,k}\,\sup_{D_{\infty}} \left|\widetilde{\varphi}\right|\cdot \left(\frac{m}{3}\right)^{-k}\label{eq: max princ1}
\end{align}
and for $m\geq \left\lfloor 3|z|^{-1}\right\rfloor+2$, we have 
\begin{align*}
 \sup_{D_{z,m,t}} \left|\widetilde{\varphi}\right|
 &\le c_{4,k} \,\sup_{D_{\infty}} \left|\widetilde{\varphi}\right|\cdot \left|z\right|^k.
\end{align*}
Combining this with \eqref{eq: split sum 3 estim sum} and \eqref{eq: sum integral estim} yields
\begin{align}
\MoveEqLeft\left|\sum_{m=\left\lfloor 3|z|^{-1}\right\rfloor+2}^{\infty}z^k\left(\widetilde{\varphi}\left(\frac{1}{z}-m\right)-\widetilde{\varphi}(-m)\right)-\sum_{m =1}^\infty \sum_{n=1}^k \frac{z^{k-n}}{n!} \widetilde{\varphi}^{(n)}(-m)\right|\notag\\
  &\le   2 \pi c_{4,k}\cdot 3^{k+1} \cdot|z|^{-1}\cdot \left(\frac{|z|}{3}\right)^{2k} \cdot \sup_{D_{\infty}} \left|\widetilde{\varphi}\right|
  \le  2\pi c_{4,k}\cdot |z|^{2k-1}\cdot \sup_{D_{\infty}}\left|\widetilde{\varphi}\right|.\label{eq: estim sum 1}
\end{align}

Combining the sum of \eqref{eq: split sum 1}, \eqref{eq: split sum 2}, and \eqref{eq: split sum 3} estimated in  \eqref{eq: estim sum 2}, \eqref{eq: estim sum 3}, and \eqref{eq: estim sum 1}, we obtain that there exists a constant $C_k>0$ such that for all $\varphi\in \Oc{k}$ and $z\in D_0\cup H_0\cup H_0'$ we have
the statement of \eqref{eq: decomp 2nd eq}.

The left handside of \eqref{eq: left handside} equals the left handside of \eqref{eq: decomp 1st eq} if we replace $\varphi$ by $\widetilde{\varphi}$. 
All the remaining calculations remain the same with the difference that we have to take the supremum over $U$ instead of $D_{\infty}$. However, there are only two instances where this is important. The first is the choice of $R$ which has to fulfill $R< \gamma_0+m-1/|z|$ instead of $R< \sqrt{3}/2-1+m-1/|z|$ and for both $R\le m/3$ is sufficient. The second is the occurrence of $\inf_{x\in D_{\infty}} \left|x\right|^{-k}$ in \eqref{eq: max princ1} which we were estimating by $c_{1,k}$ and which has to be changed to an estimate by a constant depending on $I$ and $k$.
\end{proof}

Our next concern will be the behaviour of $\sum_{\cM,k}\varphi$ for $\varphi\in\Oc{k}$ near $1$.
The following proposition is a generalization of \cite[Prop. 4.2]{MMY3}.

\begin{prop}\label{pr:decomp2}
There exists $C_k>0$ such that for all $\varphi\in\Oc{k}$ and all $z \in D_1$ we have
\begin{equation}
\left|T_k\varphi(z) + z^k \varphi\left(\frac{1}{z}-1\right) 
- \sum_{m =1}^\infty \sum_{n=1}^k \frac{z^{k-n}}{n!} \varphi^{(n)}(-m) \right| \le C_k |z-1|\sup_{D_\infty}|\varphi|.\label{eq: Tk on D1}
\end{equation}
\end{prop}

\begin{proof}
We first notice that \eqref{eq:Tk} implies
\begin{align}
\MoveEqLeft \left|T_k\varphi(z) + z^k\varphi\left(\frac{1}{z}-1\right) 
- \sum_{m =1}^\infty \sum_{n=1}^k \frac{z^{k-n}}{n!}\varphi^{(n)}(-m) \right|\notag\\
&=\left|\sum_{m=2}^{\infty} z^k\varphi\left(\frac{1}{z}-m\right)-   \sum_{m=1}^{\infty} z^k \varphi(-m)\right|\notag\\
&\le |z|^k\cdot \left|\sum_{m=2}^{\infty} \left(\varphi\left(\frac{1}{z}-m\right)-\varphi(1-m)\right)\right|.\label{eq: Tk on D1 estim}
\end{align}

For $z \in D_1$, we have $1/z - 1 \in D_0$. 
Equation \eqref{eq:jth derivative} implies that there exists a constant $c_{1,k}>0$ such that 
$|\varphi'(w)|
 \le c_{1,k} m^{-k-1}\sup_{D_{\infty}}|\varphi|$.
This implies 
\begin{align*}
 \MoveEqLeft \left|T_k\varphi(z)+  z^k\varphi\left(\frac{1}{z}-1\right) 
 - \sum_{m =1}^\infty \sum_{n=1}^k \frac{z^{k-n}}{n!}\varphi^{(n)}(-m) \right|\notag\\
& \le |z|^k\cdot \sum_{m=2}^{\infty} c_{1,k} m^{-k-1}\sup_{D_{\infty}}|\varphi|\cdot \left|\frac{1}{z}-1\right|
 = |z|^{k-1}\cdot \left|1-z\right|\sup_{D_{\infty}}|\varphi| \sum_{m=2}^{\infty} c_{1,k} m^{-k-1}\\
&= C_k \cdot \left|1-z\right|\sup_{D_{\infty}}|\varphi|,
\end{align*}
where $C_k \coloneqq \sup_{z\in D_1}|z|^{k- 1} c_{1,k} \sum_{m=2}^{\infty} m^{-k-1}$. 
\end{proof}

We can finally use the previous results and the complex continued fraction algorithm to study the behaviour of $T_k^n\varphi(z)$
when $z$ is close to the boundary. We recall that $D(m_1,\cdots, m_n)$ is the set of $z_0\in D$ whose first $n$ complex continued fraction 
entries equal $\{m_i\}_{i=1}^n$.
Next, we will prove a proposition which is an analog to \cite[Prop. 4.4]{MMY3}.
\begin{prop}\label{prop: main prop kBrjuno}
 For all $\varphi\in\Oc{k}$, $n\in\N$, and $z_0\in D(m_1,\cdots, m_n)$ we have
\begin{multline}\label{eq: Tkn varphi}
 T_k^n\varphi(z_0)= (-1)^{n(k+1)}(p_{n-1}-q_{n-1}z_0)^{k}(\varphi(z_n)+\varphi(z_n-1)+\varepsilon_n\varphi(z_n+1))\\
- (-1)^{(n-1)(k+1)} (p_{n-2}-q_{n-2}z_0)^{k}(1+z_{n-1})^{k}\varepsilon_{n-1}\varphi\left(-\frac{z_{n-1}}{1+z_{n-1}}\right)+R^{[n]}_{k}\varphi(z_0).
\end{multline}
The remainder term $R^{[n]}_k$ is holomorphic in the interior of $D(m_1,\cdots, m_n)$ and continuous in $D(m_1,\cdots, m_n)$ and there exists $C_k>0$ such that for all $\varphi\in\Oc{k}$, $n\in\N$, and $z_0\in D(m_1,\cdots, m_n)$  we have
\begin{equation}\label{eq:R bound}
\left|R^{[n]}_{k}\varphi(z_0)\right| \le C_k\,n \left(\frac{\sqrt{5}-1}{2}\right)^{nk}\sup_{D_\infty}|\varphi|.
\end{equation}
\end{prop}

\begin{proof}
We start by an estimate for $n=1$.
For $z_0 \in D(m_1)$, by definition, we have
\begin{align*}
  T_k\varphi(z_0)
  &=  - \sum_{m=1}^{\infty} z_0^k\left(\varphi\left(\frac{1}{z_0}-m\right)-\varphi(-m)\right) + \sum_{m=1}^{\infty}\sum_{j=1}^{k}\frac{z_0^{k-j}}{j!}\varphi^{(j)}(-m)\\
&= -z_0^k(\varphi(z_1)+\varphi(z_1-1)+\varepsilon_1\varphi(z_1+1)) + R_{k}^{(m_1)}(\varphi)(z_0),
\end{align*}
with
\begin{align*}
 R^{(m_1)}_{k}(\varphi)(z_0)
 &= \sum_{\substack{m\geq 1\\|m-m_1|\le 1}}z_0^k \varphi(-m) 
  + \sum_{m=1}^\infty \sum_{j=1}^{k}\frac{z_0^{k-j}}{j!}\varphi^{( j)}(-m)\\
&\qquad - \sum_{\substack{m\geq 1\\|m-m_1|> 1}}{ z_0^k}(\varphi(m_1+z_1-m)-\varphi(-m))\\
& =: I_0 + I_1 - I_2.
\end{align*}
Then, $R^{(m_1)}_{k}(\varphi)$ is holomorphic in a neighbourhood of $D(m_1)$.

We will first show that there is a constant $c_{1,k}$ such that 
\begin{equation}\label{eq:R^(m_1) bound}
|R^{(m_1)}_{k}(\varphi)(z_0)|\le {c_{1,k}}\sup_{D_\infty}|\varphi| \quad \text{ for all }z_0\in D(m_1).
\end{equation}
Note that this however is not an estimate of \eqref{eq: Tkn varphi} for $n=1$ as the terms in the second line are all contained in the error term $R^{(m_1)}_{k}(\varphi)(z_0)$.
We start by estimating $I_0$. 
Since  $|z_0|\le 1$ and $I_0 = z_0^k\big(\varphi(-m_1-1)+\varphi(-m_1)+\varepsilon_1\varphi(-m_1+1)\big)$,
we have 
\begin{equation}\label{eq: I0 estim}
 |I_0|\le 3 \sup_{D_\infty}|\varphi|.
\end{equation}
From \eqref{eq:jth derivative} with $\{\omega:|\omega+m|=m-0.8\}\subset D_\infty$, 
we have $$|\varphi^{(j)}(-m)|\le {c_{k,j}'}\frac{j!\sup_{D_\infty}|\varphi|}{(m-0.8)^{k+j}}.$$
Let ${c_{2,k}} = \sum_{m=1}^\infty \sum_{j=1}^k \frac{c_{k,j}'}{(m-0.8)^{k+j}}$.
Since 
$\sum_{m=1}^\infty\frac{1}{(m-0.8)^{k+j}}<\infty$ for $k,j\ge 1,$
$${c_{2,k}} =  \sum_{m=1}^\infty \sum_{j=1}^k \frac{c_{k,j}'}{(m-0.8)^{k+j}} =  \sum_{j=1}^k\sum_{m=1}^\infty \frac{c_{k,j}'}{(m-0.8)^{k+j}}< \infty.$$
Then we have
\begin{equation}\label{eq:I1}
|I_1| \le
c_{2,k} \sup_{D_\infty}|\varphi|.
\end{equation}

Since $z_0\in D(m_1)$, we have $1/z_0 = z_1+m_1\in \Delta+m_1$ (see \eqref{eq: def Delta} for the definition of $\Delta$).
Furthermore,
$\frac{1}{m_1+1}\le |z_0| \le \frac{1}{m_1-1}$ if $m_1\not =1$ and
$1/2\le |z_0| \le 1$ if $m_1=1$.
Let us consider
$$I_2' := \sum_{\substack{m \ge m_1 + 2}}z_0^k (\varphi(1/z_0-m)-\varphi(-m)),$$
and $I_2'' := I_2 - I_2'$. Then, $I_2'' = 0$ if $m_1\le 2$, and
$$I_2'' = \sum\limits_{1 \le m \le m_1- 2}z_0^k (\varphi(1/z_0-m)-\varphi(-m)) \quad \text{if }m_1\ge 3.$$
If $m\ge m_1+2$, then $\re(1/{z_0}-m)\le -1$. 
For $z$ in the segment $[-m, 1/{z_0}-m]\subset D_\infty$, the circle $\{\omega: |\omega-z|=m-0.8\}$ is contained in $D_\infty$.
Thus, by \eqref{eq:jth derivative}, we have 
$|\varphi'(z)|\le c_{k,1}' \frac{\sup_{D_\infty}|\varphi|}{(m-0.8)^{k+1}}$
and thus
$$|I_2'| \le \sum_{m\ge m_1+2}  \frac{c_{k,1}'}{(m-0.8)^{k+1}} \sup_{D_\infty}|\varphi|.$$
If $m_1\ge 3$, then $1/z_0-m\in \Delta+(m_1-m)\subset D_\infty$ for $1\le m \le m_1-2$. 
Thus, $$|I''_2|\le \frac{2(m_1-2)\sup_{D_\infty}|\varphi|}{(m_1-1)^k}\le 2\sup_{D_\infty}|\varphi|.$$
Then, $|I_2|\le {c_{3,k}}\sup_{D_\infty}|\varphi|$, where
${c_{3,k}} := \sum_{m=3}^\infty \frac{c_{k,1}'}{(m-0.8)^{k+1}} + 2.$
Hence, by combining the last estimate with \eqref{eq: I0 estim} and \eqref{eq:I1} and setting $c_{1,k} := c_{2,k}+c_{3,k}+3$, we obtain \eqref{eq:R^(m_1) bound}.
\bigskip

In the following we will estimate $T_k^n\varphi(z_0)$ itself which we split into a number of summands to be estimated separately.
Recall $z_i$ as in \eqref{eq:CF} and \eqref{eq:z_i}.
Iterating $T_k$ $n$ times, we get 
\begin{align*}
 T_k^n\varphi(z_0)&=(\varphi(z_n)+\varphi(z_n-1)
+\varepsilon_n\varphi(z_n+1))\prod_{i=0}^{n-1}(-z_i^k)\\
&\qquad +\sum_{j=1}^{n-1}T_k^{n-j}\varphi(z_j-1)\prod_{i=0}^{j-1}(-z_i^k)\\
&\qquad + \sum_{j=1}^{n-1}\varepsilon_jT_k^{n-j}\varphi(z_j+1)\prod_{i=0}^{j-1}(-z_i^k)\\
&\qquad+\sum_{j=1}^nR^{(m_j)}_{k}(T_k^{n-j}\varphi)(z_{j-1})\prod_{i=0}^{j-2}(-z_i^k)\\
& = : J_1+J_2+J_3+J_4,
\end{align*}
where $z_{-1}=1$. 
We recall \eqref{eq:prod z_i} and \eqref{eq:dirichlet}, i.e., $\prod_{i=0}^{n-1}(-z_i^k) =(-1)^{n(k+1)} (p_{n-1}-q_{n-1}z_0)^{k}$ and $|p_{n-1}-q_{n-1}z_0|<q_{n}^{-1}$.
Thus $J_1$ is the first term of \eqref{eq: Tkn varphi}. In the following steps (1)-(3) we estimate the summands $J_2$, $J_3$, and $J_4$.

\begin{enumerate}
\item We continue by estimating $J_4$.
By using \eqref{eq:prod z_i}, \eqref{eq:dirichlet} and \eqref{eq:R^(m_1) bound} and
by applying Proposition~\ref{pr:Tkbound}-\eqref{pr:Tkbound-item2} to each of the summands and noting that $q_j\geq \left(\frac{\sqrt{5}-1}{2}\right)^{-j}$, we obtain 
\begin{align}
\left|\sum_{j=1}^nR^{(m_j)}_{k}(T_k^{n-j}\varphi)(z_{j-1})\prod_{i=0}^{j-2}(-z_i^k)\right|
&\le \sum_{j=1}^n \frac{c_{1,k} \sup_{D_\infty} |T_k^{n-j}\varphi|}{q_{j-1}^k} \notag\\
& \le c_{1,k} \sum_{j=1}^n\overline{C}_{0,k}\left(\frac{\sqrt{5}-1}{2}\right)^{(n-1)k}\sup_{D_\infty} |\varphi| \notag\\
& \le  c_{4,k} \cdot n \left(\frac{\sqrt{5}-1}{2}\right)^{nk}\sup_{D_\infty} |\varphi|\label{eq:bdd}
\end{align}
with 
$ c_{4,k} = c_{1,k} \overline{C}_{0,k}\left( \frac{\sqrt{5}+1}{2}\right)^{k}.$

\item
In the next steps, we will show that $|J_2|\le c_{5,k} n(\frac{\sqrt{5}-1}{2})^{nk}\sup_{D_\infty}|\varphi|$ for some $c_{5,k}$.
If $z_j-1\not\in  D_0$ for $1\le j < n$, then we have
\begin{equation}\label{eq:bddz-1}
|T_k^{n-j}\varphi(z_j-1)| \le C_{K,k} \sup_{D_\infty} |(T_k^{n-j-1}\varphi)|
\end{equation}
by Proposition~\ref{pr:Tkest}-\eqref{en: 1} with $I=[0,1]$ and 
$K = \overline{(D - 1) \setminus D_0}$. 

On the other hand, let us consider the case of $z_j-1\in D_0$. 
From \eqref{eq: decomp 1st eq}, we deduce that
\begin{align*}
|T_k^{n-j}\varphi(z_j-1)| & \le C_{I,k} |z_j-1|^k\log(1+|z_j-1|^{-1}) \sup_{D_\infty} |(T_k^{n-j-1}\varphi)| \\
		&+\left|\sum_{m =1}^\infty \sum_{i=1}^k \frac{(z_j-1)^{k-i}}{i!} (T_k^{n-j-1}\varphi)^{(i)}(-m)\right|.
\end{align*}
From \eqref{eq:I1} and $|z_j-1|^k\log(1+|z_j-1|^{-1}) \le \log 2,$ we see that
\begin{equation}\label{eq:bddz-1_2}
|T_k^{n-j}\varphi(z_j-1)| \le (C_{I,k} \log2+c_{2,k})\sup_{D_\infty}|(T_k^{n-j-1}\varphi)|.
\end{equation} 
Let $c_{6,k}= \max\{C_{K, k}, C_{I,k}\log2+c_{2,k}\}$.
Then, by combining the calculations leading to \eqref{eq:bdd}, \eqref{eq:bddz-1} and \eqref{eq:bddz-1_2},
we get
\begin{align*}
\left|\sum_{j=1}^{n-1}T_k^{n-j}\varphi(z_j-1)\prod_{i=0}^{j-1}(-z_i^k)\right| 
 \le c_{5,k}\, n \left(\frac{\sqrt{5}-1}{2}\right)^{nk}\sup_{D_\infty} |\varphi|,
\end{align*}
where $c_{5,k} = \overline{C}_{0,k}\, c_{6,k} \left(\frac{\sqrt{5}+1}{2}\right)^k$.

\item
To complete the proof, it is enough to show that there is a constant $c_{7,k}$ such that
\begin{align}\label{eq:bddz+1_2}
& \Big|J_3 + \prod_{i=0}^{n-2}(-z_i^k)(1+z_{n-1})^k\varepsilon_{n-1}\varphi \Big(-\frac{z_{n-1}}{1+z_{n-1}}\Big)\Big|  \le c_{7,k}\, n \left(\frac{\sqrt{5}-1}{2}\right)^{nk}\sup_{D_\infty} |\varphi|.
\end{align}
We have
\begin{align*}
\MoveEqLeft J_3  + \prod_{i=0}^{n-2} (-z_i^k)(1+z_{n-1})^k\varepsilon_{n-1}\varphi \Big(-\frac{z_{n-1}}{1+z_{n-1}} \Big) \\
& = \sum_{j=1}^{n-2}\veps_j T_k^{n-j}\varphi(z_j+1)\prod_{i=0}^{j-1} (-z_i^k) \\
& \qquad + \left(T_k\varphi(z_{n-1}+1) + (1+z_{n-1})^k\varphi\left(-\frac{z_{n-1}}{1+z_{n-1}}\right)\right)\prod_{i=0}^{n-2}(-z_i^k) \\
& =: I_4 + I _5.
\end{align*}
\begin{enumerate}[(3a)]
\item We will estimate $I_4$.
If $m_{j+1}$ is small enough so that $z_j+1 \not\in D_1$, then by Proposition~\ref{pr:Tkest}-\eqref{en: 1} applied to $I=[0,1]$ and 
$K' = \overline{(D+1)\setminus D_1}$ we have
$$|T_k^{n-j}\varphi(z_j+1)| \le C_{K',k} \sup_{D_\infty} |(T_k^{n-j-1}\varphi)|.$$

If $m_{j+1}$ is large so that $z_j+1 \in D_1$, then by Proposition~\ref{pr:decomp2} and \eqref{eq:I1}, we deduce that 
\begin{align*}
 & |T_k^{n-j}\varphi(z_j+1)| \\
 & \le |1+z_{j}|^{k}\left|(T_k^{n-j-1}\varphi)\left(-\frac{z_{j}}{1+z_{j}}\right)\right|+ C_k |z_{j}|\sup_{D_\infty}|T_k^{n-j-1}\varphi| \\ 
  &\qquad+\left|\sum_{m =1}^\infty \sum_{i=1}^k \frac{ (z_j+1)^{k-i}}{i!}(T_k^{n-j-1}\varphi)^{(i)}(-m)\right| \\
  & \le |1+z_j|^k\left|(T_k^{n-j-1}\varphi)\left(-\frac{z_j}{1+z_j}\right)\right|+(C_k+ c_{2,k} (2\sqrt{3}/3)^k)\sup_{D\infty}|T_k^{n-j-1}\varphi|.
\end{align*}
Since $-\frac{z_j}{1+z_j}\in D_0$ and $j\le n-2$, by Proposition~\ref{pr:decomp}, we have
\begin{align*}
 \MoveEqLeft|1+z_j|^k\left|(T_k^{n-j-1}\varphi)\left(-\frac{z_j}{1+z_j}\right)\right|\\
& \le C_{I,k} |1+z_j|^k \left|-\frac{z_j}{1+z_j}\right|^k\log\left(1+\left|-\frac{z_j}{1+z_j}\right|^{-1}\right)\sup_{D_\infty}|T_k^{n-j-2}\varphi| \\
& \qquad + |1+z_j|^k\left|\sum_{m =1}^\infty \sum_{i=1}^k \frac{(-z_j/(1+z_j))^{k-i}}{i!}(T_k^{n-j-2}\varphi)^{(i)}(-m)\right|.
\end{align*}
Note that
\begin{align*}
 \MoveEqLeft|1+z_j|^k \left|-\frac{z_j}{1+z_j}\right|^k\log\left(1+\left|-\frac{z_j}{1+z_j}\right|^{-1}\right)\\&= |z_j|^k\log\left(1+\left|1+\frac{1}{z_j}\right|\right)  \le |z_j|^k \log\left(2+\frac{1}{|z_j|}\right).
\end{align*}
Since $x^k\log(2+1/x)\to 0$ as $x\to 0$, it is increasing and $|z_j|<1$, it follows that $|z_j|^k\log(2+1/|z_j|)\le \log 3$.
Thus, by a similar way to the proof of \eqref{eq:bddz-1_2}, it follows that
$$ |1+z_j|^k\left|(T_k^{n-j-1}\varphi)\left(-\frac{z_j}{1+z_j}\right)\right| 
 \le \left(C_{I,k} \log3 + c_{2,k}\left({2\sqrt{3}}/{3}\right)^k\right) \sup_{D_\infty}|(T_k^{n-j-2}\varphi)|.
$$ 
By letting $c_{8,k} = \max\big\{C_{K',k},~C_k+c_{2,k}\left(2\sqrt{3}/3\right)^k, C_{I,k}\log 3 + c_{2,k}(2\sqrt{3}/3)^k\big\}$, we have
\begin{equation}\label{eq:bddI4}\begin{split}
& \left| I_4\right| 
 \le c_{8,k} \sum_{j=1}^{n-2} \frac{\sup_{D_\infty}|T_k^{n-j-1}\varphi| + \sup_{D_\infty}|T_k^{n-j-2}\varphi|}{q_{j}^k}.
\end{split}\end{equation}

\item
We will estimate $I_5$.
If $z_{n-1}+1\not\in D_1$, then 
$\re\left(\frac{1}{(z_{n-1}+1)}-1\right)<\frac{\sqrt{3}}{2}-1$.
Thus, $\frac{1}{(z_{n-1}+1)}-1 \in D_\infty$.
By Proposition~\ref{pr:Tkest}-\eqref{en: 1}, we have
$$\left|T_k\varphi(z_{n-1}+1)  + (1+z_{n-1})^{k}\varphi\left(-\frac{z_{n-1}}{1+z_{n-1}}\right)\right|
\le (C_{K',k}+2^k)\sup_{D_\infty}|\varphi|.$$
By Proposition~\ref{pr:decomp2} in the same manner as showing \eqref{eq:I1}, we obtain that $z_{n-1}+1\in D_1$ implies
\begin{align*}
\MoveEqLeft \left|T_k\varphi(z_{n-1}+1)  + (1+z_{n-1})^{k}\varphi\left(-\frac{z_{n-1}}{1+z_{n-1}}\right)\right| \\
& \le C_k |z_{n-1}|\sup_{D_\infty}|\varphi|+\left|\sum_{m =1}^\infty \sum_{i=1}^k \frac{(z_{n-1}+1)^{k-i}}{i!} \varphi^{(i)}(-m)\right|  \\
& \le \big(C_k + c_{2,k}(2\sqrt{3}/3)^k \big) \sup_{D_\infty}|\varphi|.
\end{align*}
Let $ c_{9,k} =\max\{C_{K',k}+2^k,~ C_k+c_{2,k}(2\sqrt{3}/3)^k\}$.
Then, we have
\begin{equation}\label{eq:bddI5}
\left| I_5 \right|
\le \frac{c_{9,k} \sup_{D_\infty}|\varphi|}{q_{n-1}^k}.
\end{equation}
\end{enumerate}
Letting $c_{10,k} = \max\{c_{8,k},~ c_{9,k}\}$, from \eqref{eq:bddI4}, \eqref{eq:bddI5} and \eqref{eq:bdd},
we obtain 
$$|I_4|+|I_5| \le 2 c_{10,k}\sum_{j=0}^{n-1} \frac{\sup_{D_\infty}|T_k^{n-j-1}\varphi|}{q_{j-1}^k} \le c_{7,k} n\left(\frac{\sqrt{5}-1}{2}\right)^{nk}\sup_{D_\infty}|\varphi|,$$
where $c_{7,k} = 2 c_{10,k} \overline{C}_{0,k}(\frac{\sqrt{5}+1}{2})^{2k}$.
Note that $c_{7,k}$ depends only on $k$.
\end{enumerate}
\end{proof}

The next proposition is an analog to \cite[Prop.~4.1]{MMY3}.
We can show the following proposition in the same manner of the proof in \cite{MMY3} by using \eqref{eq:Lgint} and Proposition~\ref{pr:Tkbound}-(2) for $k=1$ and the fact that $S=-T$.
\begin{prop}\label{prop:4.1}
The following two statements hold true:
\begin{enumerate}[(i)]
\item Let $I=[\gamma_0,\gamma_1]$, $\gamma_0>-1$. There exists $C_I>0$ such that for all $\varphi\in\Occ{1}{I}$ and for all $z\in D_0\cup H_0 \cup H_0'$, one has
\begin{equation}
|S\varphi(z)  + \sum_{m\ge 1}\varphi'(-m)|\le C_I |z|\log(1+|z|^{-1})\sup_U|\varphi|,
\end{equation}
where $U$ is as in \eqref{eq: def U}.
\item There exists $C>0$ such that for all $\varphi\in \Oc{1}$ and for all $z\in D_0\cup H_0\cup H_0'$, one has
\begin{equation}
|\wsuml{\mathcal{M}}\varphi(z)-\varphi(z)  + \sum_{m\ge 1}(\wsuml{\mathcal{M}}\varphi)'(-m)|\le C|z|(1+\log|z|^{-1})\sup_{D_\infty}|\varphi|.
\end{equation}
\end{enumerate}
\end{prop}

Concerning the boundary behaviour of the sum leading to the complex Wilton function, we simply observe that since $S=-T$
Propositions 4.1 and 4.4 of \cite{MMY3} hold. The same is true for Proposition 4.11 of \cite{MMY3}, concerning the action of $T=-S$ 
on the space of analytic functions on $\overline{\C}\setminus [0,1]$ with bounded real part. However the Wilton function will not have a bounded real part we thus omit the discussion of it.

\section{Complex $k$-Brjuno function and complex Wilton function}\label{sec: compl functions}

In this section we finally define the complex $k$-Brjuno function and the complex Wilton function.
Let 
\begin{equation}
\varphi_0(z):=-\frac{1}{\pi}\Li{\frac{1}{z}}
\quad\text{ with }\quad 
\Li{z}:=\sum_{n=1}^{\infty}\frac{z^n}{n^2}.\label{eq: def varphi0}
\end{equation}
Note that $\varphi_0\in \Oc{1}$, $\im \varphi_{0}(x)=0$ if $x\in \R\setminus[0,1]$ and
\begin{equation}\label{eq:Imphi0}
\im \varphi_0 (x\pm i0)=\pm\log\left(\frac{1}{x}\right)\quad \text{for}~ x\in(0,1].
\end{equation}

We define $\tau := \big(\begin{smallmatrix}1&1\\0&1\end{smallmatrix}\big)$,
\begin{equation*}\label{eq: def phi1k}
\varphi_{1, k} := \left(L_{k,g(1)}+L_{k,\tau}\right)\varphi_0 \quad \text{ and } \quad
\wvarphi_1 := \left(\wL_{g(1)} + \wL_{\tau}\right)\varphi_0.
\end{equation*}
Note that $\wvarphi_1 = (-L_{1,g(1)}+L_{1,\tau})\varphi_0$.
Further, we note that $\varphi_0\in\Oc{1}$ implies $\varphi_{1,k},~  \wvarphi_1 \in \Occ{1}{[1/2,2]}$. 

Similarly as in \cite[Section 5]{MMY3} we note that from \eqref{eq: sum M sum T}
we have that $\sum_{\mathcal{M},k}\varphi_{ 1, k}= \varphi_{1,k}+\sum_{\mathcal{M},k}T_k\varphi_{ 1, k}$. 
Since
$Z\cM = Z\cM g(1) \sqcup Z\cM\tau,$
we have
\begin{equation}\label{eq: varphi 1k in terms of MZ}
 \sum_{Z\mathcal{M},k}\varphi_0(z)=\sum_{Z,k}[\varphi_{1, k}(z)+\sum_{\mathcal{M},k}(T_k\varphi_{1, k}(z))] = \sum_{Z\cM,k}\varphi_{1,k}(z).
\end{equation}
Similarly, we have
$\wsuml{Z\cM}\varphi_0(z)=\wsuml{Z\cM} \wvarphi_1(z)$. 
We then deduce by simply applying the definitions of $L_{k,g}$ and $\wL_g$
\begin{align*}
 \varphi_{1, k}(z)=& -z^k \varphi_0\left(\frac{1}{z}-1\right) + \varphi_0(z-1) + \sum_{n=0}^k\frac{z^{k-n}}{n!}\varphi_0^{(n)}(-1)
 \end{align*}
and
\begin{equation}\label{eq:wphi_1}
\wvarphi_1(z) = z\left(\varphi_0\left(\frac1z-1\right)-\varphi_0(-1)\right)  + \varphi_0(z-1) -\varphi_0'(-1).
\end{equation}
We deduce that
\begin{equation}\label{eq:defphi1}
 \varphi_{1,k}(z)
= \frac{1}{\pi}\left[{z^k}\Li{\frac{z}{1-z}} - \Li{\frac{1}{z-1}}\right] +\sum_{n=0}^k\frac{z^{k-n}}{n!}\varphi_0^{(n)}(-1),
\end{equation}
and since we additionally have $\varphi_0(-1) = \frac{\pi}{12}$ and $\varphi_0'(-1) = \frac{\log 2}{\pi}$, we also have
\begin{equation}\label{eq:wphi_1 Li}
\wvarphi_1(z) = -\frac{1}{\pi}\left[z\Li{\frac{z}{1-z}}+\Li{\frac{1}{z-1}}\right]-z\,\frac{\pi}{12}-\frac{\log 2}{\pi}.
\end{equation}

By \eqref{eq:Imphi0}, we have
\begin{equation}\label{eq:Im phi1,k on [1/2,2]}
\im\varphi_{1,k}(x \pm i0) = 
\begin{cases}
\pm x^k\log\frac{x}{1-x},&\text{if~}1/2\le x <1, \\
\pm \log\frac{1}{x-1},&\text{if~}1<x\le2
\end{cases}
\end{equation}
and
\begin{equation}\label{eq:Imphi1}
\im\wvarphi_1(x \pm i0)=
\begin{cases}
\mp x\log\frac{x}{1-x} & \text{if}~1/2\le x<1,\\
\pm \log\frac1{x-1} & \text{if}~1<x\le 2.
\end{cases}
\end{equation}

Since $T_k$ is defined exactly in such a way that $T_k\varphi(x)=x^k\, \varphi(1/x)$ as an operator on the space of real functions and the $k$-Brjuno function fulfills the recursion $B_k(x)=-\log x+x^k\, B_k(1/x)$, we have $\left(1-T_k\right)B_k)(x)=-\log x$ and
therefore
\begin{equation}\label{eq:cx brjuno}
 B_k(x)=((1-T_k)^{-1}f)(x) \quad \textnormal{ with } \quad f(x)=\sum_{n\in\Z}\im \varphi_0(x+i0-n).
\end{equation}
We note here that by the above considerations $f(x)=\sum_{n\in\Z}\im \varphi_0(x+i0-n)=\log \{x\}$.

Then, the complex analytic extension of the real $k$-Brjuno function is $(1-T_k)^{-1}\left(\sum_{Z}\varphi_0\right)$ and 
we have $(1-T_k)^{-1}\left(\sum_{Z}\varphi_0\right) = \sum_{Z\cM,k}\varphi_0$ by applying \eqref{eq:(1-Tk)^{-1}}, here $T_k$ is the extended operator as in \eqref{eq:Tk}.
Hence, it is natural to define \emph{the complex $k$-Brjuno function} by 
$$\mathcal{B}_k := \sum_{Z\mathcal{M},k}\varphi_0.$$

By applying the above consideration to $W$ instead of $B_k$ and $S$ instead of $T_k$, 
we may define the complex analytic extension of $W$ by 
\begin{equation}\label{eq:cx w def}
\cW := \wsuml{Z\mathcal{M}}\varphi_0,
\end{equation}
which we call \emph{the complex Wilton function}.

There is a two-to-one correspondence between $Z\cM$ and $\Q$.
From \eqref{eq:Lkg} and \eqref{eq:Lg}, we obtain analogous formulations for $\cB_k$ and $\cW$ respectively such that 
\begin{align}
\cB_k(z) 
& = - \sum_{p/q\in\Q} \det\begin{pmatrix}p'&p\\q'&q\end{pmatrix}^{k+1}\notag\\
&\qquad\Bigg\{\frac{1}{\pi}
\left[ (p'-q'z)^k \Li{\frac{p'-q'z}{qz-p}} - (q''z-p'')^k\Li{\frac{p''-q''z}{qz-p}}\right]\notag\\
&\qquad\quad  + \sum_{n=0}^k \frac{1}{n!}\,\det\begin{pmatrix}p'&p\\q'&q\end{pmatrix}^{n} \notag\\
&\qquad \qquad \quad\Bigg[\frac{(p'-q'z)^{k-n}}{(q')^n}\,\varphi_0^{(n)}\left(-\frac{q}{q'}\right) - \frac{ (q''z-p'')^{k-n}}{(q'')^n}\,\varphi_0^{(n)}\left(-\frac{q}{q''}\right)\Bigg] \Bigg\},\label{eq: complex Brjuno calculation}
\end{align}
where $[\frac{p'}{q'},\frac{p''}{q''}]$ is the Farey interval such that $\frac{p}{q}=\frac{p'+p''}{q'+q''}$ (with the convention $p'=p-1$, $q'=1$, $p''=1$, $q''=0$ if $q=1$).
To obtain the representation as in \eqref{eq: complex kBrjuno} we remember the definition of $\varphi_0$ from \eqref{eq: def varphi0} and note that 
$$\varphi_{0}'(z) = - \frac{1}{\pi} \cdot \frac{1}{z} \log\left(1-\frac{1}{z}\right).$$

If we let $f(z) = \frac{1}{z}$ and $g(z) = \log\left(1-\frac{1}{z}\right)$, then $\varphi_0'(z) = -\frac{1}{\pi}f(z)g(z)$ and 
$$f^{(j)}(z) = \frac{(-1)^j j!}{z^{j+1}} \text{ for }j\ge 0,$$
and
$$g^{(j)}(z) = \frac{(-1)^{j-1}(j-1)!}{z^j}\left(\left(\frac{z}{z-1}\right)^j-1\right) \text{ for } j\ge 1$$
holds.
Since
$$(f g)^{(n)} = \sum_{i=0}^n {n\choose i} f^{(n-i)}g^{(i)},$$
we have, for $n\ge 1$, 
\begin{align*}
\varphi_0^{(n)} (z)
& = -\frac{1}{\pi}\sum_{i=0}^{n-1} {n-1 \choose i}f^{(n-1-i)}g^{(i)} \\
& = -\frac{1}{\pi}\Bigg[  \frac{(-1)^{n-1}(n-1)!}{z^n} \log\left(1-\frac{1}{z}\right) \\
& \qquad\qquad + \sum_{i=1}^{n-1}\frac{(n-1)!}{(n-i-1)! i!}\,\frac{(-1)^{n-1-i}(n-1-i)!}{z^{n-i}}\,\frac{(-1)^{i-1}(i-1)!}{z^i}\notag\\
&\qquad \qquad\qquad\cdot\left(\left(\frac{z}{z-1}\right)^i-1\right)\Bigg] \\
& = -\frac{1}{\pi}\Bigg[  \frac{(-1)^{n-1}(n-1)!}{z^n} \log\left(1-\frac{1}{z}\right) 
+ \frac{(-1)^{n}(n-1)!}{z^n}\sum_{i=1}^{n-1}\frac{1}{i}\left(\left(\frac{z}{z-1}\right)^i-1\right)\Bigg] \\
& = \frac{(-1)^{n}(n-1)!}{\pi z^n} \Bigg[  \log\left(1-\frac{1}{z}\right) 
-\sum_{i=1}^{n-1}\frac{1}{i}\left(\left(\frac{z}{z-1}\right)^i-1\right)\Bigg]
\end{align*}
and by \eqref{eq: complex Brjuno calculation}
we obtain \eqref{eq: complex kBrjuno}.
For the complex Wilton function, we obtain by the same considerations as above the representation as in \eqref{eq: complex Wilton}.

\subsection{Behaviour of $\mathcal{B}_k$ and $\cW$ at rational points}
\subsubsection{Behaviour of $\mathcal{B}_k$ at rational points}
Let
\begin{equation*}
R_k(z) := \sum_{n=0}^k\frac{z^{k-n}}{n!}\varphi_0^{(n)}(-1), 
\end{equation*}
which coincides with the last term in \eqref{eq:defphi1}.
Note that $R_k(1)\in\R$ is well defined for all $k\in\N$, as $\varphi_0 \in C^\infty$. 

\begin{lem}\label{lem:odd}
The function $\varphi_{1,k}{(z)}+i\log(1-z)$ is continuous on $\overline{\uhp} = \uhp \cup \mathbb{R}\cup\{\infty\}$
and its value at $1$ is $R_k(1)+\frac{\pi}{2}$.  
\end{lem}
\begin{proof}
 Recall that for $t\in\C\setminus[0,+\infty)$ we have Euler's functional equation
\begin{equation}\label{eq:Li2}
 \Li{t}+\Li{\frac{1}{t}}=-\frac{1}{2}(\log(-t))^2-\frac{\pi^2}{6}.
\end{equation}
If $z\notin [0,+\infty)$, then $\frac{z}{1-z}, \frac{1}{z-1} \notin [0,+\infty)$. Thus, applying \eqref{eq:Li2} with $t=\frac{z}{1-z}$ and $t=\frac{1}{z-1}$, and substituting it into \eqref{eq:defphi1} gives
\begin{align*}
 \varphi_{1,k}(z)=&
 R_k(z)- z^k\frac{\pi}{6}  + \frac{\pi}{6}
-\frac{1}{\pi}\left(z^k\Li{\frac{1-z}{z}} - \Li{z-1} \right)\\
&-\frac{1}{2\pi}\left( z^k\log^2\left(\frac{- z}{1-z}\right) - \log^2\left(\frac{1}{1-z}\right) \right).
\end{align*}

The function $z^k\Li{\frac{1-z}{z}} - \Li{z-1}$ is continuous on $\overline{\uhp}$ and it is in $O(|z-1|)$ in the neighbourhood of $1$. We then observe that
\begin{align*}
\MoveEqLeft   z^k\log^2\left(\frac{-z}{1-z}\right) - \log^2\left(\frac{1}{1-z}\right)\\
& =( z^k-1)\log^2\left(\frac{-z}{1-z}\right)  +\log^2\left(\frac{-z}{1-z}\right)  -\log^2\left(\frac{1}{1-z}\right),
\end{align*}
and $( z^k-1)\log^2\left(\frac{-z}{1-z}\right) = O\big(|1-z|\log\big(\frac{1}{|1-z|}\big)\big)$ in the neighbourhood of $1$. 
Also, we have
\begin{align*}
 \log^2\left(\frac{- z}{1-z}\right) - \log^2\left(\frac{1}{1-z}\right)
& =  \left(\log(- z)+\log\left(\frac{1}{1-z}\right)\right)^2 - \log^2\left(\frac{1}{1-z}\right) \\
& =  \log^2(- z) + 2\log(- z)\log\left(\frac{1}{1-z}\right).
\end{align*}
In the neighbourhood of $1$ in $\overline{\uhp}$, we have $\log(-z)+i\pi=O(|z-1|)$. Thus,
\begin{equation*}
 \log^2\left(\frac{-z}{1-z}\right) - \log^2\left(\frac{1}{1-z}\right)
=-\pi^2 - 2i\pi\log\left(\frac{1}{1-z}\right)+O\left(|z-1|\, \log\left(|1-z|\right)\right).
\end{equation*}

Therefore, in the neighbourhood of $1$, we have
\begin{equation}\label{eq:varphi 1,k near 1}
\varphi_{1,k}(z)=
R_k(z) + \frac{\pi}{2}
+i\log\left(\frac{1}{1-z}\right)
+O\left(|z-1|\log\left(|1-z|\right)\right)
\end{equation}
and the result follows.
\end{proof}

\begin{cor}\label{cor:odd}
 The real part of $\varphi_{1,k}$ is bounded in $\overline{\C}\setminus\left[1/2,2\right]$. It has an extension to a continuous function on $\overline{\C}\setminus\{1\}$ and
\begin{equation*}
 \lim_{x\to 1^\pm} \re(\varphi_{1,k}(x))=R_k(1)\mp\frac{\pi}{2}.
\end{equation*}
\end{cor}
\begin{proof}
 By \eqref{eq:varphi 1,k near 1}, we have
\begin{align*}
  \lim_{x\to 1^\pm} \re(\varphi_{1,k}(x))&=R_k(1)+\frac{\pi}{2}+\lim_{x\to 1^\pm} \re\left(i\log\left(\frac{1}{1- x}\right)\right)\\
&=R_k(1)+\frac{\pi}{2}-\lim_{x\to 1^\pm} \arg\left(\frac{1}{1- x}\right).
\end{align*}
The result follows from the fact that 
\begin{align*}
&\lim_{x\to 1^+} \arg\left(\frac{1}{1- x}\right)=\pi
\quad\textnormal{ and } \quad
\lim_{x\to 1^-} \arg\left(\frac{1}{1- x}\right)=0.
\end{align*}
\end{proof}

\begin{thm}\label{thm:odd}
The real part of the complex $k$-Brjuno function has a decreasing jump of $\frac{\pi}{q^k}$ at each rational number $\frac{p}{q}$. 
\end{thm}
\begin{proof}
 The proof follows similar arguments as \cite[Section 5.2] {MMY3}.
The space $\widehat{{\uhp}/\mathbb{Z}}$ is defined by
$$\widehat{{\uhp}/\mathbb{Z}}=\uhp/\mathbb{Z} \sqcup(\mathbb{R}\backslash\mathbb{Q})/\mathbb{Z}\sqcup\left(\overline{\mathbb{Q}/\mathbb{Z}}\times\left[-\frac\pi 2, +\frac\pi 2\right]\right),$$
where $\overline{\mathbb{Q}/\mathbb{Z}}=\mathbb{Q}/\mathbb{Z}\cup\{\infty\}$.
The space corresponds to a compactification of $(\overline{\uhp}\backslash\mathbb{Q})/\mathbb{Z}$ by attaching a semicircle on each $q\in\mathbb{Q}\cup\{\infty\}$.
Then the value of a function on $\overline{\mathbb{Q}/\mathbb{Z}}\times[-\pi/2,\pi/2]$ is defined by
$$\varphi(\alpha,\theta) := \lim_{{\substack{z\to\alpha \\ z\in \ell}}}\varphi(z),$$
where $\ell$ is a ray emitting from $\alpha$ and the angle from $i\R$ to $\ell$ in the clockwise direction is $\theta$.
 
We note that the topology induced by $\widehat{\uhp/\mathbb{Z}}$
is the same as the topology induced by $\uhp/\mathbb{Z}\sqcup (\mathbb{R}\backslash\mathbb{Q})/\mathbb{Z}$.
This implies the continuity of $\mathrm{Re}\sum_{Z\mathcal{M},k} \varphi_0$ on $\widehat{{\uhp}/\mathbb{Z}}$
and hence the real part $\mathrm{Re}\sum_{Z\mathcal{M},k} \varphi_0$ of the complex $k$-Brjuno
function is continuous on ${\uhp}/\mathbb{Z}\sqcup (\mathbb{R}\backslash\mathbb{Q})/\mathbb{Z}$ in the usual sense. 
For $(\alpha_0,\pi/2)\in\overline{\mathbb{Q}/\mathbb{Z}}\times[-\pi/2,\pi/2]$, the value $\mathrm{Re}\sum_{Z\mathcal{M},k}\varphi_0 (\alpha_0 , \pi/2)$
(resp.\
$(\alpha_0, -\pi/2)$), with $\alpha_0 \in \mathbb{Q}/\mathbb{Z}$, is the right (resp. left) limit of $\mathrm{Re}\sum_{Z\mathcal{M},k}\varphi_0 (\alpha)$, as
$\alpha\in (\mathbb{R} \backslash \mathbb{Q})/\mathbb{Z}$ tends to $\alpha_0$.
Recalling Corollary \ref{cor:odd}, one has
\begin{align*}
 \mathrm{Re} \varphi_{1,k} (1, \pi/2) - \mathrm{Re} \varphi_{1,k} (1, -\pi/2) = -\pi
\end{align*}
and more precisely, 
by a similar argument as in the proof of Corollary~\ref{cor:odd}, from \eqref{eq:varphi 1,k near 1}, we have
\begin{align*}
 \MoveEqLeft\re\varphi_{1,k}(1,\theta)-\re\varphi_{1,k}(1,0)\\
 &= - \lim_{t\to 0^+}\text{arg}\left(\frac{1}{1-(1+t\mathrm{e}^{i( \pi/2 - \theta)})}\right)
  + \lim_{t\to 0^+}\text{arg}\left(\frac{1}{1-(1+t\mathrm{e}^{i(\pi/2)})}\right)\\
 &= - \lim_{t \to 0^+}\text{arg}\left(t^{ -1} \mathrm{e}^{i( \pi/2 + \theta)}\right)
  + \lim_{t\to 0^+}\text{arg}\left(t^{ -1}\mathrm{e}^{i(\pi/2)}\right)=-\theta.
\end{align*}

If $\alpha_0\in\mathbb{Q}$, $\alpha_0\neq 1$, then
$\mathrm{Re}\varphi_{1,k} (\alpha_0, \theta) = \mathrm{Re} \varphi_{1,k} (\alpha_0 , 0)$
for all $\theta\in[-\pi/2, \pi/2]$. Thus, by \eqref{eq: varphi 1k in terms of MZ} one obtains that for all $p/q \in \mathbb{Q}$, $(p \wedge q = 1)$
\begin{align}\label{eq:jump}
\re\sum_{Z\mathcal{M}, k}\varphi_0(p/q, \theta) =  \mathrm{Re} \sum_{Z\mathcal{M}, k} \varphi_0 (p/q, 0) - \theta/{q^k}.
\end{align}
To show \eqref{eq:jump}, we use the fact that
\begin{align}
 \re(L_{k,g}\varphi(\alpha,\theta))
& = \det(g)^{k+1}(a-c\alpha)^k\left[\re\varphi\left(\frac{d\alpha-b}{a-c\alpha}, \det(g)\theta\right)-\varphi\left(-\frac dc\right)\right]\notag\\
&\qquad-\sum_{n=1}^k\frac{\det(g)^{k-n+1} (a-c\alpha)^{k-n}}{c^nn!}\varphi^{(n)}\left(-\frac dc\right).\label{eq: Re Lkg}
\end{align}
To verify this formula we note that for any $m\in\N$ we have
\begin{align*}
 \MoveEqLeft\lim_{t\to 0+}\arg\left(\begin{pmatrix} 0&1\\1&m\end{pmatrix} .\left(\alpha+ t\cdot e^{i(\pi/2-\theta)}\right)
 -\begin{pmatrix} 0&1\\1&m\end{pmatrix} .\alpha\right)\\
 &=\lim_{t\to 0+}\arg\left(-\frac{t\cdot e^{i(\pi/2-\theta)}}{(\alpha+m)(\alpha+m+t\cdot e^{i(\pi/2-\theta))}}\right)
 =-\theta-\frac{\pi}{2}.
\end{align*}
(This is the ray we get after applying the M\"obius transform.)
Now, if $\varphi$ is analytic on the upper half plane and real on the real line, we can apply Schwarz' reflection principle and obtain 
\begin{align*}
 \re \varphi(\alpha, \theta)
 &= \lim_{t\to 0+} \re \varphi(\alpha+ t\cdot e^{i(\pi/2-\theta)})
 = \lim_{t\to 0+} \re \overline{\varphi}(\alpha+ t\cdot e^{-i(\pi/2-\theta)})\\
 &= \lim_{t\to 0+} \re \varphi(\alpha+ t\cdot e^{-i(\pi/2-\theta)}).
\end{align*}
Hence, if $\psi(z)=\varphi(g(m).z)$, then
$\re \psi(\alpha, \theta)= \re\varphi(g(m).\alpha,-\theta)$ and generally if 
$\psi(z)=\varphi(g.z)$ with $g\in\cM$, then
$\re \psi(\alpha, \theta)= \re\varphi(g.\alpha,\det(g)\theta)$.

Since every element of $\mathcal{M}$ can be written as a product of matrices $g(m)$, \eqref{eq: Re Lkg} follows.

By \eqref{eq: varphi 1k in terms of MZ}, we have
\begin{align*}
\MoveEqLeft \re\sum_{\mathcal{ZM}, k}\varphi_0(\alpha,\theta) - \re\sum_{\mathcal{ZM}, k}\varphi_0(\alpha,0)  \\
& = \sum_{n\in\mathbb{Z},g\in\mathcal{M}}\re \left[L_{k,g}\varphi_{1,k}(\alpha-n,\theta) - L_{k,g}\varphi_{1,k}(\alpha-n,0) \right]. 
\end{align*}
By Corollary~\ref{cor:odd}, the real part $\re\big[L_{k,g}\varphi_{1,k}(\alpha-n,\theta)-L_{k,g}\varphi_{1,k}(\alpha-n,0)\big]$ is $0$ if and only if $ \frac{d(\alpha-n)-b}{a-c(\alpha-n)} \not = 1$ and by \eqref{eq: Re Lkg}, 
$$\re\big[L_{k,g}\varphi_{1,k}(\alpha-n,\theta)-L_{k,g}\varphi_{1,k}(\alpha-n,0)\big]
= - \left( \det(g)(a-c(\alpha-n))\right)^k\theta,
$$
if $\frac{d(\alpha-n)-b}{a-c(\alpha-n)}= 1$.
There exist unique $n\in\mathbb{N}$ and $g =\big(\begin{smallmatrix}a&b\\c&d \end{smallmatrix}\big) \in\mathcal{M}$ such that
$\frac{d(\alpha-n)-b}{a-c(\alpha-n)}= 1$, i.e. $\alpha = n  + \frac{a+b}{c+d}.$
If $\alpha=p/q$, then $c+d=q$. Thus, 
$$\re\sum_{\mathcal{ZM}}\varphi_0(p/q,\theta) - \re\sum_{\mathcal{ZM}}\varphi_0(p/q,0) 
= - \frac{\text{det}(g)^{2k}}{(c+d)^k}\,\theta = -\frac\theta{q^k}. $$
Thus, $\re(\cB_k)$ has at each rational $p/q\in\mathbb{Q}/\mathbb{Z}$ a decreasing jump of $\pi/q^k$.
\end{proof}

\subsubsection{A comment on the Wilton function}

It would be interesting to get to know the behaviour of the Wilton function at rational points. 
However, the results can not easily be transferred to the Wilton function. 
By a similar argument as in the proof of Lemma~\ref{lem:odd}, we will get from the next lemma that 
$\wvarphi_1(z)+i\log(1-z)-\frac{1}{2\pi}(z+1)\log^2\left(1-\frac{1}{z}\right)$ is continuous on $\overline{\uhp}$
and in a neighbourhood of $1$ it behaves like 
$$
 \overline{\varphi}_{1}(z)=\frac{3\pi}{4}-\frac{\log{2}}{\pi}-i\log\left(1-z\right)+\frac{1}{2\pi}(z+1)\log^2\left(1-\frac{1}{z}\right)+O\left(|1-z|\log\left(|1-z|\right)\right).
$$
Thus, we do not have an analog of Corollary~\ref{cor:odd} since $\re\big(\log^2\big(1-\frac{1}{z}\big)\big)$ does not go to a finite value for $z\to 1$.

\subsection{Behaviour of the imaginary part of $\cW$}\label{sec:Wilton}
In this section, we will explain the behaviour of the imaginary part of the Wilton function near the real axis, distinguishing rational points, Wilton numbers and Diophantine numbers. Many arguments are slight adaptations of those given in \cite{MMY3}, thus they will only be sketched and many proofs will be given in the Appendix. 

The same behaviour which we observe for the imaginary part of Wilton function will also characterize the approach to the real axis of the imaginary part of the $k$-Brjuno function, but we will limit ourselves to prove the results for the former, being the arguments 
very similar.

The following lemma is an analog to \cite[Lem.~5.2]{MMY3} and Lemma \ref{lem:odd}. However, here we are considering the function $\wvarphi_1(z)+i\log(1-z)-\frac{1}{2\pi}(z+1)\log^2\left(1-\frac{1}{z}\right)$ instead of 
$\varphi_1(z)+i\log(1-z)$ in \cite{MMY3}. We will later remark how this further influences the proof of Theorem \ref{thm:ImWilton}.
\begin{lem}\label{lem:lem5.2}
The function $\wvarphi_1(z)+i\log(1-z)-\frac{1}{2\pi}(z+1)\log^2\left(1-\frac{1}{z}\right)$ is continuous on $\overline{\uhp}$ and its value at $1$ is $\frac{3\pi}{4}-\frac{\log{2}}{\pi}$.
\end{lem}

\begin{proof}
By combining \eqref{eq:wphi_1 Li} and Euler's functional equation \eqref{eq:Li2}
it follows that
\begin{equation}\label{eq:Lem6.4}
\begin{split}
 \wvarphi_1(z)
  = & \frac{1}{\pi}\left[z\Li{\frac{1-z}{z}}+\Li{z-1}\right] + \frac{1}{2\pi}\left[z\log^2\left(\frac{1-z}{-z}\right)+\log^2(-z+1)\right] \\
 & + \frac{\pi}{12}z+\frac{\pi}{6} - \frac{\log 2}{\pi}.
\end{split}
\end{equation}
The function $z\Li{\frac{1-z}{z}}+\Li{z-1}$ is regular and vanishing at $z=1$.
We have
$$z\log^2\left(\frac{1-z}{-z}\right)+\log^2(-z+1) = (z+1)\log^2\left(\frac{1-z}{-z}\right)+\log^2(-z+1)-\log^2\left(\frac{1-z}{-z}\right).$$
As in the proof of \cite[Lem.~5.2]{MMY3}, it follows that
$$\log^2(-z+1)-\log^2\left(\frac{1-z}{-z}\right) = \pi^2 - 2\pi i \log(1-z) + O\left(|z-1|\log \left(|z-1|\right)\right),$$
which completes the proof.
\end{proof}

From now on, we write $f(z) \lesssim g(z)$ if there exists a constant $C>0$ such that $f(z) \le C g(z)$ for all $z$.
The following is an analog to \cite[Thm.~5.10]{MMY3}.
\begin{thm}\label{thm:ImWilton}
For $n\ge 0$, $m_1,\cdots,m_n\ge 1$ and $z_0\in H(m_1,\cdots,m_n)$,
we have
\begin{align*}
\im\mathcal{W}(z_0)=& - W_{\textnormal{finite}}\left(\frac{p_n}{q_n}\right)+(-1)^{n}(p_{n-1}-q_{n-1}\re(z_0))\im\wvarphi_1(z_n+1)\\
&+r_n(z_0),
\end{align*}
with $|r_{n}(z_0)|\lesssim \frac{\log q_n}{q_n}|z_n|\log^{2}(1+|z_n|^{-1}).$
\end{thm}

As the proof is very similar to the one in \cite[Thm.~5.11]{MMY3}, it is given in the Appendix. However, the summand and the remainder term differ slightly from those in \cite{MMY3} about which we will comment in the following.

\begin{rmk}\label{rmk:rmk5.18}
Here we give a remark of the summand $(-1)^n(p_{n-1}-q_{n-1}\re(z_0))\im\wvarphi_1(z_n+1)$ of Theorem~\ref{thm:ImWilton}.
We have
$$(-1)^n(p_{n-1}-q_{n-1}\re(z_0))\im\wvarphi_1(z_n+1) =  \left(q_n^{-1}\log\frac{1}{|z_n|}\right)\left( 1+\frac{2}{\pi}\text{Arg}(1+1/z_n) + o(1) \right).$$
It can be derived as follows.

By \eqref{eq:dirichlet2}, we have
$$q_n(p_{n-1}-q_{n-1}\re(z_0)) = (-1)^n \re\left(\frac{q_n}{q_{n-1}z_n+q_n}\right) \to (-1)^n \quad \text{as }z_n\to 0.$$
By $\re(\log^2(w)) = \log^2|w|-(\text{Arg}(w))^2$ and $\im(\log^2(x)) = 2\log|w|\text{Arg}(w)$, we have
\begin{align*}
\im\left((z_n+2)\log^2\left(1+\frac{1}{z_n}\right)\right) = & 2(\re(z_n)+2)\text{Arg}(1+1/z_n)(\log|z_n+1|+\log|z_n|^{-1}) \\
&+ \im(z_n)(\log^2|1+1/z_n|-(\text{Arg}(1+1/z_n))^2).
\end{align*}
In the above equation, $2(\re(z_n)+2)\text{Arg}(1+1/z_n)\log|z_n+1|-\im(z_n)(\text{Arg}(1+1/z_n))^2$ is uniformly bounded.
Since $|\im(z_n)\log^2|1+1/z_n|\le |z_n|\log^2(1+1/|z_n|)\to 0$ as $z_n\to 0$, 
from Lemma~\ref{lem:lem5.2},
we have
\begin{align*}
 \frac{\im\wvarphi_1(z_n+1)}{\left(\log\frac{1}{|z_n|}\right)} 
& =\frac{-\log|z_n| + \frac{1}{2\pi}2(\re(z_n)+2)\text{Arg}(1+1/z_n)\log\left|z_n\right|^{-1}}{\log |z_n|^{-1}} + o(1) \\
& =1+\frac{2}{\pi}\text{Arg}(1+1/z_n) + o(1) = 1 + O(1),
\end{align*}
but not in $1+o(1)$.

Finally, we remark here on the estimate of the remainder term in Theorem~\ref{thm:ImWilton} which differs from the remainder term in \cite[Thm.~5.10]{MMY3} as we have an additional factor of $\log q_n\cdot \log (1+|z_n|^{-1})$. 
It is caused by the signs of the terms $z\varphi_0(1/z-1)$ and $\varphi_0(z-1)$ in the definition of $\wvarphi_1$ in \eqref{eq:wphi_1}.  The functions $\varphi_0(1/z-1)$ and $\varphi_0(z-1)$ are not continuous at $1$ because they can be written as a sum of a function continuous near $1$ and  a constant multiple of $\log^2(1-z)$ (see the proof of Lemma~\ref{lem:lem5.2}). Contrary to that in the case of the Brjuno function, the non-continuous parts cancel out (see \cite[Lem.~5.2]{MMY3}).
\end{rmk}

Finally, we will show an analog of \cite[Thm.~5.19]{MMY3}.
We will give a sufficient condition for $W(x)$ to be approximated by the imaginary part of $\cW(z)$ if $z$ approaches $x$. 
For $R>0$ and $0 < {r} < 1/2$, let 
$$U_{R}=\{u\in\uhp:~ \im(u) \ge |\re (u)|^{R}\}\quad\text{and}\quad
\widetilde{U}_{r} = \left\{u\in\uhp: ~\im(u) \ge \exp{\left( -\frac{1}{|\re(u)|^{r}} \right)} \right\}.$$
\begin{thm}\label{thm: ImWilton1} 
We have the following statements:
\begin{enumerate}[(i)]
\item For any Wilton number $x$ and any $R>0$, we have
$$\lim_{u\to 0,~ u\in U_{R}} \im\cW(u+x) = - W(x).$$

\item Let $x$ be an irrational Diophantine number and $0 < r < 1/2$ such that 
$$\liminf_{q\to\infty} \|q x \|_{\Z}~q^{1/r-1} = \infty,$$
where $\|\cdot\|_\Z$ denotes the distance from the nearest integer. Then,
$$\lim_{u\to 0,~u\in \widetilde{U}_{r}} \im\cW(u+x) = - W(x).$$
\end{enumerate}
\end{thm}
Since the singular behaviour of the k-Brjuno function as the real line is approached is tamer than the singular behaviour of the Brjuno function, we expect that the same properties proven in \cite{MMY3} and stated in the theorem above for the Wilton function hold by a modest adaptation of the proofs given there. 

\begin{proof}[Proof of Theorem \ref{thm: ImWilton1}]
Let $x\in(0,1)$ be an irrational number whose continued fraction expansion is
$$x = \cfrac{1}{m_1+\cfrac{1}{m_2+\cfrac{1}{\ddots +\cfrac{1}{m_n +\ddots}}}}.$$
Let $x_\ell = \Gauss^\ell(x)$.
Let $\{p_n/q_n\}_{n\ge 0}$ be the sequence of the partial quotients of $x$.
Let $u$ be a point near $0$ such that $|\im(u)|\le 1/2$ and $z = u + x \in \Delta$.
Then there are $N_1 \cdots, N_L$ such that  $z\in H(N_1,\cdots,N_L)$.
Let $p/q$ be a rational number such that 
$$p/q=\frac{1}{N_1+\cfrac{1}{N_2+ \cfrac{1}{\ddots +\cfrac{1}{N_L}}}}.$$
We distinguish two cases.

\begin{enumerate}
\item[(I)] 
Assume that $p/q = p_L/q_L$, i.e., $m_i = N_i$ for $i=1,\cdots, L$.
From Theorem~\ref{thm:ImWilton}, Remark~\ref{rmk:rmk5.18} and Lemma~\ref{lem:lem5.20}, 
there exist uniform constants $c_1,c_2,c_3>0$ such that
\begin{equation}\label{eq:imW on w nb}\begin{split}
& |\im\cW(u+x) + W(x)| \\
& =\Big|-W_{\mathrm{finite}}\left(\frac{p_L}{q_L}\right) + \sum_{\ell=0}^{L-1}(-1)^\ell \beta_{\ell-1}(x)\log\frac{1}{x_\ell} + \sum_{\ell=L }^\infty (-1)^\ell \beta_{\ell-1}(x)\log\frac{1}{x_\ell} \\
& \qquad +(-1)^n(p_{L-1}-q_{L-1}\re(z_0))\im\wvarphi_1(z_L+1) +r_n(z_0)  \Big| \\
& \le c_1x_L q_n^{-1} + c_2q_{L}^{-1}\log|z_L|^{-1} + c_3 q_L^{-1} \log{(q_L)} |z_L|\log^2(1+|z_L|^{-1}) \\
& \qquad  + \Big|\sum_{\ell=L}^\infty (-1)^\ell \beta_{\ell-1}(x)\log\frac{1}{x_\ell}\Big|.
\end{split}\end{equation}
This case follows in an analogous way to the proof of \cite[Thm.~5.19(I)]{MMY3} where we note that the considered terms also stay small if we consider a Wilton instead of a Brjuno number.

\item[(II)]
Assume that $p/q$ is not one of the partial quotients of $x$.
We denote by $\{p_\ell'/q_\ell'\}_{0\le \ell \le L}$ the partial quotient of $p/q$ and by $n$ the largest integer such that $p_n'/q_n'= p_n/q_n$.
Clearly, one has $n<L$ and $p_L'/q_L' = p/q$.

Note that $q_n^{-1}\log(q_n)\ge (q_L')^{-1}\log(q_L')$ if $q_n\ge 3$.
As in the proof of \cite[Thm.~5.19]{MMY3}, from the fact $|x-p/q|\ge (2q^2)^{-1}$,
if $u\in U_R$, then there exists a uniform constant $c>0$ such that
$$q^{-1}\log|z_L|^{-1}\le q^{-1}(c+(2R-2)\log q) \le q_n^{-1}(c+(2R-2)\log q_n)$$ and if $u \in \widetilde{U_r}$, then we have
$$q^{-1}\log|z_L|^{-1}\le q^{-1}cq^{2r}\le q_n^{-1}cq_n^{2r}.$$
From Theorem~\ref{thm:ImWilton}, Remark~\ref{rmk:rmk5.18} and Lemma~\ref{lem:lem5.20}, there exist uniform constants $c_1,c_2,C_1,C_2,C_3>0$ such that
\begin{equation*}\begin{split}
& |\im\cW(u+x) + W(x)| \\
& =\Big|-W_{\mathrm{finite}}\left(\frac{p_L'}{q_L'}\right) + \sum_{\ell=0}^{n-1}(-1)^\ell \beta_{\ell-1}(x)\log\frac{1}{x_\ell} + \sum_{\ell=n}^\infty (-1)^\ell \beta_{\ell-1}(x)\log\frac{1}{x_\ell} \\
& \qquad +(-1)^L(p_{L-1}'-q_{L-1}'\re(z_0))\im\wvarphi_1(z_L+1) +r_L(z_0)  \Big| \\
& \le \Big|-W_{\mathrm{finite}}\left(\frac{p_L'}{q_L'}\right) + W_{\mathrm{finite}}\left(\frac{p_n}{q_n}\right)\Big| + \Big| -W_{\mathrm{finite}}\left(\frac{p_n}{q_n}\right) + \sum_{\ell=0}^{n-1}(-1)^\ell \beta_{\ell-1}(x)\log\frac{1}{x_\ell} \Big| \\
& \qquad + \Big|\sum_{\ell=n}^\infty (-1)^\ell \beta_{\ell-1}(x)\log\frac{1}{x_\ell}\Big| + c_1\frac{\log|z_L|^{-1}}{q_{L}'} + c_2 \frac{\log q_L'}{q_L'} |z_L|\log^2(1+|z_L|^{-1}) \\
& \le \Big|-W_{\mathrm{finite}}\left(\frac{p_L'}{q_L'}\right) + W_{\mathrm{finite}}\left(\frac{p_n}{q_n}\right)\Big| + \Big|\sum_{\ell=n}^\infty (-1)^\ell \beta_{\ell-1}(x)\log\frac{1}{x_\ell}\Big| \\
& \qquad + C_3 \frac{x_n}{q_n} + C_1 \frac{1}{q_n^{\gamma}} + C_2 \frac{\log q_n}{q_n}
\end{split}\end{equation*}
since $|z_L|\log^2 (1+|z_L|^{-1})$ are bounded, where $\gamma = 1$ for the case of $u\in U_R$ and $\gamma = 1-2r$ for the case of $u\in \widetilde{U}_r$.
Thus, it is enough to show that $W_\mathrm{finite}(p/q)$ is close to $W_\mathrm{finite}(p_n/q_n)$.

Let 
$$\rho = \max_{n\le \ell < L} (q_\ell')^{-1}\log\frac{q_{\ell+1}'}{q_\ell'}(\ell-n+1)^2$$
By Lemma~\ref{lem:lem5.20}, we have
\begin{align*}
& |W_{\mathrm{finite}}(p_n/q_n) - W_{\mathrm{finite}}(p/q)| \\
& \le \left|W_{\mathrm{finite}}(p_n/q_n) - \sum_{\ell=0}^{n-1}(-1)^\ell \beta_{\ell-1}(x)\log\frac{1}{x_\ell}  + \sum_{\ell=0}^{L-1}(-1)^\ell \beta_{\ell-1}(x)\log\frac{1}{x_\ell}  - W_{\mathrm{finite}}(p/q)|\right| \\
& \le C x_nq_n^{-1} + C x_L (q_L')^{-1} + \left|\sum_{\ell=n}^{L-1}(-1)^\ell \beta_{\ell-1}(x)\log\frac{1}{x_\ell}\right|
\le C' q_n^{-1} + \left|\sum_{\ell=n}^{L-1}(-1)^\ell \beta_{\ell-1}(x)\log\frac{1}{x_\ell}\right|.
\end{align*}
We have$$
\sum_{\ell=n}^{L-1} \beta_{\ell-1}(x)\log\frac{1}{x_\ell}\lesssim \sum_{\ell=n}^{L-1}\frac{1}{q_\ell'}\log\frac{q_{\ell+1}'}{q_\ell'} \le \rho \sum_{\ell=r}^{L-1}\frac{1}{(\ell-n+1)^2}\lesssim \rho.
$$
With the same argument as in the proof of \cite[Thm.~5.19]{MMY3} we have the conclusion.
\end{enumerate}
\end{proof}

\appendix
\section{Proof of Proposition \ref{pr:Tkbound}} 

\begin{proof}
From the fact that $\mathcal{M}$ is the free monoid generated by the matrices $g(m)$, we have that
$z\in D_\infty$ and $g\in \mathcal{M}$ imply $g^{-1}. z\in D_\infty$. 
Also, by this fact combined with \cite[Lem.~3.2]{MMY3}, we then have that 
\begin{equation}\label{eq:g preserves V_rho}
z\in V_\rho(D_\infty) \text{ implies  } g^{-1}. z\in V_\rho(D_\infty).
\end{equation}

Let $\mathcal{M}^{(r)}=\{g\in\mathcal{M}\;|\;g=g(m_1)\cdots g(m_r), ~ m_i\in \N\}$. Then we have for all $\psi\in\Occ{k+2}{I}$ that
\begin{equation}
(T_k^{(k+2)})^r\psi=\sum_{g\in\mathcal{M}^{(r)}}L^{(k+2)}_{g}\psi=\sum_{g'\in\mathcal{M}^{(r-1)}}L^{(k+2)}_{g'}(T_k^{(k+2)}\psi).\label{eq: proof Tkbound}
\end{equation}
Let $g'=\left(\begin{smallmatrix} a'& b'\\ c'&d'\end{smallmatrix}\right)$, $z\in V_\rho(D_\infty)$ and $z'=(g')^{-1}. z$. 
Then for all $m\in\N$, we have
\begin{equation*}
g'g(m)=\begin{pmatrix} b'& a'+mb'\\ d'&c'+md'\end{pmatrix}.
\end{equation*}
Thus, by using the definition of $T_k^{(k+2)}$ given in Proposition \ref{pr:Tkest} together with 
\eqref{eq: Lkg deriv} and \eqref{eq: Lkg(m) deriv} we obtain
\begin{align*}
L^{(k+2)}_{g'}\left(T_k^{(k+2)}\psi\right)
&=L^{(k+2)}_{g'}\left(\sum_{m=1}^\infty L_{g(m)}^{(k+2)}\psi\right)
=L^{(k+2)}_{g'}\left(\sum_{m=1}^\infty -z^{-(k+2)}\psi\left(\frac{1}{z}-m\right)\right)\\
&=\det(g)^{k+1}\left(a'-c'z\right)^{-(k+2)} \left(-\sum_{m=1}^\infty \left(\frac{a'-c'z}{d'z-b'}\right)^{k+2}\psi\left(\frac{1}{z'}-m\right)\right)\\
&=-\det(g)^{k+1}(d'z-b')^{-(k+2)}\sum_{m=1}^\infty \psi\left(\frac{1}{z'}-m\right). 
\end{align*}
We remark that if $r=1$, then we consider $g' \in \cM^{(0)} = \{\text{id}\}$ and the following arguments work well.

It follows from \eqref{eq:O^k bound} and \eqref{eq:g preserves V_rho} that
$$\sum_{m=1}^\infty\left|\psi\left(\frac{1}{z'}-m\right)\right|\le \sum_{m=1}^\infty c_{1,\rho, k}\left|\frac{1}{z'}-m\right|^{-(k+2)}\sup_{V_\rho(D_\infty)}|\psi|$$
for some constant $c_{1,\rho, k}$.
Since $\sup_{z\in D_\infty}\re(1/z) =  1$ (see Figure~\subref{fig:iota Dinfty}), there exists $c_2>0$ such that $|1/z-m|^{-1}\le c_2m^{-1}$ for all $z\in D_\infty$ and $m\ge 2$.
Thus, a similar argument as the one showing \eqref{eq: sum m infty} in the proof of Proposition~\ref{pr:Tkest} yields that for all $z'\in V_\rho(D_\infty)$
$$\sum_{m=1}^{\infty}\left|\psi\left(\frac{1}{z'}-m\right)\right|\le {c_{3,\rho,k}}\sup_{ V_\rho(D_\infty)}|\psi|$$ 
for some constant ${c_{3,\rho, k}}$.

Also, there exists $c_{4,\rho}>0$ such that $|z-\frac{b'}{d'}|^{-1}\le c_{4,\rho}$ for all $z\in V_\rho(D_\infty)$ and $\frac{b'}{d'}\in[0,1]$. 
Here, $d'$ can be considered as the denominator of the $r$th principal convergent of some number.
Finally, 
\begin{equation}\label{eq: estim d'}
 \min_{g'\in \mathcal{M}^{(r-1)}}d'\geq \frac12\left(\frac{\sqrt{5}+1}{2}\right)^{r-1}\text{ and }
\sum_{g'\in\mathcal{M}^{(r-1)}}d'^{-2}\le c_5,
\end{equation}
for some $c_5>0$, 
see for example \cite[Prop. A1.1 in Appendix 1]{MMY3}.
Furthermore, since $g'\in \mathcal{M}^{(r-1)}$, we have that $\left|\det(g')^{k+1}\right|=1$.
Combining these results yields the statement of \eqref{en: 1Tkbound}.

In the following we will show \eqref{pr:Tkbound-item2} for $r\ge 2$, and we will discuss the case  $r=1$ separately.
Starting in the analogous way as in the proof of \eqref{en: 1Tkbound}, using
\eqref{eq: Lkgphi1}, we obtain for $\varphi \in \Occ{1}{[0,1]}$ that 
\begin{align*}
L_{k,g'}\left(T_k\varphi\right)
&=L_{k,g'}\left(\sum_{m=1}^\infty L_{k,g(m)}\varphi\right)\\
&=\frac{(c')^{-(k+1)}\det(g')^{k+1}}{k!}\,(a'-c'z)^{-1}\\
&\qquad\cdot\int_0^1(1-t)^k  \sum_{m=1}^\infty L_{k,g(m)}^{(k+2)}\varphi^{(k+1)}\left(-\frac{d'}{c'}+\frac{\det(g')t}{c'(a'-c'z)}\right)\mathrm{d}t.
\end{align*}

We have
$$\xi_{z,t}:=-\frac{d'}{c'}+\frac{\det(g')t}{c'(a'-c'z)} = -\frac{d'}{c'}\, t+\frac{\det(g')t}{c'(a'-c'z)}+\frac{d'}{c'}\, t-\frac{d'}{c'}
= tg'^{-1}.z - (1-t)\frac{d'}{c'}.$$
If we denote
$\iota(z):=1/z$, see Figure~\subref{fig:iota Dinfty}, then we
can inductively show that $g'^{-1}.z$ is contained in a $\rho$-neighbourhood of $\bigcup_{m\ge 1}(\iota(D_\infty)-m)\subset D_\infty$ with respect to the Poincar\'{e} metric.
Since $d'/c' \in (-\infty,-1]$, the line segment between $-d'/c'$ and $g^{-1}.z$ is contained in $V_\rho(D_\infty \cap  \bigcup_{m\ge 1}(\iota(D_\infty)-m))$.

Thus, there is a constant $c_{6,\rho}>0$ such that $|\xi_{z,t}|\ge c_{6,\rho}$. The constant $c_{6,\rho}$ increases when $\rho\to 0$. Since $\xi_{z,t}$ is in a $\rho$-neighbourhood of $i$, $-i$, or $\re(\xi_{z,t})\le 0 $, we have $|1/\xi_{z,t}-m|\ge m$.

Since $\varphi^{(k+1)}\in\Oc{k+2}$, by using \eqref{eq: Lkg deriv} and \eqref{eq:jth derivative}, we have
\begin{align*}
\sup_{z\in V_\rho(D_\infty)}|L_{k,g'}(T_k\varphi)| 
& \le \frac{(c')^{-(k+2)}}{k! \inf\limits_{z\in V_\rho(D_\infty)}|z-\frac{a'}{c'}|}\sum_{m\ge 1}\int_0^1\frac{|\varphi^{(k+1)}\left(1/\xi_{z,t}-m\right)|}{\left| \xi_{z,t} \right|^{k+2}}dt\\
& \le \frac{(c')^{-(k+2)}}{k! \inf\limits_{z\in V_\rho(D_\infty)}|z-\frac{a'}{c'}|}\sum_{m\ge 1}\frac{c_{1,k+1}'(k+1)! }{c_{6,\rho}^{k+2}(m-0.2)^{k+2}}\sup_{z\in V_\rho(D_\infty)}|\varphi(z)|\\
& \le \frac{c_{7,k}c_{1,k+1}'(k+1)}{c_{6,\rho}^{k+2}} \cdot \frac{(c')^{-(k+2)}}{\inf\limits_{z\in V_\rho(D_\infty)}|z-\frac{a'}{c'}|} \cdot \sup_{z\in V_\rho(D_\infty)}|\varphi(z)|,
\end{align*}
where $c_{7,k} = \sum_{m\ge 1} (m-0.2)^{-(k+2)}$.

Since $a'/c'\in[0,1]$, we have $|z-\frac{a'}{c'}|^{-1}\le c_{4,\rho}$ for $z\in V_\rho(D_\infty)$.
Using \eqref{eq: estim d'} and \eqref{eq:jth derivative}, we obtain
\begin{align}
\sup_{z\in V_\rho(D_\infty)}|(T_k^r \varphi)(z)| 
& \le \frac{2^k c_{7,k} c_{1,k+1}'(k+1)}{c_{6,\rho}^{k+2}}\left(\frac{\sqrt{5}-1}{2}\right)^{k(r-2)}\sum_{g'\in\cM^{(r-1)}}(c')^{-2} \cdot \sup_{z\in V_\rho(D_\infty)}|\varphi(z)|\notag\\
& \le \overline{C}_{\rho,k, \geq 2}\left(\frac{\sqrt{5}-1}{2}\right)^{rk} \sup_{z\in V_\rho(D_\infty)}|\varphi(z)|,\label{eq: estim r geq 2}
\end{align}
where
$\overline{C}_{\rho,k, \geq 2} = 2^k (k+1) (\frac{\sqrt{5}+1}{2})^{2k} c_{1,k+1}'c_5c_{4,\rho}c_{6,\rho}^{-(k+2)}c_{7,k}$, implying \eqref{pr:Tkbound-item2} for $r\geq 2$.

For the case $r=1$, we will estimate 
$$|T_k\varphi (z)| \le \sum_{m\ge 1} |L_{k,g(m)}\varphi(z)|$$
by using \eqref{eq: Lkg(m)phi} and \eqref{eq: Lkg(m)phi1}.
We consider $\ell$ denoting the the right arc of $\partial D_\infty$ (which is the arc of the circle $\left|z-\frac{\sqrt{3}}{3}\right| = \frac{\sqrt{3}}{3}$ between $\frac{\sqrt{3}}{2}+\frac{1}{2}i$ and $\frac{\sqrt{3}}{2}-\frac{1}{2}i$ containing $\frac{2\sqrt{3}}{3}$), see Figure~\subref{fig:Dinf}.
Since $1/z\in V_\rho(\iota(D_\infty))$, from the shape of $\iota(D_\infty)$, we can conclude that $t/z\in V_\rho(\iota(D_\infty))$ for $0\le t\le 1$.
Thus, $t/z-m\in V_\rho(\iota(D_\infty)-m)$.
If $m\ge 2$ or $z$ is not in the neighbourhood of $\ell$, then there is $c>0$ such that $|t/z-m|\ge cm$, thus by Cauchy's estimate as in \eqref{eq:jth derivative}, we have
$$\left|\varphi^{(k+1)}\left(\frac tz-m\right)\right| \le c_{1,k+1}'\frac{(k+1)!\sup_{V_\rho(D_\infty)}|\varphi|}{cm^{k+2}}.$$
Thus, if $z$ is not in the neighbourhood of $\ell$, then by using \eqref{eq: Lkg(m)phi1}, we have
\begin{align}
 \label{eq:Tk estimate} |T_k\varphi (z)| 
& \le \frac{1}{c_{6,\rho}k!} \sum_{m\ge 1} c_{1,k+1}'\frac{(k+1)!\sup_{V_\rho(D_\infty)}|\varphi|}{cm^{k+2}} \\
& \le \overline{C}_{\rho,k}' \left(\frac{\sqrt{5}-1}{2}\right)^k \sup_{V_\rho(D_\infty)}|\varphi|,\notag
\end{align}
where $\overline{C}_{\rho,k}' = \left(\frac{\sqrt{5}-1}{2}\right)^{-k}(k+1)c_{1,k+1}' c_{6,\rho}^{-1} c^{-1}\sum_{m\ge 1}\frac{1}{m^{k+2}}$.
We are left to show the case that $z$ is in a neighbourhood of $\ell$.
For $z\in V_\rho(D_\infty)$ in a neighbourhood of $\ell$, we can assume that $|z|\le 2$.
Thus, by \eqref{eq: Lkg(m)phi} and \eqref{eq:jth derivative}, we have
\begin{align}
|L_{k,g(1)}\varphi(z)|
& \le |z|^k\left|\varphi\left(\frac{1}{z}-1\right)\right|+ \sum_{n=0}^k\frac{|z|^{k-n}}{n!}|\varphi^{(n)}(-1)| \notag \\
& \le |z|^k \sup_{V_\rho(D_\infty)}|\varphi|  + \sum_{n=0}^k \frac{|z|^{k-n}}{n!} c_{1,n}' \frac{n!\sup_{V_\rho(D_\infty)}|\varphi|}{0.8^{n+1}} \notag \\
& \label{eq:Lkg(1)estimate near arc} \le \left(|z|^k + \sum_{n=1}^k \frac{c_{1,n}' |z|^{k-n}}{0.8^{n+1}} \right)\sup_{V_\rho(D_\infty)}|\varphi|\le C_k\left(\frac{\sqrt{5}-1}{2}\right)^{k}\sup_{V_\rho(D_\infty)}|\varphi|,
\end{align}
where $C_k = \left(\frac{\sqrt{5}-1}{2}\right)^{-k}2^k \left(1+ \sum_{n=1}^k \frac{c_{1,n}'}{2^n 0.8^{n+1}}\right)$.
By replacing the summand of $m=1$ in \eqref{eq:Tk estimate} with \eqref{eq:Lkg(1)estimate near arc} and combining this result with the estimate for $r\geq 2$ in \eqref{eq: estim r geq 2}, we have the conclusion when we take $\overline{C}_{\rho,k} = \max\{C_k + \overline{C}_{\rho,k}', \overline{C}_{\rho,k,\geq 2}\}$.
\end{proof}

\section{Proof of Theorem \ref{thm:ImWilton}.}

Analogously as in \cite{MMY3} the strategy of the proof of the theorem is to start from the formula
\begin{align}\label{eq: Im Wilton start}
 \im \mathcal{W}(z)=\sum_{n\in {\Z}} \sum_{\ell\ge 0}
\im S^\ell \wvarphi_1(z+n).
\end{align}
We will only give the parts of the proof which are different.

We have the following lemmas and propositions which are analogs to \cite[Lem.~5.11, 5.14 and Prop.~5.15, 5.16]{MMY3} in the same way by using $S=-T$ and substituting $\varphi_1$ to $\wvarphi_1$.
For the following we denote for a bounded closed interval $I \subset \R$, the space of
functions $\varphi\in\Occ{1}{I}$ whose real part is bounded we denote by $E(I)$, endowed with the norm $\|\varphi\|_{E(I)} =\sup_{\C\backslash I} |\re\varphi|$. If it is clear which interval $I$ is meant, we also write $\|\varphi\|$ instead of $\|\varphi\|_{E(I)}$.
Note that $\wvarphi_1\in E([1/2,2])$. If $\psi\in E([1/2,2])$ is real on the real axis outside $[1/2,2]$, then $\psi$ holds the arguments about $\varphi_1$ in the proof of \cite[Lem.~5.11, 5.14 and Prop.~5.15, 5.16]{MMY3}.

\begin{lem}\label{lem:5.11}
For $z\in\Delta\setminus[0,1]$, if we write
\begin{align*}
\im\cW(z)=& \im\wvarphi_1(z)+\im\wvarphi_1(z+1)+\im\wvarphi_1(z+2) + \im S\wvarphi_1(z-1) + \im S\wvarphi_1(z) \\
& +\sum_{\ell>1}\big[\im S^\ell\wvarphi_1(z-1) + \im S^\ell\wvarphi_1(z)+\im S^\ell\wvarphi_1(z+1) \big]+r^{0}(z),
\end{align*}
then we have
$$|r^{0}(z)|\le C|\im(z)|\| \wvarphi_1\|_{E([1/2,2])}.$$
\end{lem}

\begin{lem}\label{lem:5.14}
For $z\in\Delta\setminus[0,1]$, we have
\begin{equation}\begin{split}\label{eq:5.14}
\im\cW(z) =& (1-\veps_1)\im \wvarphi_1(z) +\im \wvarphi_1(z+1) +[1-\veps_2(1-\veps_1)]\im S \wvarphi_1(z) \\
& +\sum_{\ell>1}\im S^\ell\wvarphi_1(z) + r^{0}(z) + r^{1}(z),
\end{split}\end{equation}
and
$$|r^{1}(z)|\le C|\im (z)|\log(1+|\im (z)|^{-1})\|\wvarphi_1\|_{E([1/2,2])}.$$
\end{lem}

\begin{prop}\label{prop:5.15}
For $n\ge 1$, $m_{1},\ldots ,m_{n}\ge 1$, $z_{0}\in D (m_{1},\ldots ,m_{n})$, $\ell>0$ we have
\begin{align*}
 \MoveEqLeft\left|(-1)^{n}\im \left(S^{n+\ell}\wvarphi_{1}(z_{0})\right)-\left[\im\left( S^{\ell}\wvarphi_{1}(z_{n})\right)\right]
(p_{n-1}-q_{n-1}\re( z_{0}))\right|\\
&\le Cq_{n}^{-1}|\im (z_{n})|\log (1+|\im (z_{n})|^{-1})\times
\begin{cases}
\Vert S^{\ell-2}\wvarphi_{1}\Vert &\text{if }\ell>1,\\
\Vert \wvarphi_{1}\Vert &\text{if }\ell=1. \end{cases}
\end{align*}
\end{prop}
\begin{prop}\label{prop:5.16}
For $n\ge 1$, $m_{1},\ldots ,m_{n}\ge 1$, $z_{0}\in D(m_{1},\ldots ,m_{n})$ we have 
\begin{align*}
 \MoveEqLeft|(-1)^{n}\im (S^n\wvarphi_{1}(z_{0}))-(p_{n-1}-q_{n-1}\re( z_{0}))[
(1-\varepsilon_{n+1})\im (\wvarphi_{1}(z_{n}))+\varepsilon_{n}\im ( \wvarphi_{1}(z_{n}+1))]|\\
&\le Cq_{n}^{-1}|\im ( z_{n})|\log (1+|\im( z_{n}) |^{-1})\Vert\wvarphi_{1}\Vert \; . \cr
\end{align*}
\end{prop}

The following is an analog to \cite[Prop.~5.17]{MMY3}. We remember the definition of $H(m_1,\cdots,m_n)$ given in \eqref{eq: def H}.
\begin{prop}\label{prop:5.17}
For $n \ge 0$, $m_{1},\ldots ,m_{n}\ge 1$,
$z_0\in H(m_{1},\ldots ,m_{n})$ we have
\begin{align*}
\im\mathcal{W}(z_0)=&\sum_{\ell=0}^{n} {\varepsilon_{\ell}} {(-1)^\ell}(p_{\ell-1}-q_{\ell-1}\re(z_0)) \im\wvarphi_1(z_\ell+1)\\
&+\sum_{\ell=0}^{n-1}(1-\varepsilon_{\ell+1}) {(-1)^\ell}(p_{\ell-1}-q_{\ell-1}\re(z_0)) \im\wvarphi_1(z_\ell) + r_{n}(z_0),
\end{align*}
with $\veps_0=1$, and $|r_{n}(z_0)|\le Cq_n^{-1}|z_n|\log(1+|z_n|^{-1})$.
\end{prop}

\begin{proof}
Since $\wvarphi_1$ is real on the real axis outside $[0,1]$, by Proposition~\ref{prop:4.1}, when $z\in H$, 
\begin{equation}\label{eq:5.28}
|\im S^{n}\wvarphi_1(z)|\lesssim |z|\log(1+|z|^{-1})\sup_{D_\infty}|S^{n-1}\wvarphi_1|.
\end{equation}
We obtain $\sum_{n=1}^{\infty} \sup_{D_\infty}|S^{n} \wvarphi_1|<\infty$ by using Proposition~\ref{pr:Tkbound}-\eqref{pr:Tkbound-item2}.
Thus, from \eqref{eq:5.28} and Lemma~\ref{lem:5.14}, we get
\begin{equation}\label{eq:prop6.10-en1}
\im \cW(z) = (1-\veps_1)\im\wvarphi_1(z) + \im \wvarphi_1(z+1)+r(z),
\end{equation}
with $|r(z)|\lesssim |z|\log(1+|z|^{-1}).$

Let us consider the case of $z_0\in H$ which is corresponding to the case of $n=0$. Since we set $\veps_1=1$ for $z_0\in \Delta\setminus D$, from \eqref{eq:prop6.10-en1}, we obtain
$$|\im\cW(z_0)-\im\wvarphi_1(z_0+1)|\lesssim |z_0|\log(1+|z_0|^{-1}).$$

Next, we consider $z_0\in H(m_1,\cdots,m_{n})$ with $n\ge 1$ and $m_1,\cdots,m_{n}\ge 1$.
By Lemma~\ref{lem:5.14}, we need to show that the modulus of 
\begin{equation}\label{eq:prop5.17}\begin{split}
& \Big(\sum_{\ell=n+1}^\infty \im S^\ell \wvarphi_1(z_0) \Big) + (-1)^{n}(p_{n-1}-q_{n-1}\re(z_0))(1-\veps_{n+1}) \im\wvarphi_1(z_{n})+r^{0}(z_{0})+r^{1}(z_{0})\\
& + \Big[(1-\veps_2(1-\veps_1))\im S \wvarphi_1(z_0) + (p_0-q_0\re(z_0))(\veps_1\im\wvarphi_1(z_1+1)+(1-\veps_2)\im\wvarphi_1(z_1))\Big]\\
& + \sum_{\ell=2}^{n} \Big[\im S^\ell\wvarphi_1(z_0) - (-1)^\ell (p_{\ell-1}-q_{\ell-1}\re(z_0))\big( \varepsilon_{\ell}\im\wvarphi_1(z_\ell+1) + (1-\varepsilon_{\ell+1})\im\wvarphi_1(z_\ell)\big)\Big] \\
\end{split}\end{equation}
is in $O(q_{n}^{-1}|z_{n}|\log(1+|z_{n}|^{-1}))$.
From \cite[Prop.~4.11]{MMY3} using $S=-T$, we have $\sum_{j=0}^\infty\|S^j\wvarphi_1\|<\infty$.
Then, from Proposition~\ref{prop:5.15}, we have
\begin{align*}
& \Big|\Big(\sum_{\ell=n+1}^\infty\im S^\ell\wvarphi_1(z_0) \Big)-(p_{n-1}-q_{n-1}\re(z_0))\sum_{\ell=1}^\infty\im S^\ell\wvarphi_1(z_{n})\Big| \\
& \lesssim q_{n}^{-1}|\im (z_{n})|\log(1+|\im (z_{n})|^{-1}).
\end{align*}
Since $z_{n}\in H$, by using \eqref{eq:5.28}, the first term of \eqref{eq:prop5.17} is bounded by 
$$\sum_{\ell=n+1}^\infty\big|\im S^\ell\wvarphi_1(z_0) \big|\lesssim q_{n}^{-1}|z_{n}|\log(1+|z_{n}|^{-1}).$$

For $z\in H$, 
we have 
\begin{equation}\label{eq:im phi_1} |\im\wvarphi_1(z)|\lesssim |\im(z)|. \end{equation}

For the second term of \eqref{eq:prop5.17}, by using \eqref{eq:im phi_1} and \eqref{eq:dirichlet}, we have
$$|(p_{ n-1}-q_{n}-1\re(z_0))\im\wvarphi_1(z_{n})(1-\veps_{n+1})|\lesssim q_{n}^{-1}|\im(z_{n})|.$$

We have $|r^{0}(z_0)+r^{1}(z_0)|\lesssim |\im(z_0)|\log(1+|\im(z_0)|^{-1})\|\wvarphi_1\|_{E([1/2,2])}$ in Lemma~\ref{lem:5.11} and \ref{lem:5.14}.
Since 
$$[1-\veps_2(1-\veps_1)](1-\veps_2)=1-\veps_2\quad\text{ and }\quad [1-\veps_2(1-\veps_1)]\veps_1=\veps_1,$$
the third term of \eqref{eq:prop5.17} is $0$ or 
$$\left[\im S\wvarphi_1(z_0)+(p_0-q_0\re(z_0))(\veps_1\im\wvarphi_1(z_1+1)+(1-\veps_2)\im\wvarphi_1(z_1))\right].$$
By using Proposition~\ref{prop:5.16}, the last terms of \eqref{eq:prop5.17} are in
$$O\left(\sum_{\ell=0}^{n} q_\ell^{-1}|\im(z_\ell)|\log(1+|\im(z_\ell)|^{-1})\|\wvarphi_1\|_{E([1/2,2])}\right)$$
and we have 
$$\sum_{\ell=0}^{n} q_\ell^{-1}|\im(z_\ell)|\log(1+|\im(z_\ell)|^{-1})\le q_{n}^{-1}|\im(z_{n})|\log(1+|\im(z_{n})|^{-1})$$
as in the last part of the proof of \cite[Prop.~5.17]{MMY3}.
Since $|\re(z_0)|\le 1$, we have the conclusion.
\end{proof}

\begin{proof}[Proof of Theorem~\ref{thm:ImWilton}]
Let $n\ge 0$ and $m_1,\cdots, m_n\ge1$.
The domain $H(m_1,\cdots,m_n)$ meets $\R$ in a unique point $p_n/q_n$.
Let us denote $x_0=p_n/q_n$ and consider its continued fraction
$$x_i^{-1}=m_{i+1}+x_{i+1},\quad 0\le i<n \quad \text{and} \quad x_n=0,$$
here, $x_i$ is the point of intersection of $H(m_{i+1},\cdots,m_n)$ with $\R$.

Let $z_0\in H(m_1,\cdots,m_n)$ and $\im(z_0)>0$.
We will then have $(-1)^\ell \im(z_\ell)>0$ for $0\le\ell\le n$.
From \eqref{eq:Imphi1}, we deduce that 
\begin{equation}\label{eq:Im wphi x_l+1}
 \im\wvarphi_1 (x_\ell+1+(-1)^\ell i0) = (-1)^{\ell}\log\dfrac{1}{x_\ell}, \quad \text{ for } ~ 0\le \ell <n.
\end{equation}
Since $\frac{x_\ell}{1-x_\ell}=\frac1{x_{\ell+1}}$ if $\veps_{\ell+1}=0$, we obtain from \eqref{eq:Imphi1}
\begin{equation}\label{eq:Im wphi x_l}
(1-\veps_{\ell+1})\im\wvarphi_1(x_\ell+(-1)^\ell i0) = (1-\veps_{\ell+1})(-1)^{\ell+1} x_\ell \log\dfrac{1}{x_{\ell+1}}, \quad \text{ for }~0\le \ell< n-1. 
\end{equation}

The above equations imply for $W_{\text{finite}}$ extended to the complex plane that
\begin{equation}\label{eq:W finite}
\begin{split}
W_{\text{finite}}(p_n/q_n)  
= & \sum_{\ell=0}^{n-1}\veps_\ell (q_{\ell-1}x_0-p_{\ell-1})(-1)^{\ell} \im\wvarphi_1(x_\ell+1+(-1)^\ell i0) \\
   & + \sum_{\ell=0}^{n-2} (1-\veps_{\ell+1}) (q_{\ell-1} x_0-p_{\ell-1})(-1)^{\ell} \im\wvarphi_1(x_\ell+(-1)^\ell i0).
\end{split}
\end{equation}

Combining Proposition~\ref{prop:5.17} with \eqref{eq:W finite}, we get
\begin{equation}\begin{split}\label{eq:thm4.11}
& \im\cW(z_0)+W_{\text{finite}}(p_n/q_n)-(-1)^n(p_{n-1}-q_{n-1}\re(z_0))\im\wvarphi_1(z_n+1) - r_n(z_0) \\
& =  \sum_{\ell=0}^{n-1} \varepsilon_{\ell} {(-1)^\ell} \Big[ (p_{\ell-1}-q_{\ell-1}\re(z_0)) \im\wvarphi_1(z_\ell+1)
- (p_{\ell-1}-q_{\ell-1}x_0)\im\wvarphi_1(x_\ell+1+(-1)^\ell i0) \Big] \\
& +\sum_{\ell=0}^{n-2}(1-\varepsilon_{\ell+1}) {(-1)^\ell} \Big[ (p_{\ell-1}-q_{\ell-1}\re(z_0)) \im\wvarphi_1(z_\ell)
- (p_{\ell-1}-q_{\ell-1} x_0) \im\wvarphi_1(x_\ell+(-1)^\ell i0) \Big] \\
& - (1-\veps_n)(-1)^n\Big[ (p_{n-1}-q_{n-1}\re(z_0))\im\wvarphi_1(z_{n}+1)
 + (p_{n-2}-q_{n-2}\re(z_0))\im(\wvarphi_1(z_{n-1})) \Big].
\end{split}\end{equation}

Thus, we have to deal with the following expressions.
For $0\le \ell \le n-1$ such that $\veps_\ell=1$, 
$$A_\ell := (p_{\ell-1}-q_{\ell-1}\re(z_0)) \im\wvarphi_1(z_\ell+1) - (p_{\ell-1}-q_{\ell-1}x_0)\im\wvarphi_1(x_\ell+1+(-1)^\ell i0).$$
For $0\le \ell \le n-2$ such that $\veps_{\ell+1}=0$,
$$B_\ell := (p_{\ell-1}-q_{\ell-1}\re(z_0)) \im\wvarphi_1(z_\ell) - (p_{\ell-1}-q_{\ell-1} x_0) \im\wvarphi_1(x_\ell+(-1)^\ell i0).$$
When $\veps_n=0$,
$$C_n := (p_{n-1}-q_{n-1}\re(z_0))\im\wvarphi_1(z_{n}+1) + (p_{n-2}-q_{n-2}\re(z_0))\im\wvarphi_1(z_{n-1}).$$

The first $n$ digits of the continued fraction of $\re(z_0)$ are also $m_1,\cdots,m_n$.
This gives 
$$|\re(z_0)-x_0|\le |z_0 - p_n/q_n| = q_n^{-1}|q_{n-1}z_0-p_{n-1}||z_n| \le q_n^{-2}|z_n|.$$

Since \eqref{eq:wphi_1 Li} implies that
$$\wvarphi_1'(z) = \frac{1}{\pi}\left( - \Li{\frac{z}{1-z}} + \frac{1}{1-z}\log\left(\frac{1-2z}{1-z}\right) + \frac{1}{1-z}\log\left(\frac{z-2}{z-1}\right)\right) - \frac{\pi}{12},$$
we have 
\begin{equation}\label{eq:wphi' near1}
|\wvarphi_1'(z)|\le C|z-1|^{-1}\log\frac{1}{|z-1|} \quad \text{near }1.
\end{equation}
For $0\le \ell \le n$, the distance of $x_\ell$ and $z_\ell$ from $0$ are comparable.
Thus, for $0\le \ell \le n$, 
$$|\im\wvarphi_1(z_\ell+1)-\im\wvarphi_1(x_\ell + 1 + (-1)^\ell i 0)|\le Cx_\ell^{- 1}\log{\frac{1}{x_\ell}}|z_\ell - x_\ell|,$$
and, for $0\le \ell \le n-1$, 
$$|\im\wvarphi_1(z_\ell)-\im\wvarphi_1(x_\ell + (-1)^\ell i 0)|\le C|x_\ell-1|^{-1}\log\frac{1}{1-x_\ell}|z_\ell - x_\ell|.$$

We have $|z_\ell - x_\ell|\lesssim |z_n|q_\ell^2q_n^{-2}$, $x_\ell^{-1}\lesssim m_{\ell+1}$ and $|x_{\ell}-1|^{-1}\lesssim x_{\ell+1}^{-1}\lesssim m_{\ell+2}$.
Thus, according to \eqref{eq:Im wphi x_l+1} and \eqref{eq:Im wphi x_l}, for $0\le \ell<n$, we have
\begin{align*}
|A_\ell|  & = \Big|(p_{\ell-1}-q_{\ell-1}\re(z_0)) \Big(\im\wvarphi_1(z_\ell+1) - \im\wvarphi_1(x_\ell+1+(-1)^\ell i 0)\Big) \\
& \qquad +   q_{\ell-1}(\re(z_0)- x_0)\im\wvarphi_1(x_\ell+1+(-1)^\ell i0) \Big| \\
& \lesssim x_\ell^{-1}\log{\frac{1}{x_\ell}}|z_\ell-x_\ell|q_\ell^{-1} + q_{\ell-1}q_n^{-2}|z_n|\log\frac1{x_\ell} \\
& \lesssim |z_n|q_n^{-2}(q_\ell  m_{\ell+1} + q_{\ell-1}) \log{(m_{\ell+1})} \lesssim |z_n|q_n^{-2}q_{\ell+1}\log{(m_{\ell+1})},
\end{align*}
and, for $0\le\ell<n-1$ such that $\veps_{\ell+1}=0$, 
\begin{align*}
|B_\ell|  
& = \Big|(p_{\ell-1}-q_{\ell-1}\re(z_0)) \Big(\im\wvarphi_1(z_\ell) - \im\wvarphi_1(x_\ell+(-1)^\ell i 0)\Big) \\
&\qquad +   q_{\ell-1}(\re(z_0)- x_0)\im\wvarphi_1(x_\ell+(-1)^\ell i0) \Big| \\
& \lesssim | p_{\ell-1}-q_{\ell-1}\re(z_0) |  |x_\ell-1|^{-1} \log \frac{1}{|x_\ell-1|} |z_\ell - x_\ell| + q_{\ell-1}q_n^{-2}|z_n| \log\frac{1}{x_{\ell+1}} \\
& \lesssim |x_\ell-1|^{-1}\log\frac{1}{|x_\ell-1|}|z_\ell-x_\ell|q_\ell^{-1} + q_{\ell-1}q_n^{-2}|z_n|\log\frac1{x_{\ell+1}} \\
& \lesssim |z_n|q_n^{-2}(q_\ell  m_{\ell+2}  + q_{\ell-1}) \log(m_{\ell+2})  \lesssim |z_n|q_n^{-2}q_{\ell+2} \log(m_{\ell+2}).
\end{align*}
From $q_n\ge  2^{\lfloor (n-\ell)/2 \rfloor}q_{\ell}$, it follows that
$$\sum_{\ell=0}^{n-1}|A_\ell|+\sum_{\ell=0}^{n-2}|B_\ell| \lesssim |z_n|\,\frac{\log{q_n}}{q_n}.$$
Furthermore, when $\veps_n=0$, we have
$$\wvarphi_1(z_{n-1}) =  z_{n-1} \left(\varphi_0\left(z_n\right)-\varphi_0(-1)\right)-\varphi_0'(-1){+}\varphi_0(z_{n-1}-1).$$
$$\wvarphi_1(1+z_{n}) =  z_{n-1}^{-1} \left(\varphi_0\left(z_{n-1}-1\right)-\varphi_0(-1)\right)-\varphi_0'(-1){+}\varphi_0(z_{n}).$$
Since $$\wvarphi_1(z_{n-1}) =  (1-z_{n-1})\left(\frac\pi{12}-\frac{1}{\pi}\log2\right) + z_{n-1}\wvarphi_1(1+z_n),$$
we have
\begin{align*}
 C_n 
 & =  \im\wvarphi_1(z_{n}+1)  [(p_{n-1}-q_{n-1}\re(z_0)) + (p_{n-2}-q_{n-2}\re(z_0))\re(z_{n-1})] \\
 & + (p_{n-2}-q_{n-2}\re(z_0))\im(z_{n-1})\left[-\left(\frac\pi{12}-\frac{1}{\pi}\log2\right) +\re\wvarphi_1(1+z_n))\right]. \\
 \end{align*}

By \eqref{eq:Lem6.4}, we have
$\left|\wvarphi_1(1+z_n))\right|\lesssim \log^2(1+|z_n|^{-1}).$
As in the proof of \cite[Theorem 5.10]{MMY3},
$$|(p_{n-2}-q_{n-2}\re(z_0))\im(z_{n-1})|\lesssim q_{n}^{-1}|\im(z_n)|,$$
and
$$[(p_{n-1}-q_{n-1}\re(z_0)) + (p_{n-2}-q_{n-2}\re(z_0))\re(z_{n-1})] \lesssim q_{n}^{-1}|\im(z_n)|.$$
Thus, we have $|C_{n}|\lesssim q_{n}^{-1}|\im(z_n)|\log^2(1+|z_n|^{-1})$. 
\end{proof}

\end{document}